\documentclass[10pt,a4paper]{amsart}
\usepackage[utf8]{inputenc}
\usepackage{amsmath}
\usepackage{amsfonts}
\usepackage{amssymb}
\usepackage[left=2cm,right=2cm,top=2cm,bottom=2cm]{geometry}
\usepackage{bbm}
\usepackage{mathrsfs}
\usepackage[all]{xy}
\usepackage{bm}
\usepackage{manfnt}
\usepackage[usenames,dvipsnames]{xcolor}
\usepackage{leftindex}
\usepackage[backref=page]{hyperref}
\author{Geoffrey Powell}
\title{Relating Brauer categories, Koszul complexes, and graph complexes}
\address{Univ Angers, CNRS, LAREMA, SFR MATHSTIC, F-49000 Angers, France}
\email{Geoffrey.Powell@math.cnrs.fr}
\urladdr{https://math.univ-angers.fr/~powell/}

\subjclass[2000]{}
\newtheorem{THM}{Theorem}

\newtheorem{COR}[THM]{Corollary}
\newtheorem{thm}{Theorem}[section]
\newtheorem{prop}[thm]{Proposition}
\newtheorem{cor}[thm]{Corollary}
\newtheorem{lem}[thm]{Lemma}
\theoremstyle{definition}
\newtheorem{defn}[thm]{Definition}
\newtheorem{exam}[thm]{Example}

\theoremstyle{remark}
\newtheorem{rem}[thm]{Remark}
\newtheorem{nota}[thm]{Notation}
\newtheorem{hyp}[thm]{Hypothesis}

\renewcommand{\epsilon}{\varepsilon}
\newcommand{\f}{\mathcal{F}}
\renewcommand{\phi}{\varphi}
\renewcommand{\hom}{\mathrm{Hom}}
\newcommand{\sym}{\mathfrak{S}}
\newcommand{\kmod}{\mathtt{Mod}_\kk}
\newcommand{\calc}{\mathcal{C}}
\newcommand{\nat}{\mathbb{N}}

\newcommand{\ext}{\mathrm{Ext}}
\newcommand{\op}{^\mathrm{op}}
\newcommand{\ob}{\mathrm{Ob}\hspace{2pt}}
\newcommand{\fb}{\mathsf{FB}}
\newcommand{\finj}{{\mathsf{FI}}}
\newcommand{\id}{\mathrm{Id}}
\newcommand{\triv}{\mathsf{triv}}
\newcommand{\sgn}{\mathsf{sgn}}
\newcommand{\aut}{\mathrm{Aut}}
\newcommand{\n}{\mathbf{n}}
\newcommand{\m}{\mathbf{m}}
\newcommand{\kk}{\mathbbm{k}}
\newcommand{\dash}{\text{-}}
\newcommand{\modules}{\mathsf{mod}}
\newcommand{\opd}{\mathscr{O}}
\newcommand{\uwb}{\mathsf{uwb}}
\newcommand{\ub}{\mathsf{ub}}
\newcommand{\dwb}{\mathsf{dwb}}
\newcommand{\db}{\mathsf{db}}
\newcommand{\ddb}{\overrightarrow{\db}}
\newcommand{\dub}{\overrightarrow{\ub}}
\newcommand{\dwbord}{\dwb^{\mathrm{ord}}}
\newcommand{\tfb}{\fb \times \fb}
\newcommand{\uwbord}{\uwb^{\mathrm{ord}}}
\newcommand{\fdvs}{\mathscr{V}}
\newcommand{\cala}{\mathscr{A}}
\newcommand{\calb}{\mathscr{B}}
\newcommand{\stfb}{S_\circledcirc}
\newcommand{\ltfb}{\Lambda_\circledcirc}
\newcommand{\tor}{\mathrm{Tor}}
\newcommand{\der}{\mathrm{Der}}

\newcommand{\unwall}{\mathsf{d}\junwall}
\newcommand{\junwall}{\mathcal{J}}
\newcommand{\sfb}{S_\odot}
\newcommand{\lfb}{\Lambda_\odot}
\newcommand{\cpd}{\mathscr{C}}
\newcommand{\copds}{\mathsf{CyclOpd}}
\newcommand{\opds}{\mathsf{Opd}}
\newcommand{\diopds}{\mathsf{DiOpd}}
\newcommand{\nuco}{\copds^{\mathrm{nu}}}
\newcommand{\nudo}{\diopds^{\mathrm{nu}}}
\newcommand{\nuo}{\opds^{\mathrm{nu}}}
\newcommand{\lad}{\mathscr{L}}
\newcommand{\diag}{\mathrm{diag}}
\newcommand{\tmalg}{\widetilde{\amalg}}

\newcommand{\dubord}{\dub^\mathsf{ord}}
\newcommand{\ddbord}{\ddb^\mathsf{ord}}
\newcommand{\ouw}{\mathrm{UnW}}
\newcommand{\sgntr}{\mathsf{Tr}}
\newcommand{\dsgntr}{\mathsf{dTr}}
\newcommand{\one}{\mathbf{1}}
\newcommand{\two}{\mathbf{2}}
\newcommand{\diopd}{\mathscr{D}}
\newcommand{\X}{\mathsf{d}\Upsilon}
\newcommand{\od}{\mathbf{d}}
\newcommand{\sct}{\mathsf{Sct}}
\newcommand{\dupsilon}{\mathsf{d}\upsilon}
\newcommand{\zero}{\mathbf{0}}
\newcommand{\ump}{\mathrm{mp}}
\newcommand{\dmp}{\mathsf{d}\mathrm{mp}}
\newcommand{\cten}{\otimes}
\numberwithin{equation}{section}
\begin{document}

\begin{abstract}
The purpose of this paper is to investigate the relationship between hairy graph complexes associated to cyclic operads and their counterparts for operads (and, more generally, dioperads).
 This is based on the author's interpretation of these as Koszul complexes for the associated modules over the respective appropriate twisted downward (walled) Brauer category. 

The general question of relating such Koszul complexes is addressed by analysing the relationships between the respective twisted Brauer-type categories, proceeding through a direct analysis. The passage from the walled to unwalled context involves functors induced by the disjoint union of finite sets. 

As an application, for the cyclic operad associated to an operad, this leads to an explicit relation between the respective (hairy) graph homologies. 
\end{abstract}

\maketitle

\section{Introduction}
\label{sect:intro}

Given a cyclic operad $\cpd$, one can construct odd and even flavours of (hairy) graph complexes. (For the purposes of this introduction, we shall concentrate on the even case.)
This goes back to Kontsevich \cite{MR1247289,MR1341841}, as generalized by Conant and Vogtmann \cite{MR2026331}, and \cite{MR3347586,MR3029423} for example for the hairy case. A more sophisticated approach is to consider the cyclic operad as a modular operad and then apply the appropriate Feynman transform \cite{MR1601666} to $\cpd$.  Throughout we work over a field $\kk$ of characteristic zero.

For $\opd$ an operad, Dotsenko has considered forms of hairy graph complexes \cite{MR4945404}. These he relates to the wheeled bar construction on $\opd$ (considered as a wheeled operad with trivial wheeled operad).

The author has developed a different perspective on both of these constructions \cite{P_cyclic,P_opd_wall}. (Here we also extend the operadic theory to dioperads, which provides the natural level of generality.) This approach is based on the fact that graphs can be described by using appropriate flavours of downward Brauer categories.  The starting point in \cite{P_cyclic} for  the cyclic operad case (without requiring a unit) is that there is a natural $(\kk \db)_{(-;-)}$-module $\lfb^* \cpd$ associated to the cyclic operad $\cpd$ (see  Theorem \ref{thm:cpd_db_kdbm} here). Here $\lfb^*$ is the exterior power functor constructed with respect to Day convolution $\odot$ on $\kk \fb$-modules, where $\fb$ is the category of finite sets and bijections; $\db$ is the downward Brauer category, and $(\kk \db)_{(-;-)}$ is the appropriate twisted $\kk$-linearization of this (see Section \ref{subsect:twist_db} for this).  

For $\diopd$ a dioperad (without requiring a unit), there is a similar story: there is a natural $(\kk \dwb)_-$-module structure on $\ltfb^* \diopd$ (see Theorem \ref{thm:diopd_dwb_kkdwb-_modules}). Here, $\ltfb^*$ is the exterior power functor with respect to  Day convolution $\circledcirc$ on $\kk (\tfb)$-modules, $\dwb$ is the downward walled Brauer category, and $(\kk \dwb)_-$ is the appropriate twisted $\kk$-linearization of this (see Section \ref{subsect:twist_uwb} for this).

Now, both $(\kk \db)_{(-;-)}$ and $(\kk \dwb)_-$ are homogeneous quadratic $\kk$-linear categories (over $\kk \fb$ and $\kk (\tfb)$ respectively). Thus, for a $(\kk \db)_{(-;-)}$-module $M$, we can form the following two Koszul complexes (that we call here the {\em first} and {\em second} Koszul complexes, respectively):
\begin{eqnarray*}
&&
(\kk \ub)_{(-;+)} \otimes _{\kk \fb} M
\\
&&
(\kk \ub)_{(-;+)}^\sharp \otimes _{\kk \fb} M.
\end{eqnarray*}
Here $\ub$ is the upward Brauer category, with a different twist for the $\kk$-linearization; the $^\sharp$ indicates duality. These  are complexes of $(\kk \ub)_{(-;+)}$-modules; they are described explicitly in Section \ref{sect:kz_cx}. 

Likewise, for $F$ a $(\kk \dwb)_-$-module, we have the two Koszul complexes:
\begin{eqnarray*}
&&
\kk \uwb \otimes _{\kk (\tfb)} F
\\
&&
(\kk \uwb)^\sharp \otimes _{\kk (\tfb)} F.
\end{eqnarray*}
Here $\uwb$ is the upward Brauer category and $\kk \uwb$ its $\kk$-linearization. These are complexes of $\kk\uwb$-modules.

Specializing to the case of interest, the hairy graph complexes correspond to the respective second Koszul complexes:
\begin{eqnarray*}
&&
(\kk \ub)_{(-;+)}^\sharp \otimes _{\kk \fb} \lfb^* \cpd 
\\
&&
(\kk \uwb)^\sharp \otimes _{\kk (\tfb)} \ltfb^* \diopd .
\end{eqnarray*}
(See Corollary \ref{cor:cpd_Koszul_complexes} for the cyclic case and Corollary \ref{cor:diopd_Koszul_complexes} for the dioperad case, where the first Koszul complexes are also treated, together with the {\em odd} variants of these.) The fact that these correspond to hairy graph complexes is explained in \cite{P_cyclic,P_opd_wall}.

Now, given a cyclic operad $\cpd$, one can consider its underlying operad. More generally, one can form the associated dioperad $\diopd_\cpd$. The values of this are given (for finite sets $X$, $Y$ labelling the inputs and outputs respectively) by 
$ \diopd_\cpd (X,Y) = \cpd (X \amalg Y)$. Otherwise put, the underlying $\kk (\tfb)$-module is given by restricting along $\amalg : \tfb \rightarrow \fb$, the functor induced by the disjoint union of finite sets. Thus we write $\diopd_\cpd$ as $\amalg^* \cpd$. (See Theorem \ref{thm:cpd_to_diopd} for a precise statement identifying the dioperad structure.)

The passage in the other direction is perhaps less familiar, although related constructions occur in \cite{MR1913297}. For a dioperad $\diopd$, one has the associated cyclic operad $\cpd_\diopd$ with values given by 
$$
\cpd_\diopd (X) = \bigoplus_{X = X_1 \amalg X_2} \diopd (X_1, X_2).
$$
This corresponds to forgetting the distinction between `inputs' and `outputs'. 
At the level of the underlying objects, we may identify $\cpd_\diopd$ as $\amalg_* \cpd$, where $\amalg_*$ is both left and right adjoint to $\amalg^*$. (See Theorem \ref{thm:diopd_to_cpd} for the precise statement identifying the cyclic operad structure.)

The next goal is to understand the relationship between $\lfb^* \cpd$ and $\ltfb^* (\amalg^*\cpd)$ (respectively $\ltfb^* \diopd$ and $\lfb^* (\amalg_* \diopd)$). Theorem \ref{THM:compare_modules} below gives the `even parts' of Theorems \ref{thm:cyclic_2_diopd_modules} and \ref{thm:diopd_2_cyclic_modules} respectively. The statement uses the functors
\begin{eqnarray*}
\dupsilon_{(-;-)} ^*
&: &
(\kk \db)_{(-;-)} \dash \modules 
\rightarrow 
(\kk \dwb)_-\dash \modules
\\
\unwall_{(-;-)} &:& (\kk \dwb)_-\dash \modules 
\rightarrow
(\kk \db)_{(-; -)}\dash \modules
\end{eqnarray*}
of Notation \ref{nota:upsilon} and  from Section \ref{subsect:I} respectively.

\begin{THM}
\label{THM:compare_modules}
\ 
\begin{enumerate}
\item 
For $\cpd$ a cyclic operad, there is a natural isomorphism of $(\kk \dwb)_-$-modules:
$$
\ltfb^* (\amalg^* \cpd) \cong \dupsilon_{(-;-)} ^* \lfb^* \cpd. 
$$
\item
For $\diopd$ a dioperad, there is a natural isomorphism of $(\kk \db)_{(-;-)}$-modules:
$$
\lfb^* (\amalg_* \diopd) \cong \unwall_{(-;-)} \ltfb^* \diopd.
$$
\end{enumerate}
\end{THM}

This means that the problem of relating the respective hairy graph complexes can be phrased in more general terms using the second Koszul complexes. Namely, for $M$ a $(\kk \db)_{(-;-)}$-module, what is the relationship between
$$
(\kk \uwb)^\sharp \otimes _{\kk (\tfb)} \dupsilon_{(-;-)} ^*M
\mbox{\quad  and \quad }
(\kk \ub)_{(-;+)}^\sharp \otimes_{\kk \fb} M \ ?
$$
Likewise, for $F$ a $(\kk \dwb)_-$-module, what is the relationship between
$$
(\kk \ub)_{(-;+)}^\sharp \otimes_{\kk \fb} \unwall_{(-;-)} F
\mbox{\quad  and \quad } 
(\kk \uwb)^\sharp \otimes _{\kk (\tfb)} F \ ?
$$
The first of these questions is addressed in Section \ref{subsect:2nd_Kz_dupsilon^*} and the second in Section \ref{subsect:second_Koszul_unwall}. This is based upon the results of Part \ref{part:brauer} together with the general results of Section \ref{sect:new_compare}.

These results then feed into the comparison results of Section \ref{sect:graph_cx_compare}, which contains the main results of the paper. For example, in the cyclic operad case, the second Koszul complex part of Theorem \ref{thm:compare_cyclic_case_grph_cx} is the following (in which $\upsilon^*_{(-;+)}$ is restriction along the functor $\upsilon_{(-;+)}$ of Notation \ref{nota:upsilon}):

\begin{THM}
For $\cpd$ a cyclic operad, there is an inclusion of complexes of $\kk \uwb$-modules:
$$
\upsilon^*_{(-;+)} \big( (\kk \ub)^\sharp_{(-;+)} \otimes_{\kk \fb} \lfb^* \cpd) \big)
\hookrightarrow 
(\kk \uwb)^\sharp \otimes_{\kk (\tfb)} \ltfb^* (\amalg^* \cpd).
$$
\end{THM}

Unfortunately, this does not give us an explicit calculation of the hairy graph homology of the dioperad $\amalg^* \cpd$ in terms of the hairy graph homology of $\cpd$. Moreover, we may be more interested in the case of the associated {\em operad} rather than the {\em dioperad}; the operad can be considered as a sub dioperad $\opd _\cpd \hookrightarrow \amalg^* \cpd$. The above does not yield a morphism between the respective hairy graph complexes.

The situation is better when passing in the other direction,  i.e., starting with a dioperad. Then the second Koszul complex statement of Theorem \ref{thm:compare_diopd_case_grph_cx}
gives the following (in which $\amalg_* \op (\kk \ub)_{(-; +)}^{\cten  2}$ is the bimodule obtained from $(\kk \ub)_{(-; +)}^{\cten  2}$ of Notation \ref{nota:cten_kkub_pmmp}):

\begin{THM}
\label{THM:diopd_2_cpd}
For a dioperad $\diopd$, there is an  isomorphism of complexes of $\kk \fb$-modules:
\begin{eqnarray*}
(\kk \ub)_{(-;+)}^\sharp \otimes_{\kk \fb} \lfb^* (\amalg_* \diopd) 
&\cong &
\big( \amalg_* \op (\kk \ub)_{(-; +)}^{\cten  2}\big)^\sharp
\otimes_{\kk (\tfb)}
\big( (\kk\uwb)^\sharp  \otimes_{\kk(\tfb)} \ltfb^* \diopd \big). 
\end{eqnarray*}
This induces an isomorphism in homology:
 \begin{eqnarray*}
H_* ((\kk \ub)_{(-;+)}^\sharp \otimes_{\kk \fb} \lfb^* (\amalg_* \diopd) )
&\cong &
\big( \amalg_* \op (\kk \ub)_{(-; +)}^{\cten  2}\big)^\sharp
\otimes_{\kk (\tfb)}
H_*  \big( (\kk\uwb)^\sharp  \otimes_{\kk(\tfb)} \ltfb^* \diopd \big). 
\end{eqnarray*}
\end{THM}

This achieves what we set out to do: it expresses the hairy graph homology of the cyclic operad $\amalg_* \diopd$ in terms of that of the hairy graph homology of the dioperad $\diopd$.

The only downside to this result is that it only gives the  restricted $\kk \fb$-module structure, not the full $(\kk \ub)_{(-;+)}$-module structure. The origin of this restriction is Theorem \ref{thm:upsilon^*kkub} (see the discussion in Remark \ref{rem:more_naturality_upsilon^*kkub} and the refinement in Corollary \ref{cor:surject_uspilon^*kkub_unwall_kkuwb}). In the general theory of \cite{P_cyclic}, the full module structure has an important role in understanding the relationship between the first and second Koszul complexes, for example.

As an application of Theorem \ref{THM:diopd_2_cpd}, we obtain Corollary \ref{cor:comparison_opd_unit_case}:

\begin{COR}
\label{COR:relate_graph}
Suppose that $\opd$ is an operad with unit. Then 
there are isomorphisms of $\kk$-vector spaces
 \begin{eqnarray*}
H_* ((\kk \ub)_{(-;+)}^\sharp \otimes_{\kk \fb} \lfb^* (\amalg_* \opd) )(\zero) 
&\cong &
H_*  \big( (\kk\uwb)^\sharp  \otimes_{\kk(\tfb)} \ltfb^* \opd \big)(\zero, \zero) . 
\end{eqnarray*}
All other homology groups vanish.
\end{COR}

This says that, in the case of an operad with a unit, the hairy graph homology {\em with hairs} is zero and there is a simple isomorphism between the respective (non-hairy) graph homologies, 
 constructed in the operadic and the cyclic operadic contexts. 

This Corollary was motivated by the attempt to understand the following. For an operad $\opd$ and $V$ a $\kk$-vector space, one can consider $\der (\opd (V))$, the Lie algebra of derivations of the free $\opd$-algebra on $V$, as in \cite{MR4945404}. The underlying vector space identifies as $\der (\opd (V)) \cong V^\sharp \otimes \opd (V)$. Now, if $V$ is a symplectic vector space, we have  the canonical isomorphism $V^\sharp \cong V$ provided by the form, so that this gives 
$$
\der (\opd (V)) \cong V \otimes \opd (V) \cong \cpd_\opd (V),
$$
where the right hand side denotes the Schur functor associated to the object underlying the cyclic operad associated to $\opd$. 

Now, for any cyclic operad $\cpd$ and symplectic vector space $V$, $\cpd (V)$ has a natural Lie algebra structure (as in \cite{MR2026331}, for example). One can check that the isomorphism $\der (\opd (V)) \cong \cpd_\opd (V)$ is one of  Lie algebras.  This allows the {\em stablization} of the respective Lie algebra homologies to be compared. Using the respective results of Conant-Vogtmann-Kontsevich and of Dotsenko, this relates the corresponding (non-hairy) graph complexes. Corollary \ref{COR:relate_graph} is the explicit manifestation of this, avoiding the detour through Lie algebra homology.

\subsection*{Organization of the paper}

The paper has been separated into Parts corresponding to the major themes. 
\begin{itemize}
\item 
Part \ref{part:brauer} addresses the technical foundations related to twisted Brauer-type categories, modules over them,  and the relationships between these. 
\item 
Part \ref{part:opd_diopd_cpd} is  an intermezzo to introduce the motivational examples arising from cyclic operads and dioperads.
\item 
Part \ref{part:kz_cx} recalls the relevant Koszul complexes and then explains how these are related via the functors introduced in Part \ref{part:brauer}.
\item 
Part \ref{part:graph} puts everything together, so as to compare the hairy graph complexes. 
\item 
Appendix \ref{sect:day_convolution} provides background material on the Day convolution products.
\end{itemize}

\tableofcontents

\part{Brauer-type categories}
\label{part:brauer}

\section{Walled and unwalled upward Brauer categories}
\label{sect:brauer_cat}

The purpose of this section is to review the various flavours of Brauer categories that intervene in this work, including the twisted variants of their $\kk$-linearizations, working over a field $\kk$.  These structures are of well-established interest; for instance, they are used by Sam and Snowden in their work on stability patterns in representation theory \cite{MR3376738}. The presentation here mirrors that given in \cite{P_cyclic} and \cite{P_opd_wall}. 

\begin{hyp}
\label{hyp:car0}
Throughout, we take $\kk$ to be a field of characteristic zero. 
\end{hyp}

%%%%%%%%%%%%%%%%%%%%%%%%%%%%%%%%%%%%%%%%%%%%%%%%%%%%%%%%%%%%
\subsection{Background}

The category of $\kk$-vector spaces is denoted $\kmod$; undecorated tensor products are taken over $\kk$  and vector space duality is denoted by $(-)^\sharp$.

For a category $\calc$ with objects $X, Y$, morphisms in $\calc$ will be denoted either by $\hom_\calc (X, Y)$ or $\calc (X,Y)$.
  If $\calc$ is essentially small, the category of functors from $\calc$ to $\kmod$ by $\f (\calc)$. 
The category $\f (\calc)$ is equipped with the value-wise tensor product $\otimes $:  for functors $F$ and $G$, and $X$ an object of $\calc$, $(F \otimes G) (X) : = F(X) \otimes G(X)$.

   For $\cala$ an essentially small $\kk$-linear category, the category of $\kk$-linear functors from $\cala$ to $\kmod$, which can be viewed as the category of (left) $\cala$-modules, is denoted $\cala \dash\modules$. Thus $\f (\calc)$ is equivalent to $\kk \calc \dash\modules$, the category of $\kk \calc$-modules.

\begin{nota}
\label{nota:cten}
For $\cala$ and $\calb$ two $\kk$-linear categories, $\cala \cten \calb$ denotes the $\kk$-linear category with objects $\ob \cala \times \ob \calb$ and morphisms given by $\cala (-,-) \otimes \calb(-,-)$.
\end{nota}

\begin{nota}
\label{nota:proj}
For $X$ an object of an essentially small category $\calc$, $P^\calc_X$ denotes the standard projective functor to $\kk$-vector spaces, given by 
$
P^\calc_X (Y):= \kk \calc (X,Y)
$.
\end{nota}

\begin{rem}
In the case of an essentially-small $\kk$-linear category $\cala$, one generalizes the above by considering $\cala (X, -)$, which is a projective $\cala$-module.
\end{rem}

\begin{nota}
\label{nota:fb_finj}
For $n \in \nat$, $\n$ denotes $\{1, \ldots , n\}$. The symmetric group on $\n$ is denoted $\sym_n$; the trivial (respectively sign) representations of $\kk \sym_n$ are denoted $\triv_n$ (resp. $\sgn_n$).

Denote by $\fb$ the category of finite sets and bijections and $\finj$ that of finite sets and injections, so that $\fb$ is the maximal subgroupoid of $\finj$.
The category $\finj$ (and hence $\fb$) has skeleton given by the set of objects $\{ \n \mid n \in \nat \}$.
\end{nota}

%%%%%%%%%%%%%%%%%%%%%%%%%%%%%%%%%%%%%%%%%%%%%%%%%%%%%%%%%%%%%
\subsection{Flavours of Brauer categories}

This section provides a brief overview of the categories with which we work. 
We refer to \cite{P_cyclic} for a more detailed presentation of the upward (and downward) Brauer categories and their twisted $\kk$-linearizations, and to \cite{P_opd_wall} 
for a presentation of the upward (and downward) walled Brauer categories and their twisted $\kk$-linearizations. These references treat the ordered versions of these categories; definitions are outlined below in Remark \ref{rem:identify_morphisms_uwb_ub}.

\begin{nota}
\label{nota:Brauer_categories}
Denote by 
\begin{enumerate}
\item
$\uwbord$ the ordered version of the upward Brauer category;
\item 
$\dubord$ the ordered version of the directed upward Brauer category;
\item 
$\uwb$ the upward walled Brauer category; 
\item 
$\dub$ the directed upward Brauer category; 
\item 
$\ub$ the upward Brauer category.
\end{enumerate}
The opposite categories are respectively, $\dwbord$, $\ddbord$, $\dwb$ (the downward walled Brauer category), $\ddb$ (the directed downward Brauer category); $\db$ (the downward Brauer category).
\end{nota}

\begin{rem}
\label{rem:identify_morphisms_uwb_ub}
\ 
\begin{enumerate}
\item 
Objects of $\uwbord$ are pairs of finite sets $(X,Y)$; $\uwbord ((X,Y), (U,V))$ is empty unless $|U|-|X| = |Y|- |V| = n $, for some $n\in \nat$, the degree of the morphism; in the latter case 
$\uwbord ((X,Y), (U,V) ) := (\tfb) ( (X \amalg \n, Y \amalg \n), (U, V))$, so that  the group $\sym_n$ acts on this set (diagonally with respect to the two copies of $\n$).  For the composition of morphisms, see \cite{P_opd_wall}.

For a morphism $f \in \uwbord ((X,Y), (U,V) )= (\tfb) ( (X \amalg \n, Y \amalg \n), (U, V))$ and $i \in \n$, we consider $f(i,i) \in U \times V$ as a {\em walled pair}; the set of such walled pairs is indexed (ordered) by the elements of $\n$ and the $\sym_n$ action changes the order.
\item 
There is a  forgetful functor  $\uwbord \rightarrow \uwb$ 
 that  is the identity on objects. On morphisms, for $(X,Y)$, $(U,V)$ satisfying the above condition, 
$$
\uwb  ((X,Y), (U,V) ):=(\tfb) ( (X \amalg \n, Y \amalg \n) (U, V))/ \sym_n.
$$
This corresponds to forgetting the ordering of the walled pairs. 
\item 
Objects of $\dubord$ are finite sets; $\dubord ( A, B)$ is empty unless $|B|-|A|= 2t $ for some $t\in \nat$, the degree of the morphism; in the latter case $\dubord (A,B) := \fb (A \amalg (\mathbf{2} \times \mathbf{t}), B)$ so that the group $\sym_t$ acts. For the composition of morphisms, see \cite{P_cyclic}.

For $g \in \dubord(A,B) = \fb (A \amalg (\mathbf{2} \times \mathbf{t}), B)$ and $j\in \mathbf{t}$, we consider $(g(1,j), g (2,j))$ as a directed pair of elements in $B$. The directed pairs are indexed (ordered) by the elements of $\mathbf{t}$.
\item 
For $t \in \nat$, the wreath product $\sym_2 \wr \sym_t$ acts on $\mathbf{2} \times \mathbf{t}$. Since $\sym_2 \wr \sym_t$ is  the semidirect product $\sym_2 ^{\times t}\rtimes \sym_t$, there is an inclusion $\sym_t\hookrightarrow \sym_2 \wr \sym_t$ corresponding to the action permuting the copies of $\mathbf{2}$.

There are forgetful functors 
$$
\dubord
\rightarrow 
\dub
\rightarrow 
\ub
$$
that are the identity on objects and correspond to quotient maps on morphisms. In particular, for $A$, $B$ as above,  $\dub (A,B):=  \fb (A \amalg (\mathbf{2} \times \mathbf{t}), B)/ \sym_t$ and $\ub (A,B):= \fb (A \amalg (\mathbf{2} \times \mathbf{t}), B)/ \sym_2 \wr \sym_t$.

The first functor corresponds to forgetting the ordering of the directed pairs; the second then also forgets the `direction' of the pairs (in the notation of the previous point, retaining only the subset $\{ g(1,j), g (2,j)\} \subset B$, the value of $j$ having also been forgotten).
\end{enumerate}
\end{rem}

\begin{rem}
\label{rem:FB_opposite}
In passing to the downward categories, there is a subtlety to bear in mind. For example, consider the case of the upward Brauer category $\ub$, which contains $\fb$ as its maximal subgroupoid. On passing to the opposite categories, the inclusion $\fb\hookrightarrow \ub$ gives $\fb\op \hookrightarrow \db$. Since $\fb$ is a groupoid, passage to the inverse yields the isomorphism of categories $\fb\op \cong \fb$ and we usually identify the maximal subgroupoid of $\db$ as $\fb$ rather than $\fb\op$. 
\end{rem}

%%%%%%%%%%%%%%%%%%%%%%%%%%%%%%%%%%%%%%%%%%%%%%%%%%%%%%%%%%%%%%%%%%%%%%%%%%%%%%%%%%%%%%%
\subsection{Generating morphisms}

For the applications, it is important to understand `generating morphisms' for the categories $\dubord$ and $\uwbord$ and the categories derived from them. This starts from the following observation:

\begin{lem}
\label{lem:maximal_subgroupoid}
The maximal subgroupoid of $\dubord$ identifies as $\fb$ and that of $\uwbord$ as $\tfb$.
\end{lem}

Morphisms of these categories are then generated over the respective maximal subgroupoid by the morphisms involving a single pair. To describe these, we introduce the following notation.

\begin{nota}
\label{nota:alpha_beta}
\ 
\begin{enumerate}
\item 
For $n \in \nat$ and  $(i j)$ an ordered pair, where $i \neq j \in \mathbf{n+2}$, let $\alpha_{(i j)} \in \dubord (\n, \mathbf{n+2})$ denote the morphism represented by the bijection $\n \amalg \mathbf{2} \stackrel{\cong}{\rightarrow} \mathbf{n+2}$ such that the restriction to $\n$ is order-preserving and the restriction to $\mathbf{2}$ is given by $1 \mapsto i$ and $2 \mapsto j$.
\item 
For $m, n \in \nat$ and $i \in \mathbf{m+1}$ and $j \in \mathbf{n+1}$, let $\beta_{i,j} \in \uwbord ((\m, \n) , (\mathbf{m+1}, \mathbf{n+1}))$ denote the morphism represented by the pair of injective maps $\mathbf{m} \stackrel{\cong}{\rightarrow} (\mathbf{m+1}) \backslash \{i\} \subset \mathbf{m+1}$ and $\mathbf{n} \stackrel{\cong}{\rightarrow} (\mathbf{n+1}) \backslash \{j\} \subset \mathbf{n+1}$, where the bijection is order-preserving in each case.
\end{enumerate}
When necessary for precision, these will be denoted by $\alpha^n _{(ij)}$ and $\beta^{m,n}_{i,j}$ respectively.
\end{nota}

We  have the following: 

\begin{prop}
\label{prop:degree_one_gens_fiordev_uwbord}
\ 
\begin{enumerate}
\item 
For $n \in \nat$, $\dubord  (\n, \mathbf{n+2})$ is a free right $\sym_n$-set on $\{ \alpha_{(ij)} \}$, indexed by the set of ordered pairs of elements of $\mathbf{n+2}$. 
\item 
For $m, n \in \nat$, $\uwbord ((\m, \n) , (\mathbf{m+1}, \mathbf{n+1}))$ is a free right $\sym_m \times \sym_n$-set on $\{ \beta_{i,j} \mid i \in \mathbf{m+1}, \ j \in \mathbf{n+1} \}$.
\end{enumerate}
\end{prop}

\begin{proof}
This is a simple consequence of the explicit definitions of the morphism sets in the respective categories.
\end{proof}

This result extends to give the following:

\begin{cor}
\label{cor:decompose_mor_fiordev_uwbord}
Let $t $ be a positive integer.
\begin{enumerate}
\item 
For $n \in \nat$, $\dubord  (\n, \mathbf{n+2t})$ is a free right $\sym_n$-set on $\{ \bigcirc_{k=0}^{t-1} \alpha^{n+2k}_{(i_k j_k)} \}$, the set of all possible composites of the $\alpha$'s.
\item 
For $m, n \in \nat$, $\uwbord ((\m, \n) , (\mathbf{m+t}, \mathbf{n+t}))$ is a free right $\sym_m \times \sym_n$-set on $\{ \bigcirc_{k=0}^{t-1} \beta^{m+k,n+k}_{i_k,j_k} \}$, 
the set of all possible composites of the $\beta$'s.
\end{enumerate}
\end{cor}

\begin{proof}
It is clear that $\dubord  (\n, \mathbf{n+2t})$ is free as a right $\sym_n$-set and, as generators, one can take the morphisms such that the underlying injective map is order-preserving. Using the fact that, by definition of $\dubord$, the pairs defining a morphism are ordered, any such morphism can be expressed {\em uniquely} as an iterated composite of the maps of the form $\alpha_{(ij)}$. 

The case of $\uwbord$ is proved similarly.
\end{proof}

%%%%%%%%%%%%%%%%%%%%%%%%%%%%%%%%%%%%%%%%%%%%%%%%%%%%%%%%%%%%%%%%%%%%%%%%%%%%%%%%%%%%%%%%
\subsection{The twisted $\kk$-linear versions of the Brauer categories}
\label{subsect:twist_db}

In this section we recall the definition of the twisted $\kk$-linear versions of the upward Brauer category (and their opposites). For more details, see \cite{P_cyclic}.

We use the following notation throughout the paper:

\begin{nota}
\label{nota:pm_mp}
Write $(\pm; \mp)$  for an element of the set of ordered pairs of signs  $ \{ (+;+) ,\  (+; -),\  (-;+)  ,\ (-;-) \}$.
\end{nota}

Fix $0<t \in \nat$ and consider the wreath product $\sym_2 \wr \sym_t \subset \sym_{2t}$. The sign and trivial representations of $\sym_2$ (respectively $\sym_t$) induce four (non-isomorphic if $t \geq 2$) representations of $\sym_2 \wr \sym_t$, denoted  by
$
(\kk^{\pm}; \kk^{\mp}) 
$, 
with underlying vector space $\kk$.  Since $\sym_2 \wr \sym_t$ is the semi-direct product $\sym_2^{\times t} \rtimes \sym_t$, this has subgroups $\sym_2^{\times t}$ and $\sym_t$ which generate $\sym_2 \wr \sym_t$. We have the following  characterization of the representations $(\kk^{\pm}; \kk^{\mp}) $:
\begin{enumerate}
\item 
$(\kk^+; \kk^+)$ is the trivial representation of $\sym_2 \wr \sym_t$; 
\item 
$(\kk^+; \kk^-)$: is the restriction of $\sgn_t$ along $\sym_2 \wr \sym_t \twoheadrightarrow \sym_t$;
\item 
$(\kk^-; \kk^+)$: as a $\kk \sym_2^{\times t}$-module, this identifies as $\sgn_2^{\boxtimes t}$; $\sym_t$ acts trivially; 
\item 
$(\kk^-; \kk^-)$:  as a $\kk\sym_2^{\times t}$-module, this identifies as $\sgn_2^{\boxtimes t}$; $\sym_t$ acts via $\sgn_t$. 
\end{enumerate}

\begin{exam}
The representation $(\kk^-; \kk^-)$ is isomorphic to $(\kk^+; \kk^-) \otimes  (\kk^-; \kk^+)$ (with the diagonal action).
\end{exam}

In the following, we assume that  $A$, $B$, $t$ satisfy  $|B|- |A|=2  t$,  as in Remark \ref{rem:identify_morphisms_uwb_ub}, so that $\dubord (A, B)=\fb (A \amalg (\mathbf{2} \times \mathbf{t}), B)$.

\begin{defn}
\label{defn:twisted_kk-linear_ub}
\ 
\begin{enumerate}
\item 
Let $(\kk \dub)_{\pm}$ denote the $\kk$-linear category with objects finite sets and with morphisms given by 
$$
(\kk \dub)_\pm (A, B) := \kk \fb (A \amalg (\mathbf{2} \times \mathbf{t}) , B) \otimes_{\kk \sym_t} \kk^{\pm}, 
$$
where $\kk^+ = \triv_t$ and $\kk^-= \sgn_t$. Composition is induced by that of $\kk \dubord$.
\item 
Let $(\kk \ub)_{(\pm;\mp)}$ denote the $\kk$-linear category with objects finite sets and with morphisms given by
$$
(\kk \ub)_{(\pm;\mp)}:= \kk \fb (A \amalg (\mathbf{2} \times \mathbf{t}) , B) \otimes_{\kk(\sym_2 \wr \sym_t)} (\kk^\pm; \kk^\mp)
$$
and with composition induced by that of $\kk\dubord$.
\end{enumerate}

The twisted variants $(\kk \ddb)_\pm$ and $(\kk \db)_{(\pm; \mp)}$ are defined similarly (or by passage to the opposite category).
\end{defn}

\begin{rem}
\ 
\begin{enumerate}
\item 
We identify $(\kk \dub)_+$ with  $\kk \dub$ and $(\kk \ub)_{(+;+)}$  with $\kk \ub$. 
\item
For given $A$, $B$, as above, the underlying $\kk$-vector spaces of $(\kk \dub)_+ (A, B) = \kk \dub (A,B)$ and $(\kk \dub)_- (A, B)$ are non-canonically isomorphic; similarly the underlying $\kk$-vector spaces of  $(\kk \ub)_{(\pm; \mp)} (A, B)$, for each choice of $(\pm; \mp)$, are pairwise non-canonically isomorphic. 
\end{enumerate}
\end{rem}

By construction, we have:

\begin{prop}
\label{prop:dub_quotients}
There are $\kk$-linear full functors that are the identity on objects 
\[
\xymatrix@R0em{
(\kk \dub)_+ 
&
\kk \dubord 
\ar[l]
\ar[r]
&
(\kk \dub)_-
\\
(\kk \ub)_{(+:\mp)} 
&
(\kk \dub)_\mp
\ar[l]
\ar[r]
&
(\kk \ub)_{(-;\mp)}.
}
\]
\end{prop}

\begin{nota}
\label{nota:psi_phi}
For given $(\pm; \mp)$, let 
\begin{eqnarray*}
\ump_{(\pm; \mp)} &:& (\kk \dub)_\mp \rightarrow (\kk \ub)_{(\pm; \mp)}
\\
\dmp_{(\pm; \mp)}&:& (\kk \ddb)_\mp \rightarrow (\kk \db)_{(\pm;\mp)}
\end{eqnarray*}
denote the $\kk$-linear functors given by Proposition \ref{prop:dub_quotients} and its opposite.
\end{nota}

There is an $\nat$-grading of the morphisms of $\kk \dubord$, given by declaring $\kk \fb (A \amalg (\mathbf{2} \times \mathbf{t}), B)$ to have degree $t$. The subcategory of degree zero morphisms identifies with $\kk \fb$. 
 Moreover, $\kk \dubord$ is a homogeneous quadratic $\kk$-linear category over $\kk \fb$, freely generated by the $\kk \fb$-bimodule of degree one morphisms and with no quadratic relations. (See \cite{P_cyclic} for details and \cite[Chapter 1]{MR4398644} for background on homogeneous quadratic algebras over a base ring, which generalizes to the categorical setting here.)
 
There is an induced $\nat$-grading of the morphisms of $(\kk \dub)_{\pm}$ and of $(\kk \ub)_{(\pm; \mp)}$ and we have (see \cite{P_cyclic}):

\begin{prop}
\label{prop:homog_quadratic_ub}
The categories $(\kk \dub) _{\pm}$ and $(\kk \ub)_{(\pm;\mp)} $ are homogeneous quadratic $\kk$-linear categories over $\kk \fb$.
 Moreover, the $\kk$-linear functors of Proposition \ref{prop:dub_quotients} preserve the $\nat$-grading of the morphisms and are  the identity in degree zero.
\end{prop}

\begin{proof}
(Sketch.) 
The first statement is essentially immediate from the construction of these $\kk$-linear categories, by using that the symmetric group $\sym_t$ is generated by the transpositions of the form $(i \ i+1)$. The second statement is clear by inspection.
\end{proof}

Proposition \ref{prop:dub_quotients} has the following  consequence (cf. \cite{MR3376738}, for example).

\begin{cor}
\label{cor:full_subcat_kkdub_pmmp}
\ 
\begin{enumerate}
\item 
For each $(\pm; \mp)$, $(\kk \ub)_{(\pm; \mp)} \dash\modules$ identifies as a full $\kk$-linear subcategory of $(\kk \dub)_\mp\dash\modules$. 
\item 
For each $\mp$, $(\kk \dub)_\mp \dash\modules$ identifies as a full $\kk$-linear subcategory of $(\kk \dubord)\dash\modules$.
\end{enumerate}
\end{cor}

\begin{proof}
Restriction along $(\kk \dub)_\mp \rightarrow (\kk \ub)_{(\pm; \mp)}$ induces a functor 
$
(\kk \ub)_{(\pm; \mp)} \dash\modules
\rightarrow 
(\kk \dub)_\mp\dash\modules.
 $
This identifies as the embedding of a full $\kk$-linear subcategory. 
 The proof of the second statement is similar.
\end{proof}

\begin{rem}
The essential image of each inclusion given by Corollary \ref{cor:full_subcat_kkdub_pmmp} can be identified explicitly in terms of the action of the degree one morphisms.
\end{rem}

The statement of Proposition \ref{prop:homog_quadratic_ub} can be made more precise by identifying the bimodules of degree one morphisms
using  the following lemma, in which  $\alpha_{(ij )} \in \dubord (\n , \mathbf{n+2}) $ is as in  Notation \ref{nota:alpha_beta}:

\begin{lem}
\label{lem:kkdub_deg1_generator}
For $n \in \nat$, there is an identification of $\kk (\sym_n\op \times \sym_{n+2})$-modules 
$$
(\kk \dub)_+ (\n, \mathbf{n+2}) =  (\kk \dub)_- (\n, \mathbf{n+2}).
$$
Moreover, 
\begin{enumerate}
\item 
as a $\kk \sym_{n+2}$-module, this is  free on the  image of the class $[\alpha_{(n+1 \ n+2)}] \in \kk \dubord (\n , \mathbf{n+2})$;
\item 
as a $\kk \sym_n\op$-module, it is free on the set $\{ [\alpha_{(i j)}] \}$ indexed by ordered pairs of elements $i \neq j \in \n$.
\end{enumerate}

Hence, for each $(\pm; \mp)$, the $\kk (\sym_n\op \times \sym_{n+2})$-module $(\kk \ub)_{(\pm;\mp)}  (\n, \mathbf{n+2}) $ is
\begin{enumerate}
\item 
generated by the image of $[\alpha_{(n+1 \ n+2)}]$ as a $\kk \sym_{n+2}$-module (and hence also as a $\kk (\sym_n\op \times \sym_{n+2})$-module);
\item 
free as a $\kk \sym_n\op$-module on the set $\{ [\alpha_{(i j)}] \mid i < j \in \mathbf{n+2} \}$.
\end{enumerate}
\end{lem}

\begin{proof}
The modules considered lie in degree one with respect to the $\nat$-grading and, in particular, correspond to a single `pair'. Hence the orientation data for the ordering of pairs makes no difference in these two cases. Now, as a $\sym_{n+2}$-set, by Proposition \ref{prop:degree_one_gens_fiordev_uwbord}, it is clear that $\dubord (\n , \mathbf{n+2})$ is generated freely by $\alpha_{(n+1\  n+2)}$. This implies that the image $[\alpha_{(n+1 \  n+2)}]$ generates the respective terms as a $\kk \sym_{n+2}$-module. 

On passage to the category $(\kk \ub)_{(\pm;\mp)}$, the corresponding $\kk \sym_{n+2}$-module is still generated by the image of $[\alpha_{(n+1 \  n+2)}]$, but it is no longer free. More precisely, due to the relation corresponding to $\pm$, we have the equality
$
[\alpha_{(i  j)}]
= 
\pm [\alpha_{(j  i)}]
$ 
for all ordered pairs of elements  $i\neq j \in \mathbf{n+2}$.  Using this in conjunction with Proposition \ref{prop:degree_one_gens_fiordev_uwbord}, one deduces the final statement, by choosing the generators specified by $i<j$.
\end{proof}

%%%%%%%%%%%%%%%%%%%%%%%%%%%%%%%%%%%%%%%%%%%%%%%%%%%%%%%%%%%%%%%%%%%%%%%%%%%
\subsection{The twisted form of  $\kk \uwb$}
\label{subsect:twist_uwb}

The twisting constructions for $\kk \dub$ and  $\kk \ub$ have an analogue for the upward walled Brauer category $\uwb$,  by introducing an orientation sign corresponding to the ordering of the pairs defining a morphism. This is reviewed briefly in this section; see \cite{P_opd_wall} for more details.

\begin{defn}
\label{defn:twist_kkuwb}
For $\mp \in \{ +, -\}$, let $(\kk \uwb)_\mp$ denote the $\kk$-linear category with objects pairs of finite sets and with morphisms given  for $(X,Y)$, $(U,V)$, and $n$ 
as in Remark \ref{rem:identify_morphisms_uwb_ub} by 
$$
(\kk \uwb) _\mp ((X,Y), (U,V)) := 
\kk (\tfb) ((X \amalg \n, Y \amalg \n), (U,V) ) \otimes_{\kk \sym_n} \kk^\mp
$$
for the diagonal action of $\sym_n$,  and with composition induced by that of $\kk \uwbord$. (Here $\kk^+ := \triv_n$ and $\kk^- := \sgn_n$.)

The  downward category $(\kk \dwb)_\mp$ is  the opposite category of $(\kk \uwb)_\mp$. 
\end{defn}

\begin{rem}
The $\kk$-linear category $(\kk \uwb) _+$ identifies with $\kk \uwb$. 
\end{rem}

As for the unwalled case, there is an $\nat$-grading of the morphisms of $(\kk \uwb)_\pm$ corresponding to the number of walled pairs. The degree zero subcategory identifies as $\kk (\tfb)$ and one has the following counterpart of Proposition \ref{prop:homog_quadratic_ub} (see \cite{P_opd_wall}): 

\begin{prop}
\label{prop:twist_kuwb_hq}
The category $(\kk \uwb)_\mp$ is a homogeneous quadratic $\kk$-linear category over $\kk (\tfb)$. 
\end{prop}

The structure can be made more precise by identifying the bimodule of degree one morphisms, using the morphisms $\beta_{i,j}$ of Notation \ref{nota:alpha_beta}:

\begin{lem}
\label{lem:gen_deg_one_kkuwbmp}
For $m, n \in \nat$, the $\kk ((\sym_m \times \sym_n)\op \times (\sym_{m+1} \times \sym_{n+1}))$-module
$$
(\kk \uwb)_- ((\m, \n), (\mathbf{m+1}, \mathbf{n+1}))
= 
\kk \uwb ((\m, \n), (\mathbf{m+1}, \mathbf{n+1}))
= 
\kk \uwbord ((\m, \n), (\mathbf{m+1}, \mathbf{n+1}))
$$
is
\begin{enumerate}
\item 
generated freely as a $\kk(\sym_{m+1} \times \sym_{n+1})$-module by the class $[\beta_{m+1, n+1}]$; 
\item 
generated freely as a $\kk ((\sym_m \times \sym_n)\op)$-module by the set $ \{ [\beta_{i,j}] \mid i \in \mathbf{m+1}, \ j \in \mathbf{n+1} \}$.
\end{enumerate}
\end{lem}

\begin{proof}
For degree one morphisms, the quotient $\kk \uwbord \rightarrow (\kk \uwb)_\mp$ is an isomorphism (independent of $\pm$), so we can work with $\kk \uwbord$. The remaining statement then follows readily by using Proposition \ref{prop:degree_one_gens_fiordev_uwbord}, similarly to Lemma \ref{lem:kkdub_deg1_generator}.
\end{proof}

By construction, for given $\mp$, there is a full $\kk$-linear functor
$
\kk \uwbord \rightarrow (\kk\uwb)_\mp
$
that restricts to the identity functor on the wide subcategory $\kk (\tfb)$ of degree zero morphisms. The following is the counterpart of Corollary \ref{cor:full_subcat_kkdub_pmmp}: 

\begin{cor}
\label{cor: full_subcat_kkuwb_mp}
The category $(\kk\uwb)_\mp \dash\modules$ identifies as a full subcategory of $\kk \uwbord \dash \modules = \f (\uwbord)$.
\end{cor}

\section{Unwalling and walling}
\label{sect:unwall}

This section initiates the study of the passage between the walled and unwalled settings, exploiting the disjoint union of finite sets. This yields the unwalling functor of Proposition \ref{prop:unwall_uwbord} and its twisted counterparts of Corollary \ref{cor:uwb_pm_to_dub_pm}. Then, composing with the functors introduced in Section \ref{sect:brauer_cat}, we have $\kk$-linear functors 
$$
\upsilon _{(\pm; \mp)} :
(\kk \uwb)_{\mp}
\rightarrow 
(\kk \ub) _{(\pm; \mp)} 
$$
and hence the associated restriction functors $\upsilon_{(\pm; \mp)}^*$ between the associated module categories. Theorem \ref{thm:upsilon^*kkub} illustrates these functors by describing $\upsilon_{(\pm; \mp)}^* (\kk \ub)_{(\pm; \mp)}$.

%%%%%%%%%%%%%%%%%%%%%%%%%%%%%%%%%%%%%%%%%%%%%%%%%%%%%%%%%%%%%%%%%%%%%%%%%%%%%%%%%%%%%%%%%
\subsection{Unwalling results}
\label{subsect:unwalling_results}

Recall that we have the subcategories 
$
\tfb \subset \uwbord  $ and $
\fb \subset \dubord
$ 
(see Lemma \ref{lem:maximal_subgroupoid}) and the functor $\tfb \stackrel{\amalg}{\rightarrow } \fb$ given by the disjoint union of finite sets.

\begin{prop}
\label{prop:unwall_uwbord}
The disjoint union of finite sets 
induces a functor 
$
\ouw : 
\uwbord 
\rightarrow 
\dubord
$
and hence a functor $\uwb \rightarrow \dub$ that fits into the commutative (up to natural isomorphism) diagram
$$
\xymatrix{
\uwbord 
\ar[r]^\ouw
\ar[d]
&
\dubord
\ar[d]
\\
\uwb 
\ar[r]
&\dub
\ar[r]
&
\ub,
}
$$
in which the vertical functors are the respective forgetful functors and $\dub \rightarrow \ub$ is the forgetful functor. The composite of the bottom row yields a functor $\uwb \rightarrow \ub$ given on objects by the disjoint union of finite sets. 
\end{prop}

\begin{proof}
The behaviour on objects is clear; let us consider morphisms. By definition, for degree $n$ morphisms, we have 
$$
\uwbord ((U_1, U_2) , (X_1, X_2)) = (\tfb) ((U_1 \amalg \n, U_2 \amalg \n), (X_1, X_2)).
$$
The disjoint union gives the map 
$$(\tfb) ((U_1 \amalg \n, U_2 \amalg \n), (X_1, X_2)) 
\rightarrow 
\fb ((U_1 \amalg \n) \amalg ( U_2 \amalg \n), X_1 \amalg X_2) \cong \fb (U_1 \amalg U_2 \amalg (\mathbf{2} \times \mathbf{n}), X_1 \amalg X_2)),
$$ where the isomorphism is induced by transposing the terms $\n \amalg U_2$ and using the bijection $\n \amalg \n \cong (\mathbf{2} \times \n)$ such that the first copy of $\n$ in the domain corresponds to $\{1 \} \times \n$ and the second to $\{2\} \times \n$. By definition, $\fb (U_1 \amalg U_2 \amalg (\mathbf{2} \times \mathbf{n}), X_1 \amalg X_2))=\dubord (U_1 \amalg U_2, X_1 \amalg X_2)$. One checks directly that this is compatible with the identity morphisms and composition of morphisms, whence the first statement. 

The unordered case follows  by forgetting the ordering.
\end{proof}

\begin{exam}
\label{exam:uwbord_2_fiordev_alpha_beta}
For $m, n \in \nat$, the functor $\ouw$ yields 
$$
\uwbord ((\m, \n ), (\mathbf{m+1}, \mathbf{n+1}) ) \rightarrow \dubord (\mathbf{m+n}, \mathbf{m+n+2}).
$$
 This is a map of right $\sym_m \times \sym_n$-sets, using the restricted action for $\sym_m \times \sym_n \subset \sym_{m+n}$ on the codomain. 
  Using the morphisms introduced in  Notation \ref{nota:alpha_beta}, this is determined by 
  $$
  \beta^{m,n}_{i, j} \mapsto \alpha^{m+n}_{(i j')},
  $$
  where $ i \in \mathbf{m+1}$,  $j \in \mathbf{n+1}$, and $j' := j+m+1$. In particular, the map is injective but not surjective;  the image is the free $\sym_m \times \sym_n$-module generated by the set 
$\{ \alpha^{m+n}_{(i'' j'')} \mid   1 \leq i'' \leq m+1, \ m+2 \leq j'' \leq m+n+2\}$.
\end{exam}

In order to understand left Kan extension along $\ouw$, we consider $\dubord (U_1 \amalg U_2, X)$, recalling that  the set of degree $n$ morphisms have a natural $\sym_n$-action.

\begin{lem}
\label{lem:unwalling_lemma}
Let $U_1$, $U_2$, $X$ be finite sets with $|X| = |U_1| + |U_2 | + 2n $ (for some $n \in \nat$). Then there is a $\sym_n$-equivariant bijection
$$
\fb (U_1 \amalg U_2 \amalg (\mathbf{2} \times \n)  , X) 
\cong 
\coprod_{X = X_1 \amalg X_2} (\fb \times \fb) ((U_1 \amalg \n, U_2 \amalg \n) , (X_1, X_2)).
$$

This corresponds to a bijection 
\begin{eqnarray}
\label{eqn:fiordev_vs_ubord}
\dubord (U_1 \amalg U_2, X) 
\cong 
\coprod_{X = X_1 \amalg X_2} \uwbord ((U_1, U_2) , (X_1, X_2)).
\end{eqnarray}

In particular, for a morphism $f \in \dubord(U_1 \amalg U_2, X) $, there exists a unique pair $(X_1, X_2)$ and a unique map  $\tilde{f} \in \uwbord ((U_1, U_2) , (X_1, X_2))$  such that the image of $\tilde{f}$ under the bijection  is $f$.
\end{lem}

\begin{proof}
Analogously to the proof of Proposition \ref{prop:unwall_uwbord},
there is a $\sym_n$-equivariant isomorphism
\begin{eqnarray}
\label{eqn:shuffle_iso}
\fb (U_1 \amalg U_2 \amalg (\mathbf{2} \times \n)  , X) 
\cong 
\fb ((U_1 \amalg \n) \amalg (U_2 \amalg \n) , X).
\end{eqnarray}

Given a map on the right hand side, corresponding to a bijection  from $(U_1 \amalg \n) \amalg (U_2 \amalg \n)$ to $X$, take $X_i$ to be the image of $(U_i \amalg \n)$, $i \in \{1, 2\}$. Clearly  $X = X_1\amalg  X_2$ and  the map can be considered as belonging to 
$$
 (\tfb) ((U_1 \amalg \n, U_2 \amalg \n) , (X_1, X_2)).
$$ 
This gives rise to the bijection of the statement, which is clearly $\sym_n$-equivariant.
\end{proof}

\begin{prop}
\label{prop:uwbordop_x_dubord-set}
The associations 
\begin{eqnarray*}
((U_1, U_2), X) & \mapsto & \dubord (U_1 \amalg U_2, X) 
\\
((U_1, U_2), X) & \mapsto & \amalg_{X=X_1 \amalg X_2} \uwbord ((U_1, U_2), (X_1,X_2))
\end{eqnarray*}
define functors from $(\uwbord)\op \times \dubord$ to sets.
 The isomorphism 
$$
\dubord (U_1 \amalg U_2, X) 
\cong 
\coprod_{X = X_1 \amalg X_2} \uwbord ((U_1, U_2) , (X_1, X_2))
$$
of Lemma \ref{lem:unwalling_lemma} yields a natural isomorphism between functors from  $(\uwbord)\op \times \dubord$ to sets.
\end{prop}

\begin{proof}
That $((U_1, U_2), X)  \mapsto  \dubord (U_1 \amalg U_2, X) $ defines a functor on $(\uwbord)\op \times \dubord$ is clear: functoriality with respect to $X$ is immediate and the functoriality with respect to $(U_1, U_2)$ is given by restriction along $\ouw$. 

For $((U_1, U_2), X)  \mapsto  \amalg_{X=X_1 \amalg X_2} \uwbord ((U_1, U_2), (X_1,X_2))$, the functoriality with respect to $(U_1,U_2)$ is immediate and that with respect to $X$ corresponds to left Kan extension along $\ouw$. Lemma \ref{lem:unwalling_lemma} allows the left Kan extension to be understood directly, as follows. Consider a degree $t$ morphism $g$ in  $\dubord (X, Y)$ (so that $|Y|-|X| = 2t$). For a  pair $(X_1, X_2)$ of subsets such that $X = X_1 \amalg X_2$,  the lemma provides a unique pair $(Y_1, Y_2)$ such that $Y = Y_1 \amalg Y_2$ and  $\tilde{g} \in \uwbord ((X_1, X_2), (Y_1, Y_2))$ that corresponds to $g$. Postcomposing with $\tilde{g}$ then gives the map 
$$
\uwbord ((U_1, U_2) , (X_1, X_2))
\rightarrow 
\uwbord ((U_1, U_2) , (Y_1, Y_2))
$$
and the disjoint union of these yields the counterpart of the map 
$$
\dubord(U_1 \amalg U_2, X) 
\rightarrow 
\dubord (U_1 \amalg U_2, Y) 
$$
induced by $g$. One checks that this yields a functor from $(\uwbord)\op \times \dubord$ to sets, as required.

 From the explicit construction, it follows that the isomorphism of Lemma \ref{lem:unwalling_lemma} yields a natural isomorphism between these functors, as required.
\end{proof}

%%%%%%%%%%%%%%%%%%%%%%%%%%%%%%%%%%%%%%%%%%%%%%%%%%%%%%%%%%%%%%%%%%%%%%%%%%
\subsection{Twisted unwalling}
 There is  a twisted version of the unwalling functor $\kk \uwb \rightarrow \kk \dub$ of Proposition  \ref{prop:unwall_uwbord}:
 
\begin{cor}
\label{cor:uwb_pm_to_dub_pm}
The functor $\ouw : \uwbord \rightarrow \dubord$ of Proposition \ref{prop:unwall_uwbord} induces $\kk$-linear functors 
$$
\Upsilon_\mp : (\kk \uwb)_{\mp} 
\rightarrow 
(\kk \dub)_{\mp}.
$$
 Hence there are $\kk$-linear  functors  
$
(\kk \uwb)_\mp 
\rightarrow 
(\kk \ub)_{(\pm;\mp)}
$ 
given by composition with the functors of Proposition \ref{prop:dub_quotients}. On objects these are given by $(X, Y) \mapsto X \amalg Y$.
\end{cor}

\begin{nota}
\label{nota:upsilon}
For given $(\pm; \mp)$, denote by 
\begin{eqnarray*}
\upsilon _{(\pm; \mp)} &:& 
(\kk \uwb)_{\mp}
\rightarrow 
(\kk \ub) _{(\pm; \mp)} 
\\
\dupsilon _{(\pm; \mp)} 
&: &
(\kk \dwb)_{\mp}
\rightarrow 
(\kk \db) _{(\pm; \mp)}
\end{eqnarray*}
the composite  $\kk$-linear functor  given by Corollary \ref{cor:uwb_pm_to_dub_pm} and its opposite.
\end{nota}

\begin{rem}
\label{rem:maps_beta_alpha_upsilon}
For given $(\pm; \mp)$, the $\kk$-linear functor  $\upsilon _{(\pm; \mp)} $ is the composite
$$
(\kk \uwb)_{\mp} 
\stackrel{\Upsilon_\mp}{\rightarrow }
(\kk \dub)_{\mp}
\stackrel{\ump_{(\pm; \mp)}}{\rightarrow} 
(\kk \ub)_{(\pm; \mp)}
$$
 (where $\ump_{(\pm; \mp)}$ is as in Notation \ref{nota:psi_phi}).
 For $m, n \in \nat$, we have the corresponding morphisms between bimodules of degree one morphisms:
$$
(\kk \uwb)_{\mp} ((\m,\n), (\mathbf{m+1}, \mathbf{n+1}))
\rightarrow 
(\kk \dub)_{\mp} (\mathbf{m+n}, \mathbf{m+n+2}) 
\rightarrow 
(\kk \ub)_{(\pm; \mp)} (\mathbf{m+n}, \mathbf{m+n+2}). 
$$
These can be described  using the generators exhibited in Lemmas \ref{lem:gen_deg_one_kkuwbmp} and \ref{lem:kkdub_deg1_generator}. 
\begin{enumerate}
\item 
The first map is determined by $[\beta_{i, j}] \mapsto [\alpha_{(i j')}]$, where $j':= m+1 +j$, as in Example \ref{exam:uwbord_2_fiordev_alpha_beta}. 
\item 
The second map is determined by 
$$
[\alpha_{i'' j ''}] \mapsto
\left\{
\begin{array}{ll}
[\alpha_{i'' j ''}] & i'' < j'' \\
\pm [\alpha_{j'' i ''}] & j''< i''.
\end{array}
\right.
$$
\item 
The composite is thus determined  by $[\beta_{i, j}] \mapsto [\alpha_{(i j')}]$ (noting that $i < j'$).
\end{enumerate}
\end{rem}

\begin{exam}
\label{exam:restrict_sympl_vs_module}
We have the $\kk$-linear functor 
$
\dupsilon_{(-;+)} : 
\kk \dwb = (\kk \dwb)_+ \rightarrow (\kk \db)_{(-;+)}. 
$ 
Now, for $V$ a finite-dimensional symplectic $\kk$-vector space, it is well-known that the association $\n \mapsto V^{\otimes n}$ defines a $(\kk \db)_{(-;+)}$-module (see \cite{P_cyclic} for example). Here the `contraction maps' act via the symplectic form $V \otimes V \rightarrow \kk$. Moreover, this $(\kk \db)_{(-;+)}$-module  structure is natural with respect to the symplectic vector space $V$.

We may  restrict this $(\kk \db)_{(-;+)}$-module structure along $\dupsilon_{(-;+)}$ to obtain a $\kk \dwb$-module. By construction, this is given on objects by 
$$
(\mathbf{s}, \mathbf{t}) \mapsto V^{\otimes s} \otimes V^{\otimes t}.
$$
The  symplectic form  `contracts' across the wall; this means that the $\kk \dwb$-module structure does not see the antisymmetry of the symplectic form. 

Indeed, we can mimic this construction starting only with a vector space $W$ that is equipped with a bilinear form $ W \otimes W \rightarrow \kk$ (which need not be non-degenerate).  We then obtain a $\kk \ddb$-module given on objects by $\mathbf{n} \mapsto W^{\otimes n}$. We can restrict along $\Upsilon_+\op : \kk \dwb \rightarrow \kk \ddb$ to obtain a $\kk \dwb$-module. This is given on objects by 
$
(\mathbf{s}, \mathbf{t}) \mapsto W^{\otimes s} \otimes W^{\otimes t}.
$
\end{exam}

%%%%%%%%%%%%%%%%%%%%%%%%%%%%%%%%%%%%%%%%%%%%%%%%%%%%%%%%%%%%%%%%%%%%%%%%%%%%%
\subsection{Restricting along the functors $\upsilon$ and $\dupsilon$}

We have the functor 
$
\upsilon_{(\pm; \mp)} : (\kk \uwb)_\mp \rightarrow (\kk \ub)_{(\pm; \mp)}
$
from Notation \ref{nota:upsilon}. Restriction gives:
$$
\upsilon^*_{(\pm; \mp)} : (\kk \ub)_{(\pm; \mp)} \dash \modules 
\rightarrow 
(\kk \uwb)_\mp \dash \modules.
$$
Restricting to the subcategory of degree zero morphisms, this corresponds to restriction along $\amalg : \tfb \rightarrow \fb$.
Hence, for $F$ a $(\kk \ub)_{(\pm; \mp)}$-module, to identify 
the $(\kk \uwb)_\mp$-module $\upsilon^*_{(\pm; \mp)} F$,  it suffices to specify how the degree one morphisms of $(\kk \uwb)_\mp$ act. These are independent of the twist $\mp$: for any object $(X, Y)$ of $(\kk\uwb)_{\mp}$ and walled pair $s \in X$, $t \in Y$, we require to specify the action corresponding to the pair of inclusions $X \backslash \{s\} \subset X$ and $Y \backslash \{t\} \subset Y$.

Similarly, this applies when considering the restriction along $\dupsilon_{(\pm;\mp)}$ from Notation \ref{nota:upsilon}:
$$
\dupsilon^*_{(\pm; \mp)} : (\kk \db)_{(\pm; \mp)} \dash \modules 
\rightarrow 
(\kk \dwb)_\mp \dash \modules.
$$
Thus, the following result determines the functors $\upsilon_{(\pm; \mp)}^*$ and $\dupsilon_{(\pm; \mp)}^*$.

\begin{prop}
\label{prop:upsilon^*_dupsilon^*}
Let $(X,Y)$ be a pair of finite sets  and $(s,t) \in X \times Y$ be a walled pair.
\begin{enumerate}
\item 
For $F$ a $(\kk \ub)_{(\pm; \mp)}$-module,  the structure morphism 
$$
\upsilon^*_{(\pm; \mp)} F (X \backslash \{s\},Y \backslash \{t\})
\rightarrow 
\upsilon^*_{(\pm;\mp)} F (X,Y) 
$$
identifies with $
F ((X \amalg Y )\backslash \{s,t\})
\rightarrow 
F (X \amalg Y) 
$, 
the structure morphism of the $(\kk \ub)_{(\pm; \mp)}$-module $F$ associated to the  $(X \amalg Y )\backslash \{s,t\} \subset X \amalg Y$ and the ordered pair $(s ,t)$. 
\item
For $G$ a $(\kk \db)_{(\pm; \mp)}$-module, the structure morphism 
$$
\dupsilon^*_{(\pm;\mp)} G (X,Y) 
\rightarrow 
\dupsilon^*_{(\pm; \mp)} G (X \backslash \{s\},Y \backslash \{t\})
$$
identifies with $
G (X \amalg Y) 
\rightarrow 
G ((X \amalg Y )\backslash \{s,t\})
$, 
the structure morphism of the $(\kk \db)_{(\pm; \mp)}$-module $G$ associated to the set $X \amalg Y$ and the ordered pair $(s ,t)$. 
\end{enumerate}
\end{prop}

\begin{proof}
We prove the case of $\upsilon^* _{(\pm; \mp)}$; the case of $\dupsilon^* _{(\pm; \mp)}$ is established similarly.

By the identification of $\upsilon^* _{(\pm; \mp)}$ on the underlying objects, we have:
\begin{eqnarray*}
\upsilon^*_{(\pm;\mp)} F (X,Y)  &=& F (X \amalg Y) 
\\
\upsilon^*_{(\pm; \mp)} F (X \backslash \{s\},Y \backslash \{t\})
&=& 
 F ((X \backslash \{s\}) \amalg (Y \backslash \{t\}))
 \cong 
 F ((X \amalg Y )\backslash \{s,t\}).
\end{eqnarray*}
Moreover, the walled pair $(s,t)$ determines the {\em ordered} pair $(s,t)$ in $X \amalg Y$. 

By definition, $\upsilon_{(\pm; \mp)} = \ump_{(\pm; \mp)} \circ \Upsilon_\mp$, so that $\upsilon_{(\pm; \mp)}^*  =  \Upsilon_\mp^*  \circ\ump_{(\pm; \mp)}^*$. The restriction functor $\ump_{(\pm; \mp)}^*$ is easily understood and the restriction functor $\Upsilon_\mp^*$ is described in  Proposition \ref{prop:Upsilon^*}. Together, these yield the stated identification.
\end{proof}

%%%%%%%%%%%%%%%%%%%%%%%%%%%%%%%%%%%%%%%%%%%%%%%%%%%%%%%%%%%%%%%%%%%%%%%%%%%%%
\subsection{Walling: the covariant case}

As an example of calculating $\upsilon_{(\pm; \mp)}^*$, we consider $(\kk \ub)_{(\pm; \mp)}$, using its canonical   $(\kk \ub)_{(\pm; \mp)}$-bimodule structure. Equivalently, this can  be considered as a  $(\kk \db)_{(\pm; \mp)} \cten  (\kk\ub)_{(\pm; \mp)}$-module. Then, applying $\upsilon_{(\pm; \mp)}^*$ (with respect to the `covariant' structure), we have
$$
\upsilon_{(\pm; \mp)}^* (\kk \ub)_{(\pm; \mp)}
$$ 
which has the structure of a $(\kk \db)_{(\pm; \mp)} \cten  (\kk\uwb)_{ \mp}$-module.

\begin{rem}
We  simplify considerations by restricting the $(\kk \db)_{(\pm; \mp)}$-module structure along $\kk \fb\op \hookrightarrow (\kk \db)_{(\pm; \mp)} $. Thus, we consider $
\upsilon_{(\pm; \mp)}^* (\kk \ub)_{(\pm; \mp)}
$
 as a $\kk \fb\op  \cten  (\kk\uwb)_{ \mp}$-module.
 \end{rem}

On  objects, the functor $\upsilon_{(\pm; \mp)}^* (\kk \ub)_{(\pm; \mp)}$ is given by 
\begin{eqnarray}
\label{eqn:LHS_U_XY}
(U, (X,Y))
\mapsto 
(\kk \ub)_{(\pm; \mp)} (U, X \amalg Y).
\end{eqnarray}
We will relate this to $(\kk \uwb)_{(\pm; \mp)}$, restricting its canonical bimodule structure to a $\kk (\tfb)\op \cten  (\kk \uwb)_\mp$-module structure, using the following notation:

\begin{nota}
\label{nota:cten_kkub_pmmp}
Write $(\kk \ub)_{(\pm; \mp)}^{\cten  2}$ for the $\kk (\tfb)$-bimodule 
$$
((A,B) , (X,Y)) \mapsto (\kk \ub)_{(\pm; \mp)}(A, X) \otimes (\kk \ub)_{(\pm; \mp)} (B,Y).
$$
\end{nota}

The functor $\amalg : \tfb \rightarrow \fb$ induces the functor $\amalg_* : \kk (\tfb) \dash \modules \rightarrow \kk \fb\dash\modules$ (see Section \ref{sect:day_convolution}), and we write $\amalg_*\op$ for the corresponding functor $\kk (\tfb)\op \dash \modules \rightarrow \kk \fb\op\dash\modules$. Thus we have
$$
\amalg_*\op (\kk \ub)_{(\pm; \mp)}^{\cten  2}  \in \kk \fb\op \cten  \kk (\tfb)\dash\modules.
$$
This is given on objects by 
$$
(U, (X,Y) ) \mapsto \bigoplus_{U = U_1 \amalg U_2} (\kk \ub)_{(\pm; \mp)}(U_1, X) \otimes (\kk \ub)_{(\pm; \mp)} (U_2,Y).
$$

\begin{thm}
\label{thm:upsilon^*kkub}
There is an isomorphism of $\kk \fb\op \cten  (\kk \uwb)_\mp$-modules
$$
\upsilon_{(\pm; \mp)}^* (\kk \ub)_{(\pm; \mp)}
\cong 
(\kk \uwb)_{\mp} \otimes_{\kk (\tfb)} \big( \amalg_*\op (\kk \ub)_{(\pm; \mp)}^{\cten  2}\big) .
$$
Here, $\otimes_{\kk (\tfb)}$ is formed with respect to the restricted $\kk (\tfb)\op$-module structure of $(\kk \uwb)_\mp$ and the $\kk (\tfb)$-module structure of $ \amalg_*\op (\kk \ub)_{(\pm; \mp)}^{\cten  2}$, as dictated by the variance.
\end{thm}

\begin{rem}
\label{rem:upsilon^*kkub}
The value on $(U, (X,Y))$ of the expression on the right hand side in the statement of the Theorem can be expanded as 
\begin{eqnarray}
\label{eqn:RHS_U_XY}
\bigoplus_{s,t \in \nat}
(\kk \uwb)_{\mp} ((\mathbf{s}, \mathbf{t}), (X,Y)) \otimes_{\kk (\sym_s \times \sym_t)} 
\Big(
\bigoplus_{U=U_1 \amalg U_2} (\kk \ub)_{(\pm; \mp)}(U_1, \mathbf{s}) \otimes (\kk \ub)_{(\pm; \mp)} (U_2,\mathbf{t})
\Big).
\end{eqnarray}
Here, we may restrict to $s, t \leq \inf \{ |X|,|Y| \}$; in particular,  the indexing sets are finite.
\end{rem}

\begin{proof}[Proof of Theorem \ref{thm:upsilon^*kkub}]
Recall that, for finite sets $(A,X)$, $(\kk \ub)_{(\pm; \mp)}(A, X) $ is zero unless $|X|-|A| \in 2 \nat$, say is equal to $2n$; if this is the case, it is a quotient of $\kk \fb (A \amalg (\mathbf{2} \times \n) , X)$. Similarly, for finite sets $(B,Y)$, $(\kk \uwb)_\mp ((A,B), (X,Y))$ is zero unless $|X|-|A|= |Y|-|B|$ and this is a non-negative integer, $m$ say. In this case, it is a quotient of $\kk (\tfb) ((A \amalg \m, B \amalg \m) , (X,Y))$.

We start from the  following observation. Given a bijection $f \in \fb (U \amalg (\mathbf{2} \times \mathbf{N}), X \amalg Y)$, one has 
\begin{enumerate}
\item 
a canonical decomposition $U= U_1 \amalg U_2$, where $U_1 = f^{-1} (X)$ and $U_2 = f^{-1} (Y)$; 
\item 
a canonical decomposition $\mathbf{N} = \mathbf{N}_1 \amalg \mathbf{N}_2 \amalg \mathbf{N}'$, where $f(\mathbf{2} \times \mathbf{N}_1) \subseteq X$,  $f(\mathbf{2} \times \mathbf{N}_2) \subseteq Y$, and $\mathbf{N}_1$ and $\mathbf{N}_2$ are both maximal with respect to these properties.
\end{enumerate}
The decomposition of $\mathbf{N}$ stems from the observation that a directed pair in $X \amalg Y$ is either entirely contained within $X$ or is entirely contained within $Y$ or consists of one element of $X$ and one element of $Y$. In the latter case, the elements can occur in the pair in either order $(x,y)$ or $(y,x)$, where $x \in X$ and $y \in Y$.

Inspired by the above observation, we define a natural map from (\ref{eqn:RHS_U_XY}) to (\ref{eqn:LHS_U_XY}). We start by defining a lift (i.e., to the ordered, directed setting), fixing $U_1, U_2$ and $s,t$:
\begin{eqnarray*}
\kk (\tfb) ((\mathbf{s} \amalg \m, \mathbf{t}\amalg \m), (X,Y)) \otimes_{\kk (\sym_s \times \sym_t)} 
\big(
\kk \fb (U_1 \amalg (\two \times \n_1) , \mathbf{s}) \otimes (\kk\fb (U_2 \amalg (\two \times \n_2) ,\mathbf{t})
\big)
\\
\rightarrow 
\kk \fb (U_1 \amalg U_2 \amalg (\two \times \n_1) \amalg (\two \times \n_2) \amalg (\two \times \m), X \amalg Y)
.
\end{eqnarray*}
This is defined by using $\amalg $ to give
\begin{eqnarray*}
\kk (\tfb) ((\mathbf{s} \amalg \m, \mathbf{t}\amalg \m), (X,Y))
& \rightarrow &
\kk \fb (\mathbf{s} \amalg \mathbf{t} \amalg (\two \times \m),X \amalg Y)
\\
\kk \fb (U_1 \amalg (\two \times \n_1) , \mathbf{s}) \otimes (\kk\fb (U_2 \amalg (\two \times \n_2) ,\mathbf{t})
&\rightarrow & 
\kk \fb (U_1 \amalg U_2 (\two \times \n_1) \amalg (\two \times \n_2)  , \mathbf{s} \amalg \mathbf{t})
\end{eqnarray*}
(cf.  Lemma \ref{lem:unwalling_lemma}). Then, after applying $-\amalg (\two \times \m)$ which induces
$$
\kk \fb (U_1 \amalg U_2 (\two \times \n_1) \amalg (\two \times \n_2)  , \mathbf{s} \amalg \mathbf{t})
\rightarrow 
\kk \fb (U_1 \amalg U_2 (\two \times \n_1) \amalg (\two \times \n_2) \amalg (\two \times \m)  , \mathbf{s} \amalg \mathbf{t}\amalg (\two \times \m) ), 
$$
we may compose to obtain an element of $\kk \fb (U_1 \amalg U_2 \amalg (\two \times \n_1) \amalg (\two \times \n_2) \amalg (\two \times \m), X \amalg Y)$. This gives a $\kk$-linear map that  factorizes across the tensor product $\otimes_{\kk (\sym_s \times \sym_t)} $, as required.

We can sum these components over $U_1, U_2$ such that $U =U_1 \amalg U_2$ and over all  $s,t$. To complete the construction of the natural map from (\ref{eqn:RHS_U_XY}) to (\ref{eqn:LHS_U_XY}), we must show that this descends when we pass to the quotient taking into account the orientation data for the direction of pairs  and for the ordering of pairs. For the direction of pairs, this is clear: for pairs that lie either entirely in $X$ or entirely in $Y$, this is immediate; for pairs crossing the wall, there is nothing to show, since they come with a canonical direction. 

For the ordering of pairs, this is also straightforward: using the three possibilities corresponding to the decomposition of $\mathbf{N}$ in the discussion above, we may always reorder pairs in $X \amalg Y$ compatibly with this decomposition (following the choice of order of these implicitly made in the construction of the morphism above). Then the construction ensures compatibility of orders of pairs within each class, whence the result.  

The key point is that this $\kk$-linear map is surjective. This is a consequence of the observation above (note that having passed to undirected edges (with orientation signs) is crucial here, to palliate the problem arising from the fact that  walled pairs come with a canonical direction). 

Injectivity can either be established directly or by comparing (finite) dimensions, which is sufficient by the surjectivity. For the latter, one can reduce to the case $(+;+)$ and thereby work set theoretically with $\ub$ and $\uwb$. The result then essentially reduces to the characterization of the different types of pairs that can occur in $X \amalg Y$.
\end{proof}

\begin{rem}
\label{rem:more_naturality_upsilon^*kkub}
The statement of Theorem \ref{thm:upsilon^*kkub} restricts to $\kk \fb\op \cten  (\kk \uwb)_\mp$ since  there is no immediate way to describe the action on the right hand side of a morphism of $(\kk \db)_{(\pm; \mp)}$ that is represented by a directed pair $(u_1 u_2)$ in a component corresponding to the decomposition $U = U_1 \amalg U_2$ when $u_1 \in U_1$ and $u_2\in U_2$. Namely, we require a `distributive law'  to enable us to pass this (walled) pair through the term corresponding to $\amalg_* \op (\kk \ub)_{(\pm; \mp)}^{\cten  2}$ so that it can act as a walled pair on $(\kk \uwb)_\mp$ from the right.

There is a way to finesse this problem, by forgetting the pairs that are not walled. This can be done by using the augmentation $(\kk \ub)_{(\pm; \mp)} \rightarrow \kk \fb$ (applied to both of the factors appearing in $\amalg_* \op (\kk \ub)_{(\pm; \mp)}^{\cten  2}$). This induces a surjection of $\kk \fb\op \cten  (\kk \uwb)_\mp$-modules:
$$
(\kk \uwb)_{\mp} \otimes_{\kk (\tfb)} \big( \amalg_*\op (\kk \ub)_{(\pm; \mp)}^{\cten  2}\big)
\twoheadrightarrow 
\amalg_*\op (\kk \uwb)_{\mp}
$$
and hence a surjection of $\kk \fb\op \cten  (\kk \uwb)_\mp$-modules:
$$
\upsilon_{(\pm; \mp)}^* (\kk \ub)_{(\pm; \mp)}
\twoheadrightarrow 
\amalg_*\op (\kk \uwb)_{\mp}.
$$

Here, we are in better shape: we will see that it is possible to equip the codomain with the structure of a $(\kk \ub)_{(\pm; \mp)} \cten  (\kk \uwb)_\mp$-module, using the functor $\unwall_{(\pm; \mp)} : (\kk \dwb)_\mp\dash\modules  \rightarrow (\kk \db)_{(\pm; \mp)} \dash \modules$ introduced in Definition \ref{defn:I};  this then leads to Corollary \ref{cor:surject_uspilon^*kkub_unwall_kkuwb}.
\end{rem}

It is clear that the $\kk (\tfb)$-bimodule $\amalg_*\op (\kk \ub)_{(\pm; \mp)}^{\cten  2} $ takes finite-dimensional values; this is also true of the underlying bimodule of $(\kk \uwb)_\mp$. Together with the fact that the direct sums appearing in the statement of Theorem \ref{thm:upsilon^*kkub} are over finite indexing sets (see Remark \ref{rem:upsilon^*kkub}), this implies that we may dualize the Theorem to obtain:

\begin{cor}
\label{cor:dupsilon_kkub^sharp}
There is an isomorphism in $(\kk \dwb)_\mp \cten  \kk \fb  $-modules
$$
\dupsilon_{(\pm; \mp)}^*\big(  (\kk \ub)_{(\pm; \mp)}^\sharp \big) 
\cong 
\big( \amalg_*\op (\kk \ub)_{(\pm; \mp)}^{\cten  2}\big)^\sharp
\otimes_{\kk (\tfb)}
(\kk \uwb)_{\mp}^\sharp    .
$$
\end{cor}

\begin{proof}
Working over a field $\kk$ of characteristic zero, invariants for the action of a finite group are naturally isomorphic to coinvariants.  In particular, if $G$  is a finite group, $M$ is a right $\kk G$-module and $N$ is a left $\kk G$-module, we consider the duals $M^\sharp$ and $N^\sharp$ as left and right $\kk G$-modules respectively. Then, we have the natural isomorphism 
$
(M \otimes_{\kk G} N) ^\sharp 
\cong 
N^\sharp \otimes_{\kk G} N^\sharp. 
$  
Using this,  the right hand side in the statement is the  vector space dual of that appearing in Theorem \ref{thm:upsilon^*kkub}, equipped with its natural  $ (\kk \dwb)_\mp \cten  \kk \fb$-module structure.

On the left hand side, we use the fact that the dual $\big ( \upsilon_{(\pm; \mp)}^* (\kk \ub)_{(\pm; \mp)} \big)^\sharp$ is naturally isomorphic to 
$$
\dupsilon_{(\pm;\mp)}^* \big(  (\kk \ub)_{(\pm; \mp)}^\sharp \big) $$
 as a $ (\kk \dwb)_\mp \cten  (\kk \ub)_{(\pm; \mp)} $-module (here $\dupsilon_{(\pm;\mp)}^*$ applies to the contravariant variable of the $(\kk \ub)_{(\pm;\mp)}$-bimodule $(\kk \ub)_{(\pm; \mp)}^\sharp$). This restricts to the $(\kk \dwb)_\mp  \cten  \kk \fb$-module structure in the statement.
 
Combining these isomorphisms gives the result.
\end{proof}

\section{Transfers and sections}
\label{sect:transfers}

The purpose of this section is to introduce the transfer functors $\sgntr_{(\pm; \mp)} 
: 
(\kk \ub)_{(\pm; \mp)} \rightarrow (\kk \dub)_\mp$ and the closely related functors $\sct_{(\pm; \mp)} : (\kk \ub)_{(\pm ; \mp)} \rightarrow (\kk \dub)_{\mp}$. The significance of $\sct_{(\pm; \mp)} $ is that it is a section to the canonical surjection $\ump_{(\pm; \mp)} : (\kk \dub)_\mp \twoheadrightarrow (\kk \ub)_{(\pm; \mp)}$.

%%%%%%%%%%%%%%%%%%%%%%%%%%%%%%%%%%%%%%%%%%%%%%%%%%%%%%%%%%%%%%%%%%%%%%%%%%%%%%%%%
\subsection{The transfer functors}
\label{subsect:transfer_functors}

Proposition \ref{prop:dub_quotients}  gives the $\kk$-linear functors 
$
\ump_{(\pm; \mp)} : 
(\kk \dub)_\mp \rightarrow (\kk \ub)_{(\pm; \mp)}
$
(see Notation \ref{nota:psi_phi}). 
 These are the identity on objects and restrict to the identity on the degree zero morphisms (with respect to the $\nat$-grading).

\begin{prop}
\label{prop:signed_transfer}
For given $(\pm; \mp)$, there is a unique $\kk$-linear functor 
$$
\sgntr_{(\pm; \mp)} 
: 
(\kk \ub)_{(\pm; \mp)} \rightarrow (\kk \dub)_\mp
$$
that is the identity on the $\kk$-linear subcategory $\kk \fb$ of degree zero morphisms and which, for each $n \in \nat$, sends  $[\alpha_{(n+1 \ n+2)}]$ to 
$$
[\alpha_{(n+1 \ n+2)}] + \pm [\alpha_{(n+1 \ n+2)}  \tau], 
$$
where $\tau \in \sym_2$ is the transposition $(1 \ 2)$, considered as acting on $\fb (\mathbf{n+2}, \mathbf{n+2}) \cong \fb (\n \amalg \mathbf{2}, \mathbf{n+2})$ via the action of $\sym_2$ on $\mathbf{2}$.

The composite functor 
$$
(\kk \ub)_{(\pm; \mp)} \stackrel{\sgntr_{(\pm; \mp)}}{\longrightarrow}
 (\kk \dub)_\mp
\longrightarrow 
(\kk \ub)_{(\pm; \mp)} 
$$
is the identity on objects; on degree $t$ morphisms, it acts by multiplication by $2^t$.
\end{prop}

\begin{proof}
By Proposition \ref{prop:homog_quadratic_ub}, the $\kk$-linear category $(\kk \ub)_{(\pm; \mp)}$ is homogeneous quadratic over $\kk \fb$. Thus, to construct the required functor it suffices to define it on degree one morphisms (this must be compatible with the $\kk \fb$-bimodule structure) and check that this is compatible with the quadratic relation. 

It is straightforward to check that the recipe on the image of $[\alpha_{(n+1 \  n+2)}]$ gives a well-defined bimodule morphism on the degree one morphisms. That this is compatible with the quadratic relations is immediate, since using the same $\mp$ implies that these quadratic relations impose the same orientation data condition with respect to the ordering of pairs.

The identification of the composite functor is straightforward.
\end{proof}

%%%%%%%%%%%%%%%%%%%%%%%%%%%%%%%%%%%%%%%%%%%%%%%%%%%%%%%%%%%%%%
\subsection{Weights}
\label{subsect:weights}

This section is an interlude to introduce {\em weights}, that will allow us to introduce the section functor $\sct_{(\pm;\mp)}$ and relate it to $\sgntr_{(\pm; \mp)}$ in Section \ref{subsect:section_functors}.

Let $\cala$ be a small $\kk$-linear category  that has an $\nat$-grading of the morphisms (that is compatible with the composition). 

\begin{lem}
\label{lem:weight}
For  $\lambda \in \kk$, there is a $\kk$-linear functor 
$
\tau_\lambda : \cala \rightarrow \cala
$
 that is the identity on objects and that sends $f$ of degree $n \in \nat$ to $\lambda^n f$. 

Moreover, if $\mu \in \kk$, the $\kk$-linear functors $\tau _\mu \tau_\lambda$ and $\tau_{\mu \lambda}$ are naturally isomorphic. In particular, if $\lambda$ is invertible, then $\tau_\lambda$ is an isomorphism of $\kk$-linear categories with inverse $\tau_{\lambda^{-1}}$.

If $\cala'$ is a second such category, equipped with a $\kk$-linear functor $F :\cala \rightarrow \cala'$ that respects the $\nat$-grading, then the following diagram commutes:
\[
\xymatrix{
\cala 
\ar[r]^{\tau_\lambda}
\ar[d]_F
&
\cala 
\ar[d]^F
\\
\cala'
\ar[r]_{\tau_\lambda}
&
\cala'.
}
\]
 \end{lem}

\begin{proof}
The first statement is an elementary verification, using the fact that the $\nat$-grading of the morphisms of $\cala$ is compatible with the composition.
 The remaining statements are clear.
\end{proof}

Restriction along $\tau_\lambda$ induces the exact functor 
$$
\tau_\lambda^* : 
\cala \dash \modules 
\rightarrow 
\cala \dash \modules.
$$
If $\lambda \in \kk^\times$, then this is an equivalence of categories. 

\begin{rem}
\label{rem:tau_lambda^*}
For $\lambda \in \kk$, the  action of $\tau_\lambda^*$ is given explicitly as follows. 
For $M$ an $\cala$-module and $f$ a morphism of $\cala$ of degree $n$, 
$$
(\tau_\lambda^* M) (f) = \lambda^n M (f).
$$
\end{rem}

In general, for $M$ a $\cala$-module and $\lambda\in \kk^\times$, there is no reason for $M $ and $\tau_\lambda ^*M$ to be isomorphic. However, when the $\nat$-grading is derived from a grading of the objects of $\cala$, one has the following.

\begin{thm}
\label{thm:weight_comparison}
Let $\cala$ be a small $\kk$-linear category equipped with a degree function
$
\od : \ob \cala \rightarrow \nat
$ 
that defines a $\nat$-grading of the morphisms of $\cala$ by setting $f \in \cala (X,Y)$ to have degree $\od (Y) - \od (X)$. (In particular, $\cala (X, Y) = \emptyset$ if $\od (Y) < \od (X)$.)

Then, there is a  natural transformation 
$$
\rho_\lambda  : \id \rightarrow \tau^*_\lambda
$$
of functors given for an $\cala$-module $M$ and an object $X$ by 
$$
M (X) \stackrel{\lambda^{\od (X)}}{\rightarrow}  M(X) = \tau_\lambda^* M (X) .
$$

If $\lambda\in \kk^\times$, then this is a natural isomorphism.
\end{thm}

\begin{proof}
Consider $f \in \cala (X,Y)$. Then, by definition of $\tau_\lambda$, $f$ acts on $\tau_\lambda^* M $ by 
$\lambda^{\od (Y) - \od (X)} M(f)$. By the definition of $\rho_\lambda$, we have the following diagram 
$$
\xymatrix{
M(X) 
\ar[d]_{M (f)} 
\ar[rr]^{\rho_\lambda(X) = \lambda^{\od (X)}} 
&
 &
\tau_\lambda^* M (X) 
\ar[d]^{\lambda^{\od (Y) - \od (X)} M(f)}
\\
M(Y) 
\ar[rr]_{\rho_\lambda(Y) = \lambda^{\od (Y)}}
& 
& 
\tau_\lambda^* M (Y) 
}
$$
and this clearly commutes. This establishes that $\rho_\lambda$ is a natural transformation. 

If $\lambda \in \kk^\times$, then each $\rho _\lambda (X)$ is  an isomorphism, so that $\rho_\lambda$ is a natural isomorphism.
\end{proof}

\begin{exam}
\ 
\begin{enumerate}
\item 
Consider $\cala = (\kk \ub)_{(\pm;\mp)}$ and define $\od$ by $\od (X) := \lfloor \frac{|X|}{2} \rfloor$.  
This determines the usual $\nat$-grading of the morphisms of $(\kk \ub)_{(\pm;\mp)}$.
\item 
The case $\cala = (\kk \dub)_{(\pm; \mp)}$ can be treated similarly, as can $\kk \dubord$. 
\item 
Consider $\cala = (\kk \uwb)_\mp$ and define $\od (X,Y) := |Y|$. This determines the usual $\nat$-grading of the morphisms of $(\kk \uwb)_\mp$.
\end{enumerate}
\end{exam}

%%%%%%%%%%%%%%%%%%%%%%%%%%%%%%%%%%%%%%%%%%%%%%%%%%%%%%%%%%%%%%%%%%%%%%%%%%%%%%%%%%%
\subsection{The section functors}
\label{subsect:section_functors}

Since $\frac{1}{2} \in \kk$, one can show that the functors $\ump_{(\pm; \mp)}$  each admit a section, a weighted variant of $\sgntr_{(\pm; \mp)}$, using the material of Section \ref{subsect:weights}, in particular the $\kk$-linear functors $ \tau_{\frac{1}{2}}$.

\begin{defn}
\label{defn:sct}
Let $\sct_{(\pm; \mp)} : (\kk \ub)_{(\pm ; \mp)} \rightarrow (\kk \dub)_{\mp} $ be the $\kk$-linear functor 
$$
\sct_{(\pm; \mp)} := \tau_{\frac{1}{2}} \circ  \sgntr_{(\pm; \mp)} = \sgntr_{(\pm; \mp)} \circ \tau_{\frac{1}{2}},
$$
where the equality is given by Lemma \ref{lem:weight}.
\end{defn}

\begin{rem}
The functor $\sct_{(\pm; \mp)}$ can also be defined directly (as in Proposition \ref{prop:signed_transfer}) by specifying its behaviour on the degree one morphisms: 
$$
[\alpha_{(n+1 \ n+2)}]
\mapsto 
\frac{1}{2} \Big( [\alpha_{(n+1 \ n+2)}] + \pm [\alpha_{(n+1 \ n+2)}  \tau]\Big). 
$$
\end{rem}

The following can be proved directly or as a consequence of Proposition \ref{prop:signed_transfer}:

\begin{cor}
\label{cor:sct}
The functor $\sct_{(\pm; \mp)} : (\kk \ub)_{(\pm ; \mp)} \rightarrow (\kk \dub)_{\mp} $ is a section of $
\ump_{(\pm; \mp)} : 
(\kk \dub)_\mp \rightarrow (\kk \ub)_{(\pm; \mp)}
$.
\end{cor}

 %%%%%%%%%%%%%%%%%%%%%%%%%%%%%%%%%%%%%%%%%%%%%%%%%%%%%%%%%%%%%%%%%%%%%%%%%%%%%%%%%%%%%
\subsection{Restricting along the transfer functor}
\label{subsect:restrict_along_transfer}

By Proposition \ref{prop:signed_transfer}, for given $(\pm; \mp)$, we have the $\kk$-linear functor 
$
\sgntr_{(\pm; \mp)} 
: 
(\kk \ub)_{(\pm; \mp)} \rightarrow (\kk \dub)_\mp.
 $
This yields the restriction functor 
$$
\sgntr_{(\pm; \mp)} ^* 
: 
(\kk \dub)_\mp \dash \modules 
\rightarrow 
(\kk \ub)_{(\pm; \mp)}\dash \modules.
$$
This is described explicitly as follows.

\begin{prop}
\label{prop:sgntr^*_upward}
For  $M$ a $(\kk \dub)_\mp$-module,
\begin{enumerate}
\item 
for $X$ an object of $(\kk \ub)_{(\pm; \mp)}$, $\big(\sgntr_{(\pm; \mp)} ^* M \big)(X)$ is $M (X)$; 
\item 
for a degree one morphism $\overline{[g]}$ of $(\kk \ub)_{(\pm; \mp)}$ represented by the class of a  bijection $g : X \amalg \mathbf{2} \cong Y$ (corresponding to a morphism of $\dub (X, Y) = \fb (X \amalg \mathbf{2}, Y)$), the morphism $\big(\sgntr_{(\pm; \mp)} ^* M\big) (\overline{[g]})$ is given by 
$$
\ ([g] + \pm [g  \tau])  : M (X) \rightarrow M (Y) 
$$  
Here $\tau$ is the transposition $(1 \ 2)$ that acts on $\dub (X, Y) = \fb (X \amalg \mathbf{2}, Y)$ via the action of $\sym_2$ on $\mathbf{2}$. 
\end{enumerate}
\end{prop}

\begin{proof}
This follows directly from the definition of $\sgntr_{(\pm; \mp)}$. In particular, this ensures that the definition of the action of morphisms is well-defined (i.e., independent of the choice of representative of $\overline{[g]}$).
\end{proof}

This also describes the restriction functor $\sct_{(\pm; \mp)}^*$, by applying the following consequence of  
Theorem \ref{thm:weight_comparison}:

\begin{cor}
\label{cor:sct_versus_sgntr}
There is a natural isomorphism of functors $  (\kk \dub)_{\mp}\dash \modules  \rightarrow (\kk \ub)_{(\pm ; \mp)}\dash \modules $
$$
\sgntr_{(\pm; \mp)}^* \cong \sct_{(\pm; \mp)}^*.
$$
\end{cor}

%%%%%%%%%%%%%%%%%%%%%%%%%%%%%%%%%%%%%%%%%%%%%%%%%%%%%%%%%%%%%%%%%%%%%%%%%%
\subsection{Comparison with the left adjoint to $\ump^*_{(\pm;\mp)}$}

The functor $
\ump_{(\pm; \mp)} : 
(\kk \dub)_\mp \rightarrow (\kk \ub)_{(\pm; \mp)} 
$
of Notation \ref{nota:psi_phi}  induces the restriction functor 
$$
\ump_{(\pm; \mp)}^* : 
(\kk \ub)_{(\pm; \mp)}\dash \modules
\rightarrow 
(\kk \dub)_\mp\dash \modules .
$$
This has a left adjoint $(\ump_{(\pm;\mp)})_\sharp : (\kk \dub)_\mp\dash \modules \rightarrow (\kk \ub)_{(\pm; \mp)}\dash \modules$.

\begin{rem}
The left adjoint $(\ump_{(\pm;\mp)})_\sharp$  can be identified as the induction functor
$$
(\kk \ub)_{(\pm; \mp)} \otimes_{(\kk \dub)_\mp} - : (\kk \dub)_\mp\dash \modules \rightarrow (\kk \ub)_{(\pm; \mp)}\dash \modules.
$$
\end{rem}

For $M$ a $(\kk \dub)_\mp$-module, we have the adjunction unit 
\begin{eqnarray}
\label{eqn:psi_adjunction_unit}
M 
\rightarrow 
\ump_{(\pm; \mp)}^*(\ump_{(\pm;\mp)})_\sharp M
\end{eqnarray}
and this can  be checked to be a natural surjection.

\begin{prop}
\label{prop:sct/sgntr_to_psisharp}
For $M$ a $(\kk \dub)_\mp$-module, there is a natural surjection of $(\kk \ub)_{(\pm;\mp)}$-modules
$$
\sct_{(\pm; \mp)}^* M 
\twoheadrightarrow 
(\ump_{(\pm;\mp)})_\sharp M.
$$

Hence, using the natural isomorphism of Corollary \ref{cor:sct_versus_sgntr}, there is a natural surjection of $(\kk \ub)_{(\pm;\mp)}$-modules
$$
\sgntr_{(\pm; \mp)}^* M 
\twoheadrightarrow 
(\ump_{(\pm;\mp)})_\sharp M.
$$
\end{prop}

\begin{proof}
This follows by applying the exact functor $\sct_{(\pm; \mp)}^* $ to the adjunction unit (\ref{eqn:psi_adjunction_unit}); since $\sct_{(\pm;\mp)}$ is a section of $\ump_{(\pm;\mp)}$ by Corollary \ref{cor:sct}, we have that $\sct_{(\pm; \mp)}^* \ump_{(\pm; \mp)}^*$ is naturally isomorphic to the identity, whence the result.
\end{proof} 

\section{Upwards: restriction functors and their left adjoints}
\label{sect:restrict}

This section uses the functors introduced in Sections \ref{sect:brauer_cat} and \ref{sect:unwall} between various flavours of Brauer categories to introduce basic functors between the associated module categories. These functors are used, for example, in Proposition \ref{prop:duspilon_lad}
 to analyse $\dupsilon_{(\pm; \mp)}^* 
(\kk \ub)_{(\pm; \mp)}$.

%%%%%%%%%%%%%%%%%%%%%%%%%%%%%%%%%%%%%%%%%%%%%%%%%%%%%%%%%%%%%%%%%%%%%%%%%%%%%%%%%%%%%%%%
\subsection{Restricting along $\uwbord \rightarrow \dubord$}

The functor $\ouw : \uwbord \rightarrow \dubord$ of Proposition  \ref{prop:unwall_uwbord} yields the restriction functor 
$
\ouw^* : 
\f (\dubord) 
\rightarrow 
\f (\uwbord)
$
given by precomposition. This is compatible with the restrictions to the subcategories $\tfb \subset \uwbord$ and $\fb \subset \dubord$ via the commutative (up to natural isomorphism) diagram 
$$
\xymatrix{
\f (\dubord) 
\ar[r]^{\ouw^*}
\ar[d]
&
\f (\uwbord)
\ar[d]
\\
\f (\fb) 
\ar[r]_{\amalg^*} 
&
\f (\tfb),
}
$$
in which the vertical maps give the underlying functors and $\amalg^*$ is restriction along $\amalg : \tfb \rightarrow \fb$. Recall  (see Appendix \ref{sect:day_convolution}) that there is a functor $\amalg_* : \f (\tfb) \rightarrow \f (\fb)$ that is both  left and right adjoint to $\amalg^*$. On objects, for $G \in \ob \f (\tfb)$, $\amalg_* G$ is given by 
$$
(\amalg_* G) (X) 
= 
\bigoplus _{X = X_1 \amalg X_2} G(X_1, X_2).
$$

The restriction functor $\ouw^*$ admits a left adjoint $\ouw_\sharp : \f (\uwbord) \rightarrow \f (\dubord)$. Viewing these as the respective categories of modules over the 
 $\kk$-linearizations, the left adjoint identifies as  the induction functor 
$$
\kk \dubord \otimes_{\kk \uwbord} - \ : \ \kk\uwbord \dash \modules \rightarrow \kk \dubord \dash \modules. 
$$

This  is described explicitly as follows, using left Kan extension:

\begin{prop}
\label{prop:ouw_sharp}
The functor $\ouw_\sharp : \f (\uwbord) \rightarrow \f (\dubord)$ is compatible with $\amalg_*$, in that the following diagram commutes up to natural isomorphism
$$
\xymatrix{
 \f (\uwbord) 
\ar[r]^{\ouw_\sharp}
\ar[d]
&
\f (\dubord)
\ar[d]
\\
\f (\tfb) 
\ar[r]_{\amalg_*}
&
\f (\fb).
}
$$
In particular,  for $G \in \ob \f (\uwbord)$ and a finite set $X$, 
$$(\ouw_\sharp G) (X) 
= 
\bigoplus _{X = X_1 \amalg X_2} G(X_1, X_2)
$$
and the functor $\ouw_\sharp$ is exact.

For $g \in \dubord (X, Y)$, the induced map 
$$
(\ouw_\sharp G) (X) 
= 
\bigoplus _{X = X_1 \amalg X_2} G(X_1, X_2)
\rightarrow 
(\ouw_\sharp G) (Y) 
= 
\bigoplus _{Y = Y_1 \amalg Y_2} G(Y_1, Y_2)
$$
 restricted to the component indexed by $(X_1, X_2)$ is given by 
 $$
 G (\tilde{g}) : G(X_1, X_2) \rightarrow G(Z_1, Z_2)  \subset \bigoplus _{Y = Y_1 \amalg Y_2} G(Y_1, Y_2),
 $$
 where $(Z_1, Z_2)$ and $\tilde{g} \in \uwbord ((X_1, X_2), (Z_1, Z_2) ) $ are given by Lemma \ref{lem:unwalling_lemma}.
\end{prop}

\begin{proof}
Using Lemma \ref{lem:unwalling_lemma} one can check directly that the recipe in the statement defines a functor $\f (\uwbord) \rightarrow \f (\dubord)$. Moreover, it is clear that this is exact and that it is compatible with $\amalg_*$. By construction, this corresponds to the left Kan extension along $\ouw$, hence gives the left adjoint as required.
\end{proof}

%%%%%%%%%%%%%%%%%%%%%%%%%%%%%%%%%%%%%%%%%%%%%%%%%%%%%%%%%%%%%%%%%%%%%%%%%%%%%%%%%%%%%%%%%%%
\subsection{The restriction functors from walled to directed, with twists}
The functor 
$
\Upsilon_\mp : 
(\kk \uwb) _{\mp} 
\rightarrow
(\kk \dub) _{\mp}
$ of Corollary \ref{cor:uwb_pm_to_dub_pm} 
 induces the exact restriction functor 
\begin{eqnarray}
\label{eqn:Upsilon^*}
\Upsilon_\mp^* : 
(\kk \dub) _{\mp}\dash\modules 
\rightarrow 
(\kk \uwb) _{\mp} \dash \modules.
\end{eqnarray}

\begin{rem}
By Corollary \ref{cor:full_subcat_kkdub_pmmp}, $(\kk \dub) _{\mp}\dash\modules $ identifies as a full $\kk$-linear subcategory of $\kk \dubord\dash\modules$; likewise, by Corollary \ref{cor: full_subcat_kkuwb_mp}, $(\kk \uwb) _{\mp} \dash \modules$ identifies as a full $\kk$-linear subcategory of $\kk \uwbord \dash \modules$. One can check that the restriction functor $\ouw ^* : \kk \dubord \dash\modules \rightarrow \kk \uwbord\dash \modules$ restricts to $\Upsilon_\mp^*$ between the respective full subcategories. Using this observation, the results of this subsection can be deduced from those for $\ouw$. 
\end{rem}

The restriction functor $\Upsilon_\mp^* $ is determined by the following Proposition. Since $(\kk \uwb)_\mp$ is a homogeneous quadratic $\kk$-linear category (see Proposition \ref{prop:twist_kuwb_hq}),  it suffices to define the values on objects and the action of degree one morphisms. For the degree one morphisms, one  
uses the  identification 
$$
(\kk \uwb)_{\mp} ((X, Y), (U, V))
= 
(\kk \uwbord) ((X, Y), (U,V))
= \kk (\tfb) ((X \amalg \mathbf{1}, Y \amalg \mathbf{1}), (U, V)),
$$
where $|U|= |X |+1 $ and $|V| = |Y|+1$.

\begin{prop}
\label{prop:Upsilon^*}
The  functor $\Upsilon_\mp^*  : 
(\kk \dub) _{ \mp}\dash\modules 
\rightarrow 
(\kk \uwb) _{\mp} \dash \modules
$
is given on a $(\kk \dub) _{\mp}$-module $M$ by:
\begin{enumerate}
\item 
 the value of $\Upsilon_\mp^* M$ on the object $(X,Y)$ is $M (X\amalg Y)$;
\item 
for a  degree one morphism represented by the class $[f]$ of a morphism $f \in (\tfb) ((X \amalg \mathbf{1}, Y \amalg \mathbf{1}), (U, V))$, the induced map 
$M (X\amalg Y) \rightarrow M (U \amalg V)$ is given by $M ([\amalg f])$, where $[\amalg f]$ is the morphism in $(\kk \dub) _{\mp}$ given by the class of the morphism $\amalg f $ in $\fb$, using the direction of the pair determined by the wall.
\end{enumerate}
\end{prop}

The restriction functor $\Upsilon_\mp^*$ has a left adjoint 
$
(\Upsilon_\mp)_\sharp : 
 (\kk \uwb) _{\mp} \dash \modules
 \rightarrow 
 (\kk \dub) _{ \mp}\dash\modules .
$
 This has an explicit description that can be deduced from that of Proposition \ref{prop:ouw_sharp} for $\ouw_\sharp$.
 
 \begin{prop}
 \label{prop:Upsilon_sharp}
 The left adjoint to (\ref{eqn:Upsilon^*}),
 $$
(\Upsilon_\mp)_\sharp : 
 (\kk \uwb) _{\mp} \dash \modules
 \rightarrow 
 (\kk \dub) _{ \mp}\dash\modules ,
 $$
is determined for $N$ a $(\kk \uwb) _{\mp}$-module by the following.
\begin{enumerate}
\item
For  $X$ an object of $(\kk \dub) _{ \mp}$, 
$$
\big((\Upsilon_\mp)_\sharp N\big) (X) 
= 
\bigoplus_{X = X_1 \amalg X_2} 
N (X_1, X_2).
$$
\item 
For $[f]$ a degree one morphism in $(\kk \dub)_\mp (X,Y)$ represented by a bijection $f: X \amalg \mathbf{2} \cong Y$, and for a decomposition $X = X_1 \amalg X_2$,  the restriction of  $\big( (\Upsilon_\mp)_\sharp N\big) ([f])$ to the summand   $N (X_1, X_2)$ of $\big( (\Upsilon_\mp)_\sharp N \big)(X) $ is the composite
$$
N (X_1, X_2)
\stackrel{f_{X_1, X_2}}{\rightarrow} 
N(Y_1 , Y_2) 
\subset 
\big((\Upsilon_\mp)_\sharp N \big)(Y), 
$$
where $Y_1 := f(X_1  \amalg \{1\})$ and $Y_2 := f (X_2 \amalg \{2\})$, with associated map $f_{(X_1,X_2)}$ in $\uwbord ((X_1, X_2) , (Y_1, Y_2))$.
\end{enumerate}
 \end{prop}

%%%%%%%%%%%%%%%%%%%%%%%%%%%%%%%%%%%%%%%%%%%%%%%%%%%%%%%%%%%%%%%%%%%%%%%%%%%%%%%%%%%
\subsection{From $(\kk \uwb)_\mp$-modules to $(\kk \ub)_{(\pm; \mp)}$-modules: the functor $\junwall_{(\pm; \mp)}$}
\label{subsect:J}
Using the functors $(\Upsilon_\mp)_\sharp$ and $\sgntr_{(\pm; \mp)}^*$ (introduced in Section \ref{sect:transfers}), we introduce the following:

\begin{defn}
\label{defn:J}
Let $\junwall_{(\pm; \mp)} : (\kk \uwb)_\mp\dash\modules  \rightarrow (\kk \ub)_{(\pm; \mp)} \dash \modules$ be the 
composite functor 
$$
\xymatrix{
(\kk \uwb)_\mp\dash\modules 
\ar[r]^{(\Upsilon_\mp)_\sharp}
&
(\kk \dub)_\mp\dash \modules 
\ar[r]^{\sgntr_{(\pm; \mp)}^*}
&
(\kk \ub)_{(\pm; \mp)} \dash \modules.
}
$$ 
\end{defn}

Using the descriptions of  $(\Upsilon_\mp)_\sharp$ and $\sgntr_{(\pm; \mp)}^*$, we have:

\begin{prop}
\label{prop:J}
The functor $\junwall_{(\pm; \mp)} : (\kk \uwb)_\mp\dash\modules  \rightarrow (\kk \ub)_{(\pm; \mp)} \dash \modules$ is exact. 

For $M$ a $(\kk \uwb)_\mp$-module, it is determined  by the following:
\begin{enumerate}
\item 
for $X$ an object of $(\kk \ub)_{(\pm; \mp)}$, one has 
$$
\big(\junwall_{(\pm; \mp)} M \big)(X) 
= 
\bigoplus_{X = X_1 \amalg X_2} M (X_1, X_2);
$$
\item 
for a degree one morphism $\overline{[g]}$ of $(\kk \ub)_{(\pm; \mp)}$ represented by the class of a  bijection $g : X \amalg \mathbf{2} \cong Y$, the restriction of $\big(\junwall_{(\pm; \mp)} M \big) (\overline{[g]})$ to the summand indexed by $(X_1, X_2)$ is given by 
$$
M(X_1,X_2) 
\stackrel{\big([\tilde{g}], \pm [\tilde{g'}]\big)}{\longrightarrow } 
M(Y_1, Y_2) 
\oplus 
M(Y_1', Y_2') 
\subset 
\big(\junwall_{(\pm; \mp)} M \big) (Y),
$$
where 
\begin{enumerate}
\item 
$(Y_1, Y_2)$ and $\tilde{g} \in \uwbord ((X_1, X_2) , (Y_1, Y_2))$ is given by Lemma \ref{lem:unwalling_lemma} for $g$; 
\item 
$(Y_1', Y_2')$ and $\tilde{g'} \in \uwbord ((X_1, X_2) , (Y_1', Y_2'))$ is given by Lemma \ref{lem:unwalling_lemma} for $g':= g \circ \tau$, $\tau = (1 \ 2)$ acting on $\mathbf{2}$.
\end{enumerate}
\end{enumerate}
\end{prop}

\begin{proof}
This follows by combining Proposition \ref{prop:Upsilon_sharp} with Proposition \ref{prop:sgntr^*_upward}.
\end{proof}

%%%%%%%%%%%%%%%%%%%%%%%%%%%%%%%%%%%%%%%%%%%%%%%%%%%%%%%%%%%%%%%%%%%%%%%%%%%%%%%%%%%
\subsection{From $(\kk \uwb)_\mp$-modules to $(\kk \ub)_{(\pm; \mp)}$-modules: the left adjoint $\lad_{(\pm; \mp)}$}
\label{subsect:lad}

For given $(\pm; \mp)$,  we now exploit the $\kk$-linear functor 
$ 
\upsilon _{(\pm; \mp)} : 
(\kk \uwb) _{\mp} 
\rightarrow 
(\kk \ub)_{(\pm; \mp)}$ 
of Notation \ref{nota:upsilon}. 
Restriction yields the exact functor
$$
\upsilon _{(\pm; \mp)}^*
 :
 (\kk \ub)_{(\pm; \mp)} \dash \modules 
\rightarrow 
(\kk \uwb) _{\mp} \dash \modules. 
$$
This fits into the commutative (up to natural isomorphism) diagram 
\begin{eqnarray}
\label{eqn:upsilon^*_square}
\xymatrix{
 (\kk \ub)_{(\pm; \mp)} \dash \modules 
\ar[r]^{\upsilon_{(\pm; \mp)}^*}
\ar[d]
&
(\kk \uwb)_{\mp} \dash \modules
\ar[d]
\\
\kk \fb\dash \modules 
\ar[r]_{\amalg^*}
&
\kk(\tfb)\dash \modules ,
}
\end{eqnarray}
in which the vertical morphisms give the underlying object. 

\begin{nota}
\label{nota:lad}
Denote the left adjoint to  $  \upsilon_{(\pm; \mp)}^*$ by
\[
\lad_{(\pm; \mp)} 
: 
(\kk \uwb) _{\mp} \dash \modules
\rightarrow 
 (\kk \ub)_{(\pm; \mp)} \dash \modules. 
\]
\end{nota}

\begin{rem}
Since $\upsilon_{(\pm; \mp)}$ identifies as the composite $\ump_{(\pm; \mp)} \circ \Upsilon_\mp$, the left adjoint $\lad _{(\pm; \mp)}$ is naturally isomorphic to  the composite $(\ump_{(\pm;\mp)})_\sharp \circ (\Upsilon_\mp)_\sharp $, where $(\ump_{(\pm;\mp)})_\sharp$ is left adjoint to the restriction functor $\ump_{(\pm; \mp)}^* : (\kk \ub)_{(\pm; \mp)} \dash \modules \rightarrow (\kk \dub)_\mp \dash\modules$. 
\end{rem}

This should be compared with the definition of the functor $\junwall_{(\pm;\mp)}$ given in Definition \ref{defn:J}, which is the composite $\sgntr_{(\pm;\mp)}^* \circ (\Upsilon_\mp)_\sharp $. Indeed, we have the following consequence of Proposition \ref{prop:sct/sgntr_to_psisharp}: 

\begin{cor}
\label{cor:surject_junwall_to_lad}
There is a natural surjection of functors from $(\kk \uwb) _{\mp} \dash \modules$ to $
 (\kk \ub)_{(\pm; \mp)} \dash \modules$:
 \begin{eqnarray}
 \label{eqn:junwall_to_lad}
 \junwall_{(\pm;\mp)}
 \twoheadrightarrow
  \lad_{(\pm;\mp)} 
 .
 \end{eqnarray}
\end{cor}

\begin{proof}
Proposition \ref{prop:sct/sgntr_to_psisharp} gives the natural surjection $\sgntr_{(\pm;\mp)}^* \twoheadrightarrow (\ump_{(\pm;\mp)})_\sharp$. The result follows by precomposing with the functor $ (\Upsilon_\mp)_\sharp$.
\end{proof}

The functor $\lad_{(\pm; \mp)}$ can be identified with the induction functor 
$$
(\kk \ub)_{(\pm; \mp)} \otimes_{(\kk \uwb)_\mp} - 
 :  
(\kk \uwb) _{\mp} \dash \modules
\rightarrow 
 (\kk \ub)_{(\pm; \mp)} \dash \modules. 
$$
This is not as easy to describe explicitly  as the functor $\junwall_{(\pm; \mp)}$ (cf. Proposition \ref{prop:J} above).
In particular, the natural surjection (\ref{eqn:junwall_to_lad}) is not in general an isomorphism (see Example \ref{exam:J-L} below).

However, if we restrict to objects induced up from $\kk(\tfb)$-modules, then there is a straightforward description.
We use the following induction functors:
\begin{enumerate}
\item 
the left adjoint to the restriction functor $(\kk \uwb) _{\mp} \dash \modules \rightarrow \kk(\tfb)\dash \modules $:
\[
\kk \uwb_\mp \otimes_{\kk (\tfb)} - : 
\kk(\tfb)\dash \modules 
\rightarrow 
(\kk \uwb) _{\mp} \dash \modules
\]
\item 
the left adjoint to the restriction functor $ (\kk \ub)_{(\pm; \mp)} \dash \modules \rightarrow \kk \fb\dash \modules $:
\[
\kk \ub_{(\pm; \mp)} \otimes_{\kk\fb} - : 
\kk \fb\dash \modules 
\rightarrow 
 (\kk \ub)_{(\pm; \mp)} \dash \modules .
\]
\end{enumerate}

We have the following, in which $\amalg_*$ is considered as the left adjoint to $\amalg^*$ (see Proposition \ref{prop:amalg_tmalg_adjoints}):

\begin{lem}
\label{lem:left_adjoint_induced}
There is a commutative (up to natural isomorphism) diagram of left adjoints:
\[
\xymatrix{
\kk(\tfb)\dash \modules 
\ar[r]^{\amalg_*} 
\ar[d]
&
\kk \fb\dash \modules
\ar[d]
\\
(\kk \uwb) _{\mp} \dash \modules
\ar[r]_{\lad_{(\pm; \mp)} }
&
(\kk \ub)_{(\pm; \mp)} \dash \modules, 
}
\]
in which the vertical arrows correspond to the induction functors.  In particular, $\lad_{(\pm; \mp)}$ sends induced objects in $(\kk \uwb) _{\mp} \dash \modules$ to induced objects in $(\kk \ub)_{(\pm; \mp)} \dash \modules$.
\end{lem}

\begin{proof}
The commutative diagram follows from (\ref{eqn:upsilon^*_square}) by passing to left adjoints. 
The second statement follows immediately.
\end{proof}

We may consider $(\kk \uwb)_\mp$ as a $(\kk \uwb)_\mp$-bimodule or, equivalently, a $(\kk \dwb)_\mp \cten  (\kk \uwb)_\mp$-module. We may thus apply $\lad_{(\pm; \mp)}$ with respect to the covariant variable to give 
$$
\lad_{(\pm; \mp)} (\kk \uwb)_\mp 
\in  
(\kk \dwb)_\mp \cten  (\kk \ub)_{(\pm; \mp)} 
\dash \modules.
$$ 
Correspondingly, we may consider $(\kk \ub)_{(\pm; \mp)}$ as a $(\kk \db)_{(\pm; \mp)} \cten  (\kk \ub)_{(\pm; \mp)}$-module and apply $\dupsilon_{(\pm; \mp)}^*$ (with respect to the covariant variable) to give
$$
\dupsilon_{(\pm; \mp)}^* 
(\kk \ub)_{(\pm; \mp)}
\in 
(\kk \dwb)_\mp \cten  (\kk \ub)_{(\pm; \mp)} 
\dash \modules.
$$

We have the following:

\begin{prop}
\label{prop:duspilon_lad}
There is a natural isomorphism 
$$
\dupsilon_{(\pm; \mp)}^* 
(\kk \ub)_{(\pm; \mp)}
\cong 
\lad_{(\pm; \mp)} (\kk \uwb)_\mp 
$$
in 
$
(\kk \dwb)_\mp \cten  (\kk \ub)_{(\pm; \mp)} 
$-modules.
\end{prop}

\begin{proof}
This is a direct consequence of the identification of $\lad_{(\pm;\mp)}$ as an induction functor. For clarity we recall how this works for modules over associative, unital rings. Consider a morphism $f : A \rightarrow B$ of unital associative rings; this induces the restriction functor on left modules $f^* : B\dash\modules \rightarrow A \dash\modules$ and also on right modules $(-)\downarrow : 
B\op \dash\modules\rightarrow A\op \dash\modules$ by restriction. The left adjoint to $f^*$ is the induction functor 
$$
\ell := B\downarrow \otimes_A - 
:
A \dash \modules 
\rightarrow 
B\dash\modules,
$$
writing $B\downarrow$ for $B$ considered as a $B\otimes A\op$-module. Applied to $A$ considered as an $A$-bimodule, this gives the isomorphism of $B\otimes A\op$-modules
$
\ell A \cong B\downarrow.
$
\end{proof}

\begin{rem}
Alternatively, Proposition \ref{prop:duspilon_lad}  can be proved by  building upon Lemma \ref{lem:left_adjoint_induced} by analysing the behaviour of $\lad_{(\pm; \mp)}$ on the standard projectives of $(\kk \uwb)_\mp$-modules given by $(\kk \uwb)_\mp ((X,Y), -)$. This requires checking the naturality with respect to $(X,Y)$.
\end{rem}

\begin{exam}
\label{exam:J-L}
Take $(\pm; \mp)$ to be $(+;+)$ and consider an object $(X_1, X_2)$ of $\uwb$, so that we have the projective  $\kk \uwb$-module $P^\uwb _{(X_1, X_2)} = \kk \uwb ((X_1,X_2), -)$. By Lemma  \ref{lem:left_adjoint_induced}, we have the isomorphism
$$
\lad _{(+;+)} \kk \uwb ((X_1,X_2), -) 
\cong 
\kk \ub (X_1 \amalg X_2, -)
$$
in $\kk \ub$-modules. This can be written as $\lad_{(+;+)} P^\uwb _{(X_1, X_2)} = P^\ub _{X_1 \amalg X_2}$.

We can also apply the functor $\junwall_{(+;+)}$ to $P^\uwb _{(X_1, X_2)}$. By Proposition \ref{prop:J}, evaluating on $Y$
$$
\big( \junwall_{(+;+)} P^\uwb _{(X_1, X_2)}) \big) ( Y) \cong  \bigoplus_{Y = Y_1 \amalg Y_2} \kk \uwb ((X_1,X_2), (Y_1, Y_2)).
$$

The natural surjection (\ref{eqn:junwall_to_lad}) of Corollary \ref{cor:surject_junwall_to_lad} is not an isomorphism in this case.
 For example, take $(X_1, X_2) = (\emptyset, \emptyset) $ and $Y = \mathbf{2}$. Then, for these values:
\begin{eqnarray*}
\dim \big(\lad_{(+;+)} P^\uwb _{(X_1, X_2)}\big)  (Y) &= &1
\\
\dim \big( \junwall _{(+;+)}P^\uwb _{(X_1, X_2)} \big) ( Y)&=&2
\end{eqnarray*}
The difference stems from the fact that $\junwall_{(+;+)} $ retains  information on the `direction' of pairs.
\end{exam}

\section{Downwards: restriction functors and their right adjoints}
\label{sect:new_adjoints}

In this section we consider the counterparts of the results of Section \ref{sect:restrict} for the   downward Brauer categories and their $\kk$-linear variants. Because of the change of variance, these  results concern right adjoints rather than left adjoints. (The results can be deduced from those of Section \ref{sect:restrict}; for convenience, the statements are spelled out.)

This allows us to state and prove Corollary \ref{cor:surject_uspilon^*kkub_unwall_kkuwb} provide the refinement of Theorem \ref{thm:upsilon^*kkub} that was signposted in 
 Remark \ref{rem:more_naturality_upsilon^*kkub}.

%%%%%%%%%%%%%%%%%%%%%%%%%%%%%%%%%%%%%%%%%%%%%%%%%%%%%%%%%%%%%%%%%%
\subsection{The right adjoint to $\X_\pm^*$}
\label{subsect:Xi}

By Corollary \ref{cor:uwb_pm_to_dub_pm},   for given $\mp \in \{ +, -\}$, we have the $\kk$-linear functor 
$
\Upsilon_\mp : (\kk \uwb)_{\mp} 
\rightarrow 
(\kk \dub)_{\mp}.
 $
Passing to the opposite categories gives the following:

\begin{nota}
\label{nota:Xi}
Denote by 
$
\X_\mp : (\kk \dwb)_{\mp} 
\rightarrow 
(\kk \ddb)_{\mp}
 $
 the
$\kk$-linear functor $\Upsilon_\mp\op$. 
\end{nota}
 
On objects, $\X_\mp $ is given by $(X, Y) \mapsto X \amalg Y$. For the case $\mp = +$, this is the $\kk$-linearization of the functor $\dwb \rightarrow \ddb$ that is opposite to the functor given by Proposition   \ref{prop:unwall_uwbord}.

One has the restriction functor 
\begin{eqnarray}
\label{eqn:restriction_dwb_ddb}
\X_\mp ^* : (\kk \ddb)_{\mp} \dash \modules \rightarrow (\kk \dwb)_{\mp} \dash\modules.
\end{eqnarray}
For $\mp = +$, this identifies with the restriction functor $\f (\ddb) \rightarrow \f (\dwb)$ induced by $\dwb \rightarrow \ddb$.

\begin{prop}
\label{prop:Xi_*}
The right adjoint to (\ref{eqn:restriction_dwb_ddb}), 
$$
(\X_\mp)_* : (\kk \dwb)_{\mp} \dash\modules \rightarrow (\kk \ddb)_{\mp} \dash \modules,
$$
is given on $F \in \ob (\kk \dwb)_{\mp} \dash\modules$ by 
$$
\big((\X_\mp)_* F \big)(X) = \amalg_* F (X) = \bigoplus_{X= X_1 \amalg X_2} F (X_1, X_2).
$$

The morphism of $\ddb (X, X \backslash \{s, t\})$ given by the inclusion $X \backslash \{s, t\} \subset X$ and the ordering $(s,t)$ acts on $F(X_1, X_2)$ as follows:
\begin{enumerate}
\item 
if $s \in X_1$ and $t \in X_2$, then it acts via the $(\kk \dwb)_\mp$-module structure of $F$ via the associated morphism $(X_1, X_2) \rightarrow (X_1 \backslash \{s\}, X_2 \backslash \{t\})$; 
\item 
otherwise it acts as zero. 
\end{enumerate}

The functor $(\X_\mp)_*$ is exact. 
\end{prop}

\begin{proof}
This is analogous to the proof of Proposition \ref{prop:Upsilon_sharp}. The statement gives an explicit description of the right Kan extension based upon Lemma \ref{lem:unwalling_lemma}.
\end{proof}

%%%%%%%%%%%%%%%%%%%%%%%%%%%%%%%%%%%%%%%%%%%%%%%%%%%%%%%%%%%%%%%%%%%%%%%%%%%%%%%%%%%%%
\subsection{Relating $(\kk \db)_{(\pm; \mp)}$ and $(\kk \ddb)_\mp$}
\label{subsect:sgntr^*_downward}

\begin{nota}
\label{nota:dsgntr}
Write 
$
\dsgntr_{(\pm; \mp)} 
: 
(\kk \db)_{(\pm; \mp)} \rightarrow (\kk \ddb)_\mp 
 $
 for the  opposite of the $\kk$-linear functor $
\sgntr_{(\pm; \mp)} 
: 
(\kk \ub)_{(\pm; \mp)} \rightarrow (\kk \dub)_\mp
$ introduced in Proposition \ref{prop:signed_transfer}.
\end{nota}

There is a counterpart of Proposition \ref{prop:sgntr^*_upward}, giving the explicit description of $\dsgntr_{(\pm; \mp)} ^*$:

\begin{prop}
\label{prop:sgntr^*_downward}
For given $(\pm; \mp)$ and $N$ a ($\kk \ddb)_\mp$-module:
\begin{enumerate}
\item 
for $X$ an object of $(\kk \db)_{(\pm; \mp)}$, $\big (\dsgntr_{(\pm; \mp)} ^*N \big) (X)$ is $N (X)$; 
\item 
for a degree one morphism $\overline{[f]}$ of $(\kk \db)_{(\pm; \mp)}$ represented by the class $[f]$ of $f \in \ddb (Y, X) = \fb (Y, X \amalg \mathbf{2})$, the morphism $\big(\dsgntr_{(\pm; \mp)} ^* N \big) (\overline{[f]})$ is given by 
$$
 [f]  + \pm [\tau f] : N (Y) \rightarrow N (X).
$$
\end{enumerate}  
\end{prop}

%%%%%%%%%%%%%%%%%%%%%%%%%%%%%%%%%%%%%%%%%%%%%%%%%%%%%%%%%%%%%%%%%%%%%%%%%%%%%%%%%
\subsection{From $(\kk \dwb)_\mp$-modules to $(\kk \db)_{(\pm; \mp)}$-modules: the functor $\unwall_{(\pm;\mp)}$}
\label{subsect:I}

As in Section \ref{subsect:J}, we may now compose the functors introduced above: 

\begin{defn}
\label{defn:I}
For given $(\pm; \mp)$, let $\unwall_{(\pm;\mp)}$ be the composite functor 
$$
\xymatrix{
(\kk \dwb)_\mp\dash \modules 
\ar[r]^{(\X_\mp )_*}
&
(\kk \ddb)_\mp\dash \modules 
\ar[r]^(.45){\dsgntr_{(\pm; \mp)} ^*}
&
(\kk \db)_{(\pm; \mp)}\dash \modules.
}
$$
\end{defn}

The functor $\unwall_{(\pm; \mp)}$ is described explicitly by the  following: 

\begin{prop}
\label{prop:unwall_functor}
The functor $\unwall_{(\pm; \mp)} : (\kk \dwb)_\mp\dash \modules 
\rightarrow
(\kk \db)_{(\pm; \mp)}\dash \modules$ is determined by the following. 

For $M$ a $(\kk \dwb)_\mp$-module: 
\begin{enumerate}
\item 
for $X$ an object of $(\kk \db)_{(\pm; \mp)}$, there is a natural isomorphism
$$
\big( \unwall_{(\pm; \mp)} M \big)  (X) = \bigoplus_{X = X_1 \amalg X_2} M (X_1, X_2);
$$
\item 
for a morphism $\overline{[h]}$ of $(\kk \db)_{(\pm; \mp)}$ represented by  $h \in  \fb (U \amalg \mathbf{2},X)$ (which identifies with $\ddb (X,U)$), restricting 
to the summand indexed by $(X_1, X_2)$, set $U_i =  h^{-1} (X_i)$, and, if $|U_i|= |X_i|-1$, $\tilde{h} \in \dwbord ((X_1, X_2), (U_1, U_2))$ the map given by the inclusions $U_i \subset X_i$; then  
\begin{enumerate}
\item 
if $h(1) \in X_1$ and $h(2) \in X_2$, then $\overline{[h]}$ acts via $M(\tilde{h}) : M(X_1, X_2) \rightarrow M(U_1, U_2)$; 
\item 
if $h(1) \in X_2$ and $h(2) \in X_1$, then $\overline{[h]}$ acts via $ \pm M(\tilde{h}) : M(X_1, X_2) \rightarrow M(U_1, U_2)$;
\item 
otherwise $\overline{[h]}$ acts by zero.
\end{enumerate}
\end{enumerate}
\end{prop}

\begin{proof}
This is the analogue of Proposition \ref{prop:J} and follows by combining Proposition \ref{prop:Xi_*} with
 Proposition \ref{prop:sgntr^*_downward}.

We provide some details on the action of the morphism $\overline{[h]}$ on $\big( \unwall_{(\pm; \mp)} M \big)  (X) = \bigoplus_{X = X_1 \amalg X_2} M (X_1, X_2)$, based on the description of the functor $\dsgntr^*$ from Proposition \ref{prop:sgntr^*_downward}. By that result, $\overline{[h]}$ acts via $(\id + \pm \tau ) \circ [h]$ acting on $ (\X_* M)   (X) = \bigoplus_{X = X_1 \amalg X_2} M (X_1, X_2)$.

Now, consider the restriction to the summand $M (X_1, X_2)$ and define $U_i = h^{-1} (X_i)$ as in the statement. If $|U_i| \neq |X_i|-1$, then Proposition \ref{prop:Xi_*} implies that the action is zero on $M (X_1, X_2)$. Otherwise,  we have $\tilde{h} \in \dwbord ((X_1, X_2), (U_1, U_2))$ given by the inclusions $U_i \subset X_i$;  this is independent of whether $h(1)$ belongs to $X_1$ or to $X_2$. 

By Proposition \ref{prop:Xi_*}, if $h(1) \in X_1$ and $h(2) \in X_2$, then $(\id + \pm \tau ) \circ [h]$ acts as $[h]$, hence as $M (\tilde{h}) : M (X_1, X_2) \rightarrow M (U_1, U_2) \subset ((\X_\mp)_* M)   (U)$. In the other case, i.e., $h(1) \in X_2$ and $h(2) \in X_1$, it acts as $\pm \tau [h]$, hence as $\pm M (\tilde{h}): M (X_1, X_2) \rightarrow M (U_1, U_2) \subset ((\X_\mp)_* M)   (U)$ (here the $\tau$ has been absorbed within the construction, since the definition of  $(U_1, U_2)$ and of $\tilde{h}$  does not depend on the case).
\end{proof}

\begin{exam}
Let $V$ be a finite-dimensional $\kk$-vector space. Then, analogously to the symplectic case considered in  Example \ref{exam:restrict_sympl_vs_module}, we have the $\kk \dwb$-module with values given by
$$
(\mathbf{s}, \mathbf{t}) \mapsto (V^\sharp) ^{\otimes s} \otimes V^{\otimes t}.
$$
Contractions across the wall act via the evaluation map $\varepsilon : V^\sharp \otimes V \rightarrow \kk$.

Now, taking $(\pm; \mp)$ to be $(-;+)$, we can apply the functor $\unwall_{(-;+)}$, which yields a $(\kk \db)_{(-;+)}$-module. It is straightforward to see that this has values:
$$
\n \mapsto (V^\sharp \oplus V) ^{\otimes n};
$$
 the $V^\sharp$ terms come from the left of the wall and the $V$ terms from the right. The action of $\sym_n$ is by place permutations.

To understand the action of the other morphisms in $(\kk \db)_{(-;+)}$, one can  reduce to considering the action of $\overline{[h]}$ for $h \in \ddb (\mathbf{2}, \mathbf{0})$ the map corresponding to the ordered pair $(1 , 2)$. For $X = \mathbf{2}$, we can focus upon the cases $(X_1, X_2) \in \{ (\{1\}, \{2\}), (\{2\}, \{1\}) \}$, since the structure maps in the other cases are obviously zero. 
 In the case  $(X_1, X_2) = (\{1\}, \{2\})$, the structure map identifies as 
$
V^\sharp \otimes V \stackrel{\varepsilon}{\rightarrow}\kk$.
 In the case  $(X_1, X_2) = (\{2\}, \{1\})$, the structure map identifies as 
$
V \otimes V^\sharp \cong 
V^\sharp \otimes V \stackrel{-\varepsilon}{\rightarrow}\kk$,
 with sign due to $\pm =-$.

It follows that this $(\kk \db)_{(-;+)}$-module identifies with that associated to $V^\sharp \oplus V$ equipped with its standard symplectic form (see Example \ref{exam:restrict_sympl_vs_module}).
\end{exam}

%%%%%%%%%%%%%%%%%%%%%%%%%%%%%%%%%%%%%%%%%%%%%%%%%%%%%%%%%%%%%%%%%%%%%%%%%%%%%%%
\subsection{A refinement of Theorem \ref{thm:upsilon^*kkub}}

Now that we have the functor $\unwall_{(\pm; \mp)}$ available to us, we can provide the following consequence of Theorem \ref{thm:upsilon^*kkub} that was advertised in 
 Remark \ref{rem:more_naturality_upsilon^*kkub}:
 
 \begin{cor}
 \label{cor:surject_uspilon^*kkub_unwall_kkuwb}
 There is a surjection of $(\kk \db)_{(\pm; \mp)} \cten  (\kk \uwb)_\mp$-modules:
 $$
 \upsilon_{(\pm; \mp)}^* (\kk \ub)_{(\pm; \mp)}
\twoheadrightarrow 
\unwall_{(\pm; \mp)}(\kk \uwb)_{\mp},
$$
where the variance dictates how the functors are applied.
 This admits a section after restriction to $\kk \fb\op\cten  (\kk \uwb)_\mp$-modules.
\end{cor}
 
 \begin{proof}
 Theorem \ref{thm:upsilon^*kkub} provides the surjection in $\kk \fb\op \cten  (\kk \uwb)_\mp$-modules:
 $$
 \upsilon_{(\pm; \mp)}^* (\kk \ub)_{(\pm; \mp)}
\twoheadrightarrow 
\amalg_\sharp\op (\kk \uwb)_{\mp},
$$
as explained in Remark \ref{rem:more_naturality_upsilon^*kkub}. This is the restriction of the morphism in the statement and, moreover, it is straightforward to check that the Theorem also provides the section in  $\kk \fb\op\cten  (\kk \uwb)_\mp$-modules. It remains to prove that this surjection is compatible with the $(\kk \db)_{(\pm; \mp)} $-module structure.

To show this, let us unwind the definition of this surjection. Consider an element $[f]$ of the domain represented by a class $f \in \fb (U \amalg (\two \times \n), X \amalg Y)$ (as in the proof of Theorem  
\ref{thm:upsilon^*kkub}). If one of the pairs $\two \times \{i\}$ lies entirely within $X$ or entirely within $Y$, then $[f]$ is sent to zero.  Otherwise, without loss of generality, we may assume that the pairs send $(1,i)$ to $X$ and $(2,i)$ to $Y$; thus, writing $U_1 := f^{-1} (X) $ and $U_2:= f^{-1}(Y)$, the map $f$ can be rewritten canonically as an element  $f' \in \tfb ((U_1 \amalg \n, U_2 \amalg \n), (X,Y))$. The class $[f']$ in $(\kk \uwb)_\mp ((U_1, U_2), (X,Y))$ corresponds to the image of $[f]$.

A subtlety in the above is at the `without loss of generality' step. We may require to change the representative of $f$ by changing the direction of some pairs;  this may introduce signs. 

To conclude, we must show that the above surjection is compatible with the action of the degree one morphisms of $ (\kk \db)_{(\pm; \mp)} $, using the identification of $\unwall_{(\pm; \mp)}$ 
 given in Proposition \ref{prop:unwall_functor}. Namely, it suffices to consider the structure morphisms associated to $U \backslash \{ a, b\} \subset U$ for a directed pair $(ab) $ in $U$. Now, for $f$ as above, with the associated decomposition $U= U_1 \amalg U_2$, if either $(ab)$ is contained in $U_1$ or it is contained in $U_2$, then it is clear that the structure morphism is zero on both domain and codomain. 
 
 In the remaining cases, either $a \in U_1$ and $b \in U_2$, in which case the direction of the pair is already compatible with the wall, or $a \in U_2$ and $b\in U_2$, when it is not. Using Proposition \ref{prop:unwall_functor} and the explicit description of the surjection, one checks that the latter is compatible with the respective structure maps. This reduces to the fact that the signs that appear are compatible.
 \end{proof}

 Dualizing (in the spirit of Corollary \ref{cor:dupsilon_kkub^sharp}) gives the following:
 
  \begin{cor}
 \label{cor:inject_junwall_dupsilon}
 There is an injection of $ (\kk \dwb)_\mp \cten   (\kk \ub)_{(\pm; \mp)}$-modules:
 $$
\junwall_{(\pm; \mp)}\big( (\kk \uwb)_{\mp}^\sharp\big) 
\hookrightarrow 
 \dupsilon_{(\pm; \mp)}^*\big( (\kk \ub)_{(\pm; \mp)} ^\sharp \big), 
$$
where the variance dictates how the functors are applied.
 This admits a retract after restriction to $(\kk \dwb)_\mp \cten   \kk \fb$-modules.
 \end{cor}

\part{Structures associated to operads, dioperads, and cyclic operads}
\label{part:opd_diopd_cpd}

 \section{Cyclic operads, dioperads, and operads}
\label{sect:dioperads}

In this section we review cyclic operads, dioperads, and operads together with the relationships between these.
 The presentation is designed so as to make the introduction in Section \ref{sect:mod_cyc_di} of the associated module structures over the appropriate twisted Brauer-type categories transparent.

%%%%%%%%%%%%%%%%%%%%%%%%%%%%%%%%%%%%%%%%%%%%%%%%%%%%%%%%
\subsection{Recollections on cyclic operads}
\label{subsect:recollect_cyclic}

As a warm-up, we give a definition of the category of cyclic operads that is similar to \cite[Definition 6.3]{MR4199072}, but using the Day convolution product $\odot$ in the category of $\kk \fb$-modules to encode the structure morphism. (See Section \ref{sect:day_convolution} for some recollections on $\odot$ and Section \ref{sect:symm_exterior} for the associated symmetric and exterior powers $\sfb^*$ and $\lfb^*$.) Denote by $\kk_\mathbf{2}$  the $\kk \fb$-module supported on $\mathbf{2}$ with value $\kk$ with trivial $\sym_2$-action. 

\begin{exam}
\label{exam:diagonal_kkfb}
For $M$ a $\kk \fb$-module, the $\kk \fb$-module $M \odot \kk _\mathbf{2}$ has values
$$
(M \odot \kk_\mathbf{2}) (X) = \bigoplus_{\{u,v\}\subset X} M (X \backslash \{u,v\}), 
$$
where the sum is over the  cardinal two subsets of $X$.

Consider a morphism of $\kk \fb$-modules $\zeta : M \odot M \rightarrow M \odot \kk_\mathbf{2}$. The restriction of $\zeta_X$ to  the  component  
$M(U) \otimes M(V) \subset M \odot M (X)$ 
corresponding to the ordered decomposition $X = U \amalg V$ is thus  of the form 
\begin{eqnarray}
\label{eqn:M_zeta_components}
M (U) \otimes M (V) \rightarrow \bigoplus_{\{x,y\} \subset X} M(X \backslash \{ x, y\}).
\end{eqnarray}
\end{exam} 

\begin{defn}
\label{defn:M_zeta}
A morphism of $\kk \fb$-modules $\zeta : M \odot M \rightarrow M \odot \kk_\mathbf{2}$ is 
\begin{enumerate}
\item 
{\em diagonal} if, for each decomposition $X = U \amalg V$, the component of (\ref{eqn:M_zeta_components}) indexed by $\{x, y \}$ is zero if either $\{x,y\}$ is contained in $U$ or it is contained in $V$;
\item 
{\em symmetric} if it factorizes as $M \odot M \twoheadrightarrow \sfb^2 M \rightarrow M \odot \kk_\mathbf{2}$, where $M \odot M \twoheadrightarrow \sfb^2 M$ is 
the canonical surjection.
\end{enumerate}
\end{defn}

\begin{rem}
\label{rem:non-triviality_muC}
Suppose that $\zeta$  as in Definition \ref{defn:M_zeta} is diagonal. 
 Thus, in the notation of  (\ref{eqn:M_zeta_components}), the component indexed by  $(U,V) $ and $\{x, y\}\subset X $ can be non-zero only in one of the following cases:
\begin{enumerate}
\item 
$x \in U$ and $y \in V$; 
\item 
$x \in V$ and $y \in U$.
\end{enumerate}
If $\zeta$ is symmetric then $\zeta$ is determined by the components corresponding to the first case.
\end{rem}

\begin{lem}
\label{lem:extend_zeta_sfb}
For $\zeta : M \odot M \rightarrow M \odot \kk _\mathbf{2}$ a symmetric morphism in the sense of Definition \ref{defn:M_zeta}, for each $t \in \nat$ there is a natural extension:
$$
\sfb^t M \rightarrow \sfb^{t-1}M \odot \kk_\mathbf{2}
$$
such that:
\begin{enumerate}
\item 
this is zero for $t<2$; 
\item 
for $t\geq  2$, it is the composite:
$$
\sfb^t M
\rightarrow 
\sfb^{t-2} M \odot \sfb^2 M 
\rightarrow 
\sfb^{t-2} M \odot M \odot \kk_\mathbf{2} 
\rightarrow 
\sfb^{t-1} M \odot \kk_\mathbf{2}, 
$$ 
where the first map is the coproduct $\sfb^t M \rightarrow \sfb^{t-2} M \odot \sfb^2 M$, the second is induced by $\zeta : \sfb^2 M \rightarrow M \odot \kk_\mathbf{2}$, and the last map is induced by the product $\sfb^{t-2} M \odot M \rightarrow \sfb^{t-1}M$.
\end{enumerate} 
\end{lem}

\begin{proof}
The construction of the extension is given in the statement. Naturality  follows from that of the coproduct and product of $\sfb^* M$ and the fact that the morphism $\sfb^2 M \rightarrow M \odot \kk_\mathbf{2}$ is uniquely determined by $\zeta$, since $M \odot M \twoheadrightarrow \sfb^2 M$ is surjective.
\end{proof}

\begin{defn}
\label{defn:cyclic_operad}
A cyclic operad {\em with unit} $\cpd$ is a $\kk \fb$-module such that  $\cpd (\mathbf{0})=0$, equipped with structure maps 
\begin{eqnarray*}
\mu_\cpd &:& \cpd \odot \cpd \rightarrow \cpd \odot \kk_{\mathbf{2}} 
\\
\eta_\cpd &:& \kk_{\mathbf{2}} \rightarrow \cpd
\end{eqnarray*}
that satisfy the following axioms:
\begin{enumerate}
\item 
$\mu_\cpd$ is {\em diagonal};
\item 
$\mu_\cpd$ is {\em symmetric}; 
\item 
$\mu_\cpd$ is {\em associative} in that the following composite is zero
\begin{eqnarray}
\label{eqn:mu_cpd_associative}
\sfb^3 \cpd \rightarrow \sfb^2 \cpd \odot \kk_\mathbf{2} \rightarrow \cpd \odot \kk _\mathbf{2} \odot \kk_\mathbf{2} \twoheadrightarrow \cpd \odot \lfb^2 \kk_{\mathbf{2}},
\end{eqnarray}
where the first two maps are induced by $\mu_\cpd$ (using the extension given by Lemma \ref{lem:extend_zeta_sfb} for the first) and the last is induced by the canonical projection $ \kk _\mathbf{2} \odot \kk_\mathbf{2} \twoheadrightarrow \lfb^2 \kk_\mathbf{2}$.
\item 
$\eta_\cpd$ is the {\em unit} for $(\cpd, \mu_\cpd)$, i.e., the following diagram commutes:
$$
\xymatrix{
\kk _\mathbf{2} \odot \cpd 
\ar[rd]_{\tau}
\ar[r]^{\eta_\cpd \odot \id} 
&
\cpd \odot \cpd 
\ar[d]|{\mu_\cpd} 
&
\cpd \odot \kk_\mathbf{2}
\ar[l]_{\id \odot \eta_\cpd} 
\ar[ld]^\id
\\
&
\cpd \odot \kk_\mathbf{2},
}
$$
where $\tau : \kk _\mathbf{2} \odot \cpd  \stackrel{\cong}{\rightarrow} \cpd \odot \kk_\mathbf{2}$ is the symmetry for $\odot$.
\end{enumerate}

A morphism of cyclic operads with unit $(\cpd, \mu_\cpd, \eta_\cpd) \rightarrow (\cpd' , \mu_{\cpd'} , \eta_{\cpd'})$ is  a morphism $\cpd \rightarrow \cpd'$ of $\kk \fb$-modules that is compatible with the structure morphisms. The category of cyclic operads {\em with unit} is denoted $\copds$.

The category of cyclic operads {\em without unit} is defined as above, but omitting $\eta_\cpd$ and all axioms concerning it. This category is denoted by $\nuco$, so that there is a forgetful functor $\copds \rightarrow \nuco$.
\end{defn}

\begin{rem}
\label{rem:associativity_cyclic}
To show that the definition of the category of cyclic operads given above is equivalent to \cite[Definition 6.3]{MR4199072}, the main point is to show that the associativity axiom above  is equivalent to \cite[Definition 6.3, Axiom (iv)]{MR4199072}. 

Consider the composite 
$
\sfb^3 \cpd \rightarrow \sfb^2 \cpd \odot \kk_\mathbf{2} \rightarrow \cpd \odot \kk _\mathbf{2} \odot \kk_\mathbf{2}. 
$
 Evaluated on a finite set $X$, this has component corresponding to a decomposition $X = U \amalg V \amalg W$ represented by
$$
\cpd (U) \otimes \cpd (V) \otimes \cpd (W)
\rightarrow 
\bigoplus_{(x_1, x_2) , (x_3, x_4)} \cpd ( X \backslash \{ x_1 , x_2, x_3, x_4\}), 
$$
where the $x_i$ are pairwise disjoint, the elements of the pairs $(x_1, x_2)$ and $(x_3, x_4)$ are unordered whereas the pairs are ordered.

Fix a choice of pairs $(x_1, x_2)$, $(x_3, x_4)$ (which can occur in two different orders). Using the symmetry properties together with the fact that $\mu_\cpd$ is diagonal and using Remark \ref{rem:non-triviality_muC}, for these composites to be non-trivial, without significant loss of generality, we may suppose that $x_1 \in U$, $x_2, x_3 \in V$ and $x_4 \in W$. (This is essentially the pigeonhole principle, exploiting diagonality.)

The associativity axiom gives the equality of the two composites obtained using the two possible orders of these pairs; this  is equivalent to  \cite[Definition 6.3, Axiom (iv)]{MR4199072}. 
\end{rem}

%%%%%%%%%%%%%%%%%%%%%%%%%%%%%%%%%%%%%%%%%%%%%%%%%%%%%%%%%%%
\subsection{Dioperads}

We give a presentation of dioperads  using the framework of $\kk (\tfb)$-modules, analogous to that for cyclic operads given in Section \ref{subsect:recollect_cyclic}. (This extends the approach for operads used in \cite{P_opd_wall}.) 

\begin{nota}
\label{nota:kk_mn}
For $m, n \in \nat$, denote by $\kk _{(\m, \n)}$ the $\kk(\tfb)$-module with constant value $\kk$ supported on $(\m, \n)$. This identifies with $\triv_m \boxtimes \triv_n$, considered as a $\kk (\fb \times \fb)$ module supported on $(\m, \n)$.
\end{nota}

The category $\kk (\tfb)$ is equipped with the symmetric monoidal structure $\circledcirc$ given by Day convolution, with unit $\kk_{(\mathbf{0}, \mathbf{0})}$ (see Section \ref{sect:day_convolution} for some recollections and properties of $\circledcirc$). In particular, this yields the symmetric power functors $\stfb^*$ and the exterior power functors $\ltfb^*$ (see Section \ref{sect:symm_exterior}).

\begin{exam}
\label{exam:M_cc_kk}
For a $\kk (\tfb)$-module $M$, one can form $M \circledcirc \kk _{(\one, \one)}$. For $(X,Y)$ an object of $\tfb$, one has 
\begin{eqnarray}
\label{eqn:circledcirc_kk_1_1}
(M \circledcirc \kk_{(\one, \one)})(X,Y) \cong \bigoplus_{x \in X, y\in Y} M (X \backslash \{ x\} , Y \backslash \{y \}).
\end{eqnarray}
Thus, $-\circledcirc \kk_{(\one, \one)}$ can be considered as a `global' form of the induction functor from $\kk (\sym_m \times \sym_n)$-modules to $\kk (\sym_{m+1} \times \sym_{n+1})$-modules, as $(m,n)$ ranges over $\nat^{\times 2}$.

The functor $-\circledcirc \kk_{(\one, \one)}$  can be iterated; for example, one has 
\begin{eqnarray}
\label{eqn:M_cc_kk_cc_kk}
(M \circledcirc \kk_{(\one, \one)} \circledcirc \kk_{(\one, \one)})(X,Y) \cong \bigoplus_{x_1 \neq x_2 \in X, y_1 \neq y_2\in Y} M (X \backslash \{ x_1, x_2\} , Y \backslash \{y_1,y_2\}).
\end{eqnarray}
Here,  the `walled' pairs $(x_1, y_1)$ and $(x_2, y_2)$ are ordered. This corresponds to the fact that $\kk _{(\one, \one)} \circledcirc \kk_{(\one,\one)}$ is supported on $(\mathbf{2}, \mathbf{2})$ and identifies as $\kk \sym_2 \boxtimes \kk \sym_2$ (considered as a $\kk (\tfb)$-module). Changing the order of the pairs corresponds to the diagonal action of $\sym_2$ on $\kk \sym_2 \boxtimes \kk \sym_2$ from the right.

Passing to the quotient by the action of $\sym_2$ exchanging the pairs yields $\stfb^2 \kk_{(\one, \one)}$. One can consider the functor $- \circledcirc \stfb^2 \kk _{(\one, \one)}$; this takes values:
$$
(M \circledcirc \stfb^2 \kk_{(\one, \one)})(X,Y) \cong \bigoplus_{\substack{(u,v), (x,y) \in X\times Y \\  x \neq u, y \neq v}}  M (X \backslash \{ x, u\} , Y \backslash \{y,v\}).
$$
where the two elements of $X \times Y$ are not ordered.
\end{exam}

\begin{rem}
\label{rem:diagonal}
Analogously to Example \ref{exam:diagonal_kkfb}, 
using (\ref{eqn:circledcirc_kk_1_1}), 
a morphism  
$
\gamma :
M \circledcirc M 
\rightarrow 
M \circledcirc \kk_{(\one, \one)}
$
has components of the form 
\begin{eqnarray}
\label{eqn:gamma_components}
M (U, V) \otimes M (X,Y) 
\rightarrow 
M ((U \amalg X) \backslash \{ s\}, (V \amalg Y) \backslash\{ t\}),
\end{eqnarray}
where $s \in (U \amalg X)$ and $t \in (V \amalg Y)$. 
\end{rem}

\begin{defn}
\label{defn:gamma_diagonal_symmetric}
A morphism $
\gamma :
M \circledcirc M 
\rightarrow 
M \circledcirc \kk_{(\one, \one)}
$
is 
\begin{enumerate}
\item 
{\em diagonal} if the component (\ref{eqn:gamma_components}) is   zero whenever one of the following conditions is satisfied: either ($s \in U$ and $t \in V$) or ($s \in X$ and $t \in Y$);
\item 
{\em symmetric} if it factorizes as $M \circledcirc M \twoheadrightarrow \stfb^2 M \rightarrow M \circledcirc \kk_{(\one, \one)}$.
\end{enumerate}
\end{defn}

We have the following counterpart of Lemma \ref{lem:extend_zeta_sfb}, based on the same construction {\em mutatis mutandis}.

\begin{lem}
\label{lem:extend_gamma_stfb}
A symmetric morphism $
\gamma :
M \circledcirc M 
\rightarrow 
M \circledcirc \kk_{(\one, \one)}
$
extends naturally to morphisms for each $t \in \nat$:
$$
\stfb^t M \rightarrow \stfb^{t-1} M \circledcirc \kk_{(\one, \one)}.
$$
\end{lem}

\begin{defn}
\label{defn:dioperads}
A dioperad {\em with unit} $\diopd$ is a $\kk (\tfb)$-module such that $\diopd (\mathbf{0}, \mathbf{0})=0$,  equipped with structure morphisms 
\begin{eqnarray*}
\mu_\diopd &:& \diopd \circledcirc \diopd \rightarrow \diopd \circledcirc \kk_{(\one, \one)}
\\
\eta_\diopd &:& \kk_{(\one, \one)} \rightarrow \diopd
\end{eqnarray*}
that satisfy the following axioms: 

\begin{enumerate}
\item 
$\mu_\diopd$ is {\em diagonal};
\item 
$\mu_\diopd$ is {\em symmetric};
\item 
$\mu_\diopd$ is {\em associative}, in that  the following composite is zero:
\begin{eqnarray}
\label{eqn:mu_diopd_associative}
\stfb^3 \diopd 
\rightarrow 
\stfb^2 \diopd \circledcirc \kk_{(\one,\one)} 
\rightarrow 
\diopd  \circledcirc \kk_{(\one,\one)} \circledcirc \kk_{(\one,\one)} 
\twoheadrightarrow 
\diopd \circledcirc \ltfb^2 \kk_{(\one,\one)},
\end{eqnarray}
where the first two maps are induced by the extension of $\mu_\diopd$ given by Lemma \ref{lem:extend_gamma_stfb}, and the last map is induced by the canonical projection $\kk_{(\one,\one)} \circledcirc \kk_{(\one,\one)} 
\twoheadrightarrow 
 \ltfb^2 \kk_{(\one,\one)}$.
\item 
$\eta_\diopd$ is the {\em unit}, i.e., the following diagram commutes
$$
\xymatrix{
\kk _{(\one, \one)} \circledcirc \diopd
\ar[r]^{\eta_\diopd \circledcirc \id} 
\ar[rd]_{\tau}
&
\diopd \circledcirc \diopd 
\ar[d]|{\mu_\diopd}
&
\diopd \circledcirc \kk _{(\one,\one)} 
\ar[l]_{\id \circledcirc \eta_\diopd}
\ar[ld]^{\id}
\\
&
\diopd \circledcirc \kk_{(\one, \one)},
}
$$
where $\tau$ is the symmetry $\kk _{(\one, \one)} \circledcirc \diopd \stackrel{\cong}{\rightarrow}  \diopd \circledcirc \kk_{(\one, \one)}$ for $\circledcirc$.
\end{enumerate}

A morphism of dioperads $(\diopd, \mu_\diopd, \eta_\diopd) \rightarrow (\diopd', \mu_{\diopd'}, \eta_{\diopd'})$ is a morphism $\diopd \rightarrow \diopd'$ of $\kk (\tfb)$-modules that is compatible with the structure morphisms. The category of dioperads is denoted $\diopds$.

A dioperad {\em without unit} is a pair $(\diopd, \mu_\opd)$ that satisfies the first three axioms above; morphisms are defined analogously. The category of dioperads with unit is denoted by $\nudo$, so that there is a forgetful functor $\diopds\rightarrow \nudo$.

A dioperad $\diopd$ is {\em positive} if for any finite set $X$, $\diopd (X, \mathbf{0}) = 0 = \diopd (\mathbf{0}, X)$. The respective full subcategories of such objects are denoted $\diopds^+$ and $(\nudo)^+$ respectively.
\end{defn}

\begin{rem}
In the definition of dioperads formulated in \cite[Section 1.1]{MR1960128}, Gan imposes a {\em positivity} condition that is the counterpart of that for $\diopds^+$.
\end{rem}

One has the following straightforward relationship between $\diopds^+$ and $\diopds$:

\begin{prop}
\label{prop:diopd_positive}
The inclusion $\diopds^+ \hookrightarrow \diopds$ (respectively $(\nudo)^+ \hookrightarrow \nudo$) admits a right adjoint given by $\diopd \mapsto \diopd^+$, where:
$$
\diopd^+  (X,Y) = 
\left\{
\begin{array}{ll}
\diopd (X,Y) & \mbox{if $|X||Y| >0$} \\
0 & \mbox{otherwise,}
\end{array}
\right.
$$
where the structure maps $\mu_{\diopd^+}$ (and $\eta_{\diopd^+}$ in the case with unit) are the restriction of those of $\diopd$.

The adjunction counit corresponds to the natural inclusion $\diopd^+ \subset \diopd$ in $\diopds$ (respectively $\nudo$) and the adjunction unit identifies as the identity map.
\end{prop}

\begin{proof}
It is clear that $\eta_\diopd$ maps to $\diopd^+$. Hence the result follows by showing that the sub $\kk (\tfb)$-module $\diopd^+ \subset \diopd$ is `stable' under the structure map $\mu_{\diopd}$ of $\diopd$ in the obvious sense. This follows from the fact that, if $|Y_1|>0$ and $|Y_2|>0$ then $|Y_1 \amalg Y_2 | -1 >0$.  It follows that $(-)^+$ is right adjoint to the inclusion and that the adjunction counit is the defining inclusion and that the adjunction unit identifies as the identity.
\end{proof}

\begin{rem}
\label{rem:positive_no_retract}
In general, the inclusion $\diopd^+ \hookrightarrow \diopd$ does not admit a retract in $\diopds$. This is analogous to the fact that an operad does not always admit an augmentation.
\end{rem}

The definition of dioperads given above is equivalent to existing ones in the literature:

\begin{prop}
\label{prop:diopd+_versus_Gan}
The category $\diopds^+$ is equivalent to the category of dioperads defined in \cite[Section 1.1]{MR1960128}.
\end{prop}

\begin{proof}
This is analogous to  the fact that the definition of cyclic operads given in Section \ref{subsect:recollect_cyclic} is equivalent to the classical definition.
We outline the argument for dioperads without unit; the unit is treated similarly. 

The definition in \cite[Section 1.1]{MR1960128} uses a skeletal version, so a dioperad $\diopd$ is given by a family of $\kk (\sym_m \times \sym_n)$-modules $\diopd (m,n)$ indexed by positive integers $m,n $. This extends by Kan extension to a $\kk (\tfb)$-module satisfying the vanishing property:  for all finite sets $X$, $\diopd (X, \emptyset) = 0 = \diopd (\emptyset, X)$.

The structure morphisms in  \cite[Section 1.1]{MR1960128} are of the form 
$$
_i \circ _j \ : \ \diopd (m_1,n_1) \otimes \diopd (m_2, n_2) \rightarrow \diopd (m_1 +m_2 -1, n_1 +n_2 -1) 
$$
where $1 \leq i \leq n_1$ and $1 \leq j \leq m_2$. These satisfy equivariance properties that ensure that these are equivalent to a {\em diagonal} and {\em symmetric} structure map in $\kk (\tfb)$-modules: 
$$
\mu_\diopd : \diopd \circledcirc \diopd \rightarrow \diopd \circledcirc \kk_{(\one, \one)}.
$$ 
(Recall from Remark \ref{rem:diagonal}, that for a morphism $\gamma$  as in the Remark, the diagonal condition reduces to considering components as in (\ref{eqn:gamma_components}) where either ($s \in U$ and $t\in V$) or ($s \in V$ and $t \in U$). In the {\em symmetric} case it suffices to specify the component in the first case in order to define $\gamma$.)

Finally, we consider the associativity axiom(s). The argument is analogous to that outlined in the case of cyclic operads in Remark \ref{rem:associativity_cyclic}. In \cite[Section 1.1]{MR1960128} two associativity conditions (axioms (a) and (b)) are given. In the definition of $\diopds^+$ the single associativity axiom is sufficient, since we have extended the structure maps to  $\mu_\diopd$ which is symmetric.
\end{proof}

\begin{rem}
As in the cyclic case, using the fact that $\mu_\diopd$ is both diagonal and symmetric, we can apply the pigeonhole principle to analyse the composite 
$$
\stfb^3 \diopd 
\rightarrow 
\stfb^2 \diopd \circledcirc \kk_{(\one,\one)} 
\rightarrow 
\diopd  \circledcirc \kk_{(\one,\one)} \circledcirc \kk_{(\one,\one)} .
$$
Consider the evaluation on $(X,Y)$. The codomain is understood as in Example \ref{exam:M_cc_kk}; in particular, we have the expression (\ref{eqn:M_cc_kk_cc_kk}). Using the notation of {\em loc. cit.}, consider the two  components indexed by the two walled pairs $(x_1, y_1)$ and $(x_2, y_2)$ (the two components corresponding to the different orderings of these). 

Consider the domain $(\stfb^3 \diopd) (X,Y)$ and the component corresponding to the  decompositions $X = X_1 \amalg X_2 \amalg X_3$ and $Y= Y_1 \amalg Y_2 \amalg Y_3$:
$$
\diopd (X_1, Y_1 ) \otimes \diopd (X_2, Y_2) \otimes \diopd (X_3, Y_3).
$$ 
Using diagonality, symmetry, and the pigeonhole principle, without significant loss of generality, we may assume that $(x_1, y_1)$  `links' the first two factors and $(x_2, y_2)$ `links' the second two factors (the meaning of linking should become apparent below). However, each `linkage' has two distinct possibilities: 
\begin{enumerate}
\item 
either ($x_1 \in X_1$ and $y_1 \in Y_2$) or ($x_1 \in X_2$ and $y_1 \in Y_1$); 
\item 
either ($x_2 \in X_2$ and $y_2 \in Y_3$) or ($x_2 \in X_3$ and $y_2 \in Y_2$). 
\end{enumerate}

Thus the associativity axiom encompasses the associated four different types of `commutativity/ associativity'. This should be compared with the definition of operads using partial compositions, in which two cases are present: an `associativity' type condition and a `parallel' condition.  In the dioperad case, we also have the `opposite' conditions, switching the r\^ole of inputs and outputs, thus giving the four conditions. 
\end{rem}

We note the following immediate consequence of the definition (this corresponds to the opposite dioperad  introduced in \cite[Section 1.3]{MR1960128} for the case $\diopds^+$).

\begin{prop}
\label{prop:involution_dioperads}
The involution $(-)\op$ of $\tfb$ given on objects by $(X,Y) \mapsto (Y,X)$ induces an involution of $\diopds$, so that $\diopd \op (X,Y):= \diopd (Y,X)$. This restricts to an involution of $\nudo$.
\end{prop}

\begin{proof}
This follows  from the fact that the involution is symmetric monoidal with respect to $\circledcirc$.
\end{proof}

%%%%%%%%%%%%%%%%%%%%%%%%%%%%%%%%%%%%%%%%%%%%%%%%%%%%%%%%%%%%%%%%%%%%%%%%%%%%%%%%%%%%%
\subsection{Operads as dioperads}

We can define operads as being a special sort of dioperad:

\begin{defn}
\label{defn:operads}
An operad (respectively without unit) is a dioperad (resp. without unit) that is supported on objects of the form $(X, \one)$; the associated full subcategory is denoted $\opds \subset \diopds$ (respectively $\nuo \subset \nudo$), so that there is a forgetful functor $\opds \rightarrow \nuo$.
\end{defn}

\begin{rem}
Using the involution $(-)\op$ of Proposition \ref{prop:involution_dioperads}, we could equivalently have defined $\opds$ as the full subcategory of dioperads supported on objects of the form $(\one , X)$. In particular, there are two inclusion functors $\opds \rightrightarrows \diopds$ and these are related by $(-)\op$.
\end{rem}

\begin{prop}
\label{prop:restrict_diopd}
The inclusion $\opds \hookrightarrow \diopds$ admits a right adjoint given by $\diopd\mapsto \opd_\diopd$, where 
 $$
 \opd_\diopd (X,Y) = \left\{ 
 \begin{array}{ll}
 \diopd (X, Y) & \mbox{if $|Y|=1$;}
 \\
 0 & \mbox{otherwise.}
 \end{array}
 \right.
 $$
The structure morphisms are given by restriction of those of $\diopd$. 
  The adjunction counit $\opd_\diopd \rightarrow \diopd$ is the canonical inclusion and the adjunction unit identifies with the identity.
 
The corresponding statements hold for $\nuo \hookrightarrow \nudo$, using the same definition for $\opd_\diopd$.  
\end{prop}

\begin{proof}
The proof is similar to that of Proposition \ref{prop:diopd_positive}.
\end{proof}

\begin{rem}
\label{rem:opd_diopd_retract}
If $\diopd$ is positive (i.e., belongs to $\diopds^+$), then the canonical inclusion $\opd_\diopd \hookrightarrow \diopd$ has a canonical retract in $\diopds^+$. This yields a split extension diagram in $(\nudo)^+$:
$$
\xymatrix{
\overline{\diopd}
\ar@{^(->}[r]
&
\diopd 
\ar@{->>}[r]
&
\opd_\diopd .
\ar@/_1pc/[l] 
}
$$
\end{rem}

%%%%%%%%%%%%%%%%%%%%%%%%%%%%%%%%%%%%%%%%%%%%%%%%%%%%%%%%%%%%%%%%%%%%%%%%%%%%%%%%%%%%%%%%%%%%%%%%%%%%%
\subsection{From cyclic operads to dioperads}

Recall (see Section \ref{sect:day_convolution}) that restriction along $\amalg : \tfb \rightarrow \fb$ induces the functor $\amalg^* : \kk \fb \dash\modules \rightarrow \kk (\tfb)\dash\modules$.
 Explicitly, if $G$ is a $\kk \fb$-module, $(\amalg^* G) (X, Y) = G (X \amalg Y)$.

\begin{exam}
There is an isomorphism in $\kk (\tfb)$-modules:
$$
\amalg^* \kk_\mathbf{2} \cong \kk _{(\one, \one)} \oplus \kk_{(\mathbf{0}, \mathbf{2}) } \oplus \kk_{(\mathbf{2}, \mathbf{0})}.
$$
In particular, there is an inclusion and a projection
$
\kk _{(\one, \one)} 
\hookrightarrow 
\amalg^* \kk_\mathbf{2}
\twoheadrightarrow
\kk _{(\one, \one)}
$.
These can be taken to be  the identity maps when evaluated on $(\one,\one)$.
\end{exam}

The following extends the fact that a cyclic operad has an underlying operad.

\begin{thm}
\label{thm:cpd_to_diopd}
The functor $\amalg^* : \kk \fb \dash\modules \rightarrow \kk (\tfb)\dash\modules$ induces functors
\begin{eqnarray*}
\amalg^* &:& \copds \rightarrow \diopds
\\
 \amalg^* &:& \nuco \rightarrow \nudo
\end{eqnarray*}
that are compatible under forgetting the unit.
\end{thm}

\begin{proof}
We treat the case {\em with unit}; the case {\em without unit} follows by the same argument, simply by omitting the consideration of the unit.

As in Definition \ref{defn:cyclic_operad}, a cyclic operad $\cpd$  with unit is  given by the triple $(\cpd, \mu_\cpd, \eta_\cpd)$, where $\cpd$ is a $\kk \fb$-module such that $\cpd (\mathbf{0})=0$, 
 equipped with the two structure morphisms 
 \begin{eqnarray*}
\mu_\cpd & :& \sfb^2 \cpd \rightarrow \cpd \odot \kk_\mathbf{2}
\\
\eta_\cpd &:& \kk_\mathbf{2} \rightarrow \cpd,
\end{eqnarray*}
using the symmetric form of $\mu_\cpd$. These satisfy the axioms specified in Definition \ref{defn:cyclic_operad}.

The $\kk (\tfb)$-module $\amalg^* \cpd$ clearly vanishes on $(\mathbf{0}, \mathbf{0})$, since $\cpd (\mathbf{0} )=0$. 
Now, the functor $\amalg^*$ is symmetric monoidal with respect to $\odot$ and $\circledcirc$ (see Section \ref{sect:day_convolution}). Hence the structure morphisms of a cyclic operad $\cpd$ induce
\begin{eqnarray*}
&&
\stfb^2 (\amalg^* \cpd ) \rightarrow \amalg ^* \cpd  \circledcirc \amalg^* \kk_\mathbf{2}
\\
&&
\amalg^* \kk_\mathbf{2} 
\rightarrow 
\amalg^* \cpd .
\end{eqnarray*}

Composing the first with the morphism induced by the projection $\amalg^* \kk_\mathbf{2} \twoheadrightarrow \kk_{(\one, \one)}$ and precomposing the second with the inclusion $\kk_{(\one, \one)} \hookrightarrow \amalg^* \kk_\mathbf{2}$ yields the structure morphisms 
\begin{eqnarray*}
\mu_{\amalg^* \cpd} &:&
\stfb^2 ( \amalg^* \cpd ) \rightarrow \amalg ^* \cpd \circledcirc  \kk_{(\one, \one)} 
\\
\eta_{\amalg^* \cpd} 
&:&
 \kk_{(\one, \one)} 
\rightarrow 
\amalg^* \cpd .
\end{eqnarray*}
This construction is clearly natural with respect to the cyclic operad $\cpd$. 

It remains to check that these satisfy the axioms of a dioperad (see Definition \ref{defn:dioperads}), thus inducing a natural dioperad structure on $\amalg^* \cpd$. This is achieved by  using the fact that $\amalg^*$ is symmetric monoidal. The symmetry of $\mu_{\amalg^* \cpd}$ has been built into the construction and the fact that it is diagonal follows directly from the corresponding property of $\mu_\cpd$. Likewise, it is straightforward to check that $\eta_{\amalg^* \cpd}$ satisfies the unit axiom. 

Finally, consider the associativity axiom. Applying $\amalg^*$ to the composite morphism (\ref{eqn:mu_cpd_associative}) in the cyclic case
$$
\sfb^3 \cpd \rightarrow \cpd \odot \kk_\mathbf{2} \odot \kk_\mathbf{2}\twoheadrightarrow \lfb^2 \kk_\mathbf{2} 
$$
yields the top row of the following commutative diagram
$$
\xymatrix{
\stfb^3 (\amalg^* \cpd) \ar[r]
&
 \amalg^* \cpd \circledcirc \amalg^* \kk_\mathbf{2} \circledcirc \amalg^* \kk_\mathbf{2}
\ar@{->>}[r]
\ar@{->>}[d]
& 
  \amalg^* \cpd \circledcirc \ltfb^2 (\amalg^* \kk _\mathbf{2}) 
\ar@{->>}[d]
\\
&
 \amalg^* \cpd \circledcirc \kk_{(\one ,\one)} \circledcirc \kk_{(\one,\one)} 
\ar@{->>}[r]
&
 \amalg^* \cpd \circledcirc \ltfb^2 \kk_{(\one,\one)},
}
$$
in which the vertical morphisms are induced by the projection $\amalg^* \kk_\mathbf{2} \twoheadrightarrow \kk_{(\one,\one)}$; commutativity follows by the naturality of the projection to $\ltfb^2$.
The total composite in the diagram identifies as the composite of  (\ref{eqn:mu_diopd_associative}).
 
The associativity condition for $\mu_\cpd$ implies that the composite of the top row is zero, whence so is the total composite of the diagram, establishing the associativity axiom for $\mu_{\amalg^* \cpd}$.
\end{proof}

\begin{rem}
The dioperad $\amalg^* \cpd$ constructed in Theorem \ref{thm:cpd_to_diopd} has two underlying operads $\opd _{\amalg^* \cpd}$ and $\opd_{(\amalg^* \cpd)\op}$, by Proposition \ref{prop:restrict_diopd}. These are related by the obvious isomorphism $
\opd_{(\amalg^* \cpd)\op} \cong (\opd _{\amalg^* \cpd}) \op 
$
(considering these as dioperads so as to form $(-)\op$).
 The underlying operad of $\cpd$ can thus be taken (essentially) unambiguously to be $\opd _{\amalg^* \cpd}$.
\end{rem}

\begin{cor}
\label{cor:copds_to_opds}
The composites
\begin{eqnarray*}
&&\copds \stackrel{\amalg^*} {\rightarrow} \diopds \stackrel{\opd_\bullet}{\rightarrow} \opds
\\
&&
\nuco \stackrel{\amalg^*} {\rightarrow} \nudo \stackrel{\opd_\bullet}{\rightarrow} \nuo
\end{eqnarray*}
define functors from cyclic operads to operads (respectively with and without unit).
\end{cor}

%%%%%%%%%%%%%%%%%%%%%%%%%%%%%%%%%%%%%%%%%%%%%%%%%%%%%%%%%%%%%%%%%%%%%%%%%%%%%%%%%%%%%%%%%%%%%%%%%%%%%%%
\subsection{From dioperads to cyclic operads}

One also has the functor $\amalg_* : \kk (\tfb)\dash\modules \rightarrow  \kk \fb \dash\modules $, the right adjoint (and left adjoint) to $\amalg^*$ (see Section \ref{sect:day_convolution}).  The functor $\amalg_*$ is given explicitly, for $F$ a $\kk (\tfb)$-module, by
$$
(\amalg_* F) (X) = \bigoplus_{X= X_1 \amalg X_2} F (X_1, X_2). 
$$
This is  symmetric monoidal with respect to the Day convolution products $\circledcirc$ and $\odot$ (see Section \ref{sect:day_convolution}).

\begin{exam}
\label{exam:splitting_amalg_*_kk11}
There is an isomorphism $\amalg_* \kk_{(\one, \one)} \cong \kk \sym_2$, where $\kk \sym_2$ is considered as a $\kk \fb$-module supported on $\mathbf{2}$. Under Hypothesis \ref{hyp:car0}, one has the  splitting $\kk \sym_2 \cong \triv_2 \oplus \sgn_2$, hence one can restrict (or project) to the trivial representation $\triv_2 = \kk_\mathbf{2}$.

To be concrete, we take the projection $\kk \sym_2 \rightarrow \kk$ to be the morphism induced on group rings by $\sym_2 \rightarrow \{e\}$. Correspondingly, we take the inclusion $\kk \hookrightarrow \kk \sym_2$ to be the section of this, so that $1 \mapsto \frac{1}{2} ( [e] + [\tau])$, writing $\sym_2 = \{ e,\tau \}$.
\end{exam}

We have the following counterpart of Theorem \ref{thm:cpd_to_diopd}.

\begin{thm}
\label{thm:diopd_to_cpd}
The functor $\amalg_* : \kk (\tfb)\dash\modules \rightarrow  \kk \fb \dash\modules $ induces functors
\begin{eqnarray*}
\amalg_* &: &\diopds \rightarrow \copds
\\
\amalg_* &:& \nudo \rightarrow \nuco
\end{eqnarray*}
that are compatible under forgetting the unit.
\end{thm}

\begin{proof}
Consider a dioperad $(\diopd, \mu_\diopd, \eta_\diopd)$, with structure morphisms
\begin{eqnarray*}
\mu_\diopd &:& \stfb^2 \diopd \rightarrow \diopd \circledcirc \kk_{(\one, \one)} \\
\eta_\diopd &:& \kk_{(\one, \one)} \rightarrow \diopd.
\end{eqnarray*}

Applying $\amalg_*$ to these structure morphisms and composing with $\amalg_* \kk_{(\one, \one)} \twoheadrightarrow \kk_\mathbf{2}$ (respectively precomposing with $\kk_\mathbf{2} \hookrightarrow \amalg_* \kk_{(\one,\one)}$) yields:
\begin{eqnarray*}
\mu_{\amalg_* \diopd} &: &
\sfb^2 ( \amalg_* \diopd ) \rightarrow \amalg_* \diopd \odot \kk_\mathbf{2} \\
\eta_{\amalg_* \diopd} &:& 
\kk_{\mathbf{2}} \rightarrow \amalg_* \diopd.
\end{eqnarray*}
This construction is clearly natural with respect to $(\diopd, \mu _\diopd, \eta_\diopd)$.

It remains to show that, with respect to these structure morphisms, $\amalg_* \diopd$ has the structure of a cyclic operad. This follows by the same reasoning as used in the proof of Theorem \ref{thm:cpd_to_diopd}.
\end{proof}

We may restrict to the respective full subcategories of operads:

\begin{cor}
\label{cor:opds_to_copds}
The functor $\amalg_*$ of Theorem \ref{thm:diopd_to_cpd} yields the composite functors:
\begin{eqnarray*}
&&\opds \hookrightarrow \diopds \stackrel{\amalg_*}{ \rightarrow} \copds
\\
&&\nuo \hookrightarrow \nudo \stackrel{\amalg_*}{ \rightarrow} \nuco.
\end{eqnarray*}
\end{cor}

\begin{rem}
One can show that the functor $\nuo \rightarrow \nuco$ of Corollary \ref{cor:opds_to_copds} is right adjoint to $\nuco \rightarrow \nuo$ of Corollary \ref{cor:copds_to_opds}.
 This is most easily seen by showing that the latter sends free cyclic operads to free operads, compatibly with $\amalg^*$. This is not the case for dioperads and will not be used here.
\end{rem}

 \section{Modules over twisted Brauer categories associated to cyclic operads and dioperads}
\label{sect:mod_cyc_di}

The purpose of this section is to explain the construction of modules over various flavours of twisted Brauer categories that are associated to cyclic operads and to dioperads; the constructions do not require a unit.

%%%%%%%%%%%%%%%%%%%%%%%%%%%%%%%%%%%%%%%%%%%%%%%%%%%%%%%%%%%%%%%%%%%%%%%
\subsection{The cyclic operad case}

We work with cyclic operads without unit (i.e., objects of $\nuco$), so that we have a $\kk \fb$-module  $\cpd$ such that $\cpd(\mathbf{0}) =0$, equipped with a structure morphism 
$$
\mu_\cpd : \cpd \odot \cpd \rightarrow \cpd \odot \triv_\mathbf{2}
$$
that factorizes across the canonical surjection $\cpd \odot \cpd \twoheadrightarrow \sfb^2 \cpd$.

The result that motivated our consideration of modules over the categories $(\kk \db)_{(\pm; \mp)}$ is the following. 

\begin{thm}
\label{thm:cpd_db_kdbm}
\cite{P_cyclic}
For $\cpd$ in $\nuco$, the structure map $\mu_\cpd$ induces:
\begin{enumerate}
\item 
 a natural $\kk \db$-module structure  on $\sfb^* \cpd$; 
\item 
a natural $(\kk \db)_{(-;-)}$-module structure on $\lfb^* \cpd$. 
\end{enumerate}
This yields the functors:
\begin{eqnarray*}
&& \nuco \rightarrow \kk \db \dash \modules
\\
&&
\nuco \rightarrow (\kk \db)_{(-;-})\dash \modules.
\end{eqnarray*}
\end{thm}

\begin{proof}
We include a sketch proof here, as a warm-up for the proof of the counterpart for dioperads, Theorem \ref{thm:diopd_dwb_kkdwb-_modules} below.

Consider the case of $\sfb^* \cpd$, which is, by construction, a $\kk \fb$-module. Since $\kk \db$ is a homogeneous quadratic category over $\kk \fb$, to define the $\kk \db$-module structure, we need to specify how the degree one morphisms of $\kk \db$ act and then check that this is compatible with the quadratic relations. 

The action of the degree one morphisms is given by the structure morphisms 
$$
\sfb^n \cpd \rightarrow \sfb^{n-1}\cpd  \odot \triv_\mathbf{2}, 
$$ 
constructed from $\mu_\cpd$ as in Lemma \ref{lem:extend_zeta_sfb}. Using the associativity of the product and coassociativity of the coproduct of $\sfb^*$, the fact that these are compatible with the quadratic relations is an immediate consequence of the associativity axiom for $\mu_\cpd$ for the cyclic operad $\cpd$.

The proof in the case of $\lfb^* \cpd$ is similar. In this case, the structure morphism $\mu_\cpd$ induces 
$$
\lfb^2 \cpd \rightarrow \cpd \odot \sgn_\mathbf{2},
$$
where the $\sgn_\mathbf{2}$  (in place of $\triv_\mathbf{2}$) arises because of that appearing in the definition of $\lfb^2$.

There is an obvious counterpart of Lemma \ref{lem:extend_zeta_sfb} using $\lfb^*$ in place of $\sfb^*$.
This yields the structure morphisms
$$
\lfb^n \cpd \rightarrow \lfb^{n-1} \cpd \odot \sgn_\mathbf{2}
$$
(taken to be zero for $n \leq 1$). These give the action of degree one morphisms of $(\kk \db)_{(-;-)}$. 

The final point is to check the quadratic relation; in this case, a sign arises due to the permutations that are used in the construction, accounting for the second `$-$' in the twisting. 
\end{proof}

\begin{rem}
Theorem \ref{thm:cpd_db_kdbm}  yields two different (non-equivalent) flavours of modules associated to a cyclic operad $\cpd$. This is related to the fact that there are odd and even graph complexes associated to a cyclic operad. 
\end{rem}

%%%%%%%%%%%%%%%%%%%%%%%%%%%%%%%%%%%%%%%%%%%%%%%%%%%%%%%%%%%%%%%%%%%%%%%%
\subsection{The dioperad case}

We work with dioperads without unit (i.e., objects of $\nudo$). Thus we have an underlying $\kk (\tfb)$-module $\diopd$ such that $\diopd (\mathbf{0}, \mathbf{0})=0$, equipped with a structure morphism 
$$
\mu_\diopd : \diopd \circledcirc \diopd \rightarrow \diopd \circledcirc \kk_{(\one, \one)}. 
$$
This factorizes across the canonical surjection $\diopd \circledcirc \diopd  \twoheadrightarrow \stfb^2 \diopd$.

The special case of the following result for operads was proved in \cite{P_opd_wall}: 

\begin{thm}
\label{thm:diopd_dwb_kkdwb-_modules}
For $\diopd$ in $\nudo$, the structure map $\mu_\diopd$ induces:
\begin{enumerate}
\item 
a natural $\kk \dwb$-module structure on $\stfb^* \diopd$; 
\item 
a natural $(\kk \dwb)_-$-module structure on $\ltfb^* \diopd$. 
\end{enumerate}
\item 
This yields the  functors: 
\begin{eqnarray*}
&& 
\nudo \rightarrow \kk \dwb\dash\modules 
\\
&&
\nudo \rightarrow (\kk \dwb)_-\dash\modules.
\end{eqnarray*}
\end{thm}

\begin{proof}
The proof is analogous to that of Theorem \ref{thm:cpd_db_kdbm} and is also  a straightforward generalization of the operad case treated in \cite{P_opd_wall}.

In the case of $\stfb^* \diopd$, by Lemma \ref{lem:extend_gamma_stfb}, the structure morphism 
$
\stfb^2 \diopd \rightarrow \diopd \circledcirc \kk_{(\one,\one)}
$
induces morphisms 
 $$
 \stfb^n \diopd \rightarrow \stfb^{n-1} \diopd \circledcirc \kk_{(\one,\one)}
 $$
for $n\in \nat$. These give the action of the degree one morphisms of $\kk \dwb$ on $\stfb^* \diopd$. 

The fact that the quadratic relations in $\kk \dwb$ are satisfied is, analogously to  the cyclic case, a consequence of the associativity axiom for $\mu_\diopd$.

In the case of $\ltfb^* \diopd$, we first show that $\mu_\diopd$ induces a structure map 
\begin{eqnarray}
\label{eqn:ltfb^2_diopd_mu}
\ltfb^2 \diopd \rightarrow \diopd \circledcirc \kk_{(\one,\one)}
\end{eqnarray}
in $\kk (\tfb)$.

We require to show that, for  all possible  $X_1 \backslash \{s\} \subset X_1$ and  $X_2 \backslash \{t\} \subset X_2$, the structure map $\mu_\diopd : \diopd \circledcirc \diopd \rightarrow \diopd \circledcirc \kk_{(\one, \one)}$ induces 
$$
\big (\ltfb^2 \diopd \big)  (X_1, X_2) \rightarrow \diopd (X_1 \backslash \{s\}, X_2 \backslash \{t\}).
$$
and that these define a morphism of the form (\ref{eqn:ltfb^2_diopd_mu}).

Now, by definition
$$
\big (\ltfb^2 \diopd \big)  (X_1, X_2)
\cong 
\Big (
\bigoplus_{\substack{X_1 = U_1 \amalg V_1 \\X_2 = U_2 \amalg V_2}}
\diopd (U_1, U_2) \otimes \diopd (V_1 , V_2) \otimes \sgn_2
\Big) / \sym_2,
$$
where the action of $\sym_2$ transposes the tensor factors of $\diopd (-,-)$ up to  the sign from $\sgn_2$, the sum being over compatibly ordered decompositions of $X_1$ and $X_2$. After passage to the quotient, as a $\kk$-vector space, this is  non-canonically isomorphic to 
$$
\bigoplus_{\substack{(U_1,U_2) \\ (V_1, V_2) }}
\diopd (U_1, U_2) \otimes \diopd (V_1 , V_2)
$$
where the sum is now over {\em  unordered} decompositions such that $X_1 = U_1 \amalg V_1 $ and $X_2 = U_2 \amalg V_2$.

Consider a term of the decomposition corresponding to $(U_1, U_2)$ and $(V_1, V_2)$ (unordered). For given $s, t$, by the diagonal axiom for $\mu_\diopd$, considering $\diopd (U_1, U_2) \otimes \diopd (V_1 , V_2)$ as a component of $\diopd \circledcirc \diopd$ by choosing an arbitrary  ordering, the component of $\mu_\diopd$ corresponding to $(s,t)$ acts trivially  unless either ($s \in U_1 $ and $t \in V_2$) or ($s \in V_1$ and $t \in U_2$). If one of these holds, we can thus choose a canonical (depending on $s,t$) ordering of $(U_1, U_2)$ and $(V_1, V_2)$ by assuming that $s\in U_1$ (otherwise switch the labelling). The corresponding  map is taken to be 
\begin{eqnarray}
\label{eqn:ltfb_diopd_structure_map}
\diopd (U_1, U_2) \otimes \diopd (V_1 , V_2)
\rightarrow 
 \diopd (X_1 \backslash \{s\}, X_2 \backslash \{t\})
\end{eqnarray}
given by $\mu_\diopd$.  One checks that this yields a well-defined structure map on $\ltfb^2 \diopd$ (i.e., there is no ambiguity due to the presence of $\sgn_2$). 

The analogue of Lemma \ref{lem:extend_gamma_stfb} for $\ltfb^* $ in place of $\stfb^*$ shows that this structure map induces, for all $n \in \nat$:
\begin{eqnarray}
\label{eqn:ltfb_diopd_structure}
\ltfb^n \diopd 
\rightarrow 
(\ltfb^{n-1} \diopd)  
\circledcirc 
\kk_{(\one, \one)}.
\end{eqnarray}
These define the action of the degree one morphisms of $(\kk \dwb)_-$ on $\ltfb^* \diopd$. 

The compatibility with the quadratic relation is checked similarly to  the cyclic case (Theorem \ref{thm:cpd_db_kdbm}) as a consequence of  the associativity axiom for dioperads. As in the cyclic case, there is a Koszul sign that intervenes, accounting for the fact that the structure is over $(\kk \dwb)_-$ rather than $\kk \dwb$.
\end{proof}

\begin{rem}
In the general dioperad case, one cannot  deduce the case of $\ltfb^* \diopd$ from the case $\stfb^* \diopd$ by `twisting' (as opposed to the operad case treated in \cite{P_opd_wall}). 
\end{rem}

\section{Relating module structures associated to cyclic operads and dioperads}
\label{sect:relating_mod}

The aim of this section is to relate the module structures introduced in Section \ref{sect:mod_cyc_di}. 
We work with cyclic operads and dioperads without requiring a unit. Recall that, if $\cpd$ is in $\nuco$, then Theorem \ref{thm:cpd_to_diopd} provides the dioperad $\amalg^* \cpd$ (without unit),  and that, if $\diopd$ is in $\nudo$, Theorem \ref{thm:diopd_to_cpd} provides the cyclic operad $\amalg_* \diopd$ (without unit).

%%%%%%%%%%%%%%%%%%%%%%%%%%%%%%%%%%%%%%%%%%%%%%%%%%%%%%%%%
\subsection{From dioperads to cyclic operads}

Recall from Section \ref{subsect:I} that we have the functor
$$
\unwall_{(\pm; \mp)} : (\kk \dwb)_\mp\dash \modules 
\rightarrow
(\kk \db)_{(\pm; \mp)}\dash \modules
$$
that acts on the underlying objects via $\amalg_* : \kk (\tfb) \dash \modules \rightarrow \kk \fb\dash \modules$. This functor is described explicitly in Proposition \ref{prop:unwall_functor}.

\begin{thm}
\label{thm:diopd_2_cyclic_modules}
For $\diopd$ a dioperad in $\nudo$, there are natural isomorphisms:
\begin{eqnarray*}
\sfb^* (\amalg_* \diopd) & \cong & \unwall_{(+;+)} \stfb^* \diopd 
\\
\lfb^* (\amalg_* \diopd) & \cong & \unwall_{(-;-)} \ltfb^* \diopd 
\end{eqnarray*}
of $\kk \db$-modules and $(\kk \db)_{(-;-)}$-modules respectively.
\end{thm}

\begin{proof}
The functor $\amalg_* : \kk (\tfb) \dash \modules \rightarrow \kk \fb\dash \modules$ is symmetric monoidal with respect to the Day convolution products (see Section \ref{sect:day_convolution}). This implies the isomorphisms at the level of the underlying $\kk \fb$-modules. Moreover, in each case, by construction, both sides are equipped with the appropriate module structure. Hence, since $\kk \db$ and $(\kk \db)_{(-:-)}$ are both homogeneous quadratic categories, to prove the result, it suffices to prove that the actions of the degree one morphisms coincide. 

Write $\cpd$ for the cyclic operad (without unit) associated to $\diopd$, which has underlying $\kk \fb$-module $\amalg_* \diopd$. Below we consider the second case, since this is the most delicate due to the signs. The first case is treated by the same arguments, {\em mutatis mutandis}.

 Suppose give a finite set $X$ and an ordered pair $(s,t)$ with $\{s, t\} \subset X$.   
 We have to show that the associated structure morphisms 
 \begin{eqnarray*}
&& \lfb^* \cpd (X) \rightarrow \lfb^* \cpd (X \backslash \{s,t\}) 
 \\
&&(\unwall_{(-;-)} \ltfb^* \diopd) (X) 
\rightarrow  
(\unwall_{(-;-)} \ltfb^* \diopd) (X \backslash \{s,t\}) 
\end{eqnarray*}
identify, 
using the fact that $\amalg_*$ is symmetric monoidal to identify the underlying object $\amalg_* \ltfb^* \diopd$ with $\lfb^* (\amalg_* \diopd) = \lfb^* \cpd$. (Note that, by the constructions, changing the order of $s$ and $t$ introduces a sign in both cases.)

In each case, using the above identification, the morphisms are the direct sum over $n\in \nat$ of morphisms of the form 
$$
\lfb^n \cpd (X) 
\rightarrow 
\lfb^{n-2} \cpd (X\backslash \{s, t\}).
$$
We analyse the case $n=2$; the general case can be deduced from this. 

First consider the morphism given by the cyclic structure of $\cpd$. Using the given order of $s$ and $t$, we can write this canonically (i.e., fixing a choice of order of the tensor product in the domain) as 
$$
\bigoplus_{\substack{X = U \amalg V \\ s\in U}}
\cpd (U) \otimes \cpd (V)
\rightarrow 
\cpd (X \backslash \{s, t\}).
$$
On the components with $\{s, t\} \subset U$, this map is zero. On the components with $s \in U$ and $t \in V$
it is given by the structure map of $\cpd$. (This is well-defined: changing the order of $s$, $t$ requires transposing the order of the tensor factors of the domain; the two associated signs compensate for each other.)

Now consider the structure morphism for $\unwall_{(-,-)} \ltfb^* \diopd$. Here, we use the same normalization as above, giving the expression:
$$
\bigoplus_{\substack{X = U \amalg V \\ s\in U}}
\bigoplus_{\substack{U = U_1 \amalg U_2 \\ V = V_1 \amalg V_2}}
\diopd (U_1, U_2) 
\otimes 
\diopd (V_1 , V_2) 
\rightarrow 
\bigoplus_{X\backslash\{ s, t\} = Y_1 \amalg Y_2} \diopd (Y_1, Y_2).
$$

Fix a summand of the domain $\diopd (U_1, U_2) 
\otimes 
\diopd (V_1 , V_2)$ and consider the restriction of the structure map to this summand. Since this is induced from the structure of $ \diopd$, this restriction is zero if $t \in U$, as in the cyclic case. Moreover, it is also zero if either ($s \in U_1$ and $t \in V_1$) or ($s\in U_2$ and $t\in V_2$), by the identification of $\unwall_{(-;-)}$ given in Proposition \ref{prop:unwall_functor}. 

Hence, we reduce to considering the  two possibilities: 
\begin{enumerate}
\item 
$s \in U_1$ and $t \in V_2$;
\item 
$s \in U_2$ and $t \in V_1$.
\end{enumerate}
In the first case, the chosen order $(s, t)$ is compatible with the wall. In the second it is not;   we reduce to the first case by changing the order of $(s,t)$. Here, the sign in the explicit description of $\unwall_{(-;-)}$ given in Proposition  \ref{prop:unwall_functor} corresponds to the sign that arises from $(\kk \db)_{(-;-)}$ in changing the order of the pair. 

Thus, we may assume that the only non-zero component of the structure map restricted to $\diopd (U_1, U_2) 
\otimes 
\diopd (V_1 , V_2) $ is 
$$
\diopd (U_1, U_2) 
\otimes 
\diopd (V_1 , V_2) 
\rightarrow 
\diopd ((U_1 \amalg V_1)\backslash\{s\}, (U_2 \amalg V_2) \backslash \{t\}). 
$$

By Proposition \ref{prop:unwall_functor}, this is given by the structure map of the $(\kk \dwb)_{-}$-module $\ltfb^* \diopd$, as constructed in Theorem \ref{thm:diopd_dwb_kkdwb-_modules}. By equation (\ref{eqn:ltfb_diopd_structure_map}) of the proof of that theorem, this is given by the structure map $\mu_\diopd$. Since this also underlies the definition of the structure map $\mu_\cpd$ of $\cpd$ by  Theorem \ref{thm:diopd_to_cpd}, putting the above identifications together, the result follows for $n=2$. The general case then follows from this.
\end{proof}

\begin{rem}
Recall from Definition \ref{defn:I} that  $\unwall_{(\pm;\mp)}$ is the composite functor 
$$
\xymatrix{
(\kk \dwb)_\mp\dash \modules 
\ar[r]^{(\X_\mp )_*}
&
(\kk \ddb)_\mp\dash \modules 
\ar[r]^(.45){\dsgntr_{(\pm; \mp)} ^*}
&
(\kk \db)_{(\pm; \mp)}\dash \modules.
}
$$
Hence  the following intermediate structures  
\begin{enumerate}
\item 
the $\kk \ddb$-module $(\X_+)_* \stfb^* (\diopd)$; 
\item 
the $(\kk \ddb)_-$-module $(\X_-)_* \ltfb^* (\diopd)$.
\end{enumerate}
arise implicitly in Theorem \ref{thm:diopd_2_cyclic_modules}

These are in some sense more fundamental structures, in that they retain information from the fact that a dioperad is encoded using trees with directed edges.
\end{rem}

%%%%%%%%%%%%%%%%%%%%%%%%%%%%%%%%%%%%%%%%%%%%%%%%%%%%%%%%%%%%%%%%%%%%%%%%%%%%%%%%%%%%%%%%%%%
\subsection{From cyclic operads to dioperads}

Recall that, for a cyclic operad $\cpd$ in $\nuco$, we have the associated dioperad $\amalg^* \cpd$ in $\nudo$ given by Theorem \ref{thm:cpd_to_diopd}. We also have the functor 
$$
\dupsilon_{(\pm; \mp)} ^*
: 
(\kk \db)_{(\pm; \mp)} \dash \modules 
\rightarrow 
(\kk \dwb)_\mp \dash \modules
$$
of Notation \ref{nota:upsilon}.

The functor $\dupsilon_{(\pm;\mp)}^*$ is described explicitly in Proposition \ref{prop:upsilon^*_dupsilon^*}. Restricted to the underlying objects, it identifies with $\amalg^* : \kk \fb\dash \modules \rightarrow \kk (\tfb)\dash \modules$, so that  $\amalg ^* \cpd (X,Y) = \cpd (X \amalg Y)$.

\begin{thm}
\label{thm:cyclic_2_diopd_modules}
For $\calc$ a cyclic operad in $\nuco$, there are natural isomorphisms
\begin{eqnarray*}
\stfb^* (\amalg^* \cpd) &\cong & \dupsilon_{(+;+)}^* \sfb^* (\cpd) 
\\
\ltfb^* (\amalg^* \cpd) & \cong & \dupsilon_{(-;-)}^* \lfb^* (\cpd)
\end{eqnarray*}
of $\kk \dwb$-modules and $(\kk \dwb)_-$-modules respectively. 
\end{thm}

\begin{proof}
The proof is analogous to that of Theorem \ref{thm:diopd_2_cyclic_modules}. The underlying $\kk(\tfb)$-modules are isomorphic, so we reduce to showing that the action of the degree one morphisms of $(\kk \dwb)_\mp$ identify. Once again, we focus upon the more delicate second case, the first being proved by a similar argument. 

Write $\diopd$ for the dioperad associated to $\cpd$, which has underlying object $\amalg^* \cpd$. Consider a pair of finite sets $(X_1, X_2)$ and $(s,t) \in X_1 \times X_2$ a walled pair. We require to prove that the associated structure morphisms 
\begin{eqnarray*}
&&\ltfb^* \diopd (X_1, X_2) 
\rightarrow 
\ltfb^* \diopd (X_1 \backslash \{s\}, X_2\backslash \{ t\} ) 
\\
&&
\dupsilon_{(-;-)}^* \lfb^* \cpd (X_1, X_2) 
\rightarrow 
\dupsilon_{(-;-)}^* \lfb^* \cpd  (X_1 \backslash \{s\}, X_2\backslash \{ t\} ) 
\end{eqnarray*}
identify. The first is given by the $(\kk \dwb)_-$-module structure furnished by Theorem \ref{thm:diopd_dwb_kkdwb-_modules}. The second is given by applying $\dupsilon_{(-;-)}^*$ to the $(\kk \db)_{(-;-)}$-module structure given by Theorem \ref{thm:cpd_db_kdbm}.

These both identify as the direct sum of morphisms of the form 
$$
\ltfb^n \diopd (X_1, X_2) 
\rightarrow 
\ltfb^{n-1} \diopd (X_1 \backslash \{s\}, X_2\backslash \{ t\} ) .
$$
As in the proof of Theorem \ref{thm:diopd_2_cyclic_modules}, one can reduce to the case $n=2$. 

First consider the structure morphism coming from the $(\kk \dwb)_-$-module $\ltfb^* \diopd$. This can be written as 
$$
\bigoplus_{\substack{X_1 = U_1 \amalg U_2 \\ X_2 = V_1  \amalg V_2 \\ s\in U_1}}
\diopd (U_1, V_1) \otimes \diopd (U_2, V_2) 
\rightarrow 
\diopd (X_1 \backslash \{s\}, X_2\backslash \{ t\} ),
$$
where we have ordered the tensor product by the condition $s \in U_1$.

Restricted to the summand $\diopd (U_1, V_1) \otimes \diopd (U_2, V_2) $, by construction, this morphism is zero if $t\in V_1$, hence we may suppose that $t \in V_2$. In this case, by Theorem \ref{thm:cpd_to_diopd}, it is given by the structure morphism of $\diopd$, and hence by the cyclic operad structure morphism
$$
\cpd (U_1 \amalg V_1) \otimes \cpd (U_2 \amalg V_2) 
\rightarrow 
\cpd (X_1  \amalg X_2 \backslash \{s,t\} ).
$$

Now consider the structure morphism arising from $\dupsilon_{(-;-)}^* \lfb^* \cpd$, using the identification of $\dupsilon_{(\pm;\mp)}^*$ given in Proposition \ref{prop:upsilon^*_dupsilon^*}, 
 together with the $(\kk \db)_{(-;-)}$-module structure of $\lfb^* \cpd$. This yields:
$$
\lfb^2 \cpd (X_1 \amalg X_2) 
\rightarrow 
\cpd (X_1  \amalg X_2 \backslash \{s,t\} )
$$
as determined by Proposition \ref{prop:upsilon^*_dupsilon^*}. This can be written as 
$$
\bigoplus_{\substack{X_1 = U_1 \amalg U_2 \\ X_2 = V_1  \amalg V_2 \\ s\in U_1}}
\cpd (U_1\amalg  V_1) \otimes \cpd (U_2 \amalg  V_2) 
\rightarrow 
\cpd (X_1  \amalg X_2 \backslash \{s,t\} ).
$$
Moreover, restricted to $\cpd (U_1\amalg  V_1) \otimes \cpd (U_2 \amalg  V_2) $, the morphism is zero if $t \in V_1$ and, in the remaining case $t= V_2$, is given by the structure morphism of the cyclic operad $\cpd$. 

This gives the desired identification in the case $n=2$. From this one deduces the cases $n>2$, thus establishing the result.
\end{proof}

\begin{rem}
Recall from Notation \ref{nota:upsilon} that $\dupsilon _{(\pm; \mp)} 
$ is the composite 
$$
(\kk \dwb)_{\mp}
\stackrel{\X_\mp}{\rightarrow}
(\kk \ddb)_\mp 
\stackrel{\dmp_{(\pm; \mp)}} {\rightarrow} 
(\kk \db) _{(\pm; \mp)},
$$ 
where $\X_\mp$ is as in Notation \ref{nota:Xi} and  $\dmp_{(\pm; \mp)}$ in Notation \ref{nota:psi_phi}.
 Thus Theorem \ref{thm:cyclic_2_diopd_modules} implicitly transits by the respective structures 
\begin{enumerate}
\item 
$\dmp^*_{(+;+)} \sfb^* \cpd $ in $(\kk \ddb)$-modules; 
\item 
$\dmp^*_{(-;-)} \lfb^* \cpd $ in $(\kk \ddb)_-$-modules. 
\end{enumerate}
This does not correspond to new information: restriction along $\dmp_{(\pm; \mp)}$ simply provides the embedding of $(\kk \db)_{(\pm; \mp)}$-modules in $(\kk \ddb)_\mp$-modules (compare Corollary \ref{cor:full_subcat_kkdub_pmmp}).
\end{rem}

%%%%%%%%%%%%%%%%%%%%%%%%%%%%%%%%%%%%%%%%%%%%%%%%%%%%%%%%%%%%%%%%%%%%%%%
\subsection{Restricting to the associated operad}

For $\cpd$ a cyclic operad in $\nuco$, one has the natural inclusion in $\nudo$:
$$
\opd_{\amalg^*\cpd} \hookrightarrow \amalg^* \cpd, 
$$
where $\opd_\diopd$ is the underlying operad of a dioperad (without unit) $\diopd$, as in Proposition \ref{prop:restrict_diopd}.

\begin{cor}
\label{cor:compare_opd_cpd_amalg^*}
For $\cpd$  in $\nuco$, there are natural inclusions 
\begin{eqnarray*}
&&
\stfb^* (\opd_{\amalg^* \cpd}) \hookrightarrow 
\stfb^* (\amalg^* \cpd) \cong  \dupsilon_{(+;+)}^* \sfb^* (\cpd) 
\\
&&
\ltfb^* (\opd_{\amalg^* \cpd}) \hookrightarrow 
 \ltfb^* (\amalg^* \cpd)  \cong  \dupsilon_{(-;-)}^* \lfb^* (\cpd)
\end{eqnarray*}
of $\kk \dwb$-modules and $(\kk \dwb)_-$-modules respectively. 
\end{cor}

\begin{rem}
\label{rem:modules_dioperad_operad}
Apart from in degenerate cases, the dioperad $\amalg^* \cpd$ is not positive and the inclusion $\opd_{\amalg^* \cpd} \hookrightarrow \amalg^* \cpd$ of dioperads does not admit a retract (compare the case of Remark \ref{rem:opd_diopd_retract}). This means that the inclusions of Corollary \ref{cor:compare_opd_cpd_amalg^*} do not in general split in $\kk \dwb$-modules and $(\kk \dwb)_-$-modules respectively (although they do split in $\kk (\tfb)$-modules).

Unfortunately, this makes it harder to exploit homological results for $\stfb^* (\amalg^* \cpd)$ (respectively $ \ltfb^* (\amalg^* \cpd) $) to deduce homological results for $\stfb^* (\opd_{\amalg^* \cpd})$  (resp. $\ltfb^* (\opd_{\amalg^* \cpd})$).
\end{rem}

\part{Koszul complexes}
\label{part:kz_cx}

\section{Recollections on Koszul complexes}
\label{sect:kz_cx}

This section provides a quick review of the Koszul complexes that interest us, working with modules over $(\kk \dwb)_\mp$, $(\kk \ddb)_{\mp}$, and $(\kk \db)_{(\pm ; \mp)}$ respectively, using that these are all homogeneous quadratic $\kk$-linear categories that satisfy the requisite (right) projectivity condition. For further details, the reader is referred to \cite{P_cyclic} and \cite{P_opd_wall}. For generalities on homogeneous quadratic duality over a base ring, see \cite[Chapter 1]{MR4398644}.

\begin{rem}
Recall that $\pm$ and $\mp$ both denote designated signs (i.e., elements of $\{ -1, 1\}$). Thus $- \mp$ is interpreted as follows: the operator $-$ changes the value (and should not be replaced by $\pm$).
\end{rem}

%%%%%%%%%%%%%%%%%%%%%%%%%%%%%%%%%%%%%%%%%%%%%%%%%%%
\subsection{Input}

Suppose that $\cala$ is a $\kk$-linear category that is homogeneous quadratic over the full $\kk$-linear subcategory $\cala^0$ and that $\cala$ is projective as a right $\cala^0$-module.
In particular, $\cala$ comes with a canonical length grading over $\nat$. Then we can form the right quadratic dual of $\cala$; this is the homogeneous quadratic category $\cala^\perp$ over $\cala^0$ with generators given by the right dual $\hom_{(\cala^0) \op } (\cala^1, \cala^0)$ and quadratic relations given by the right dual of $\cala^2$. (One can recover $\cala$ by forming the {\em left} quadratic dual of $\cala^\perp$; these constructions correspond to inverse equivalences of categories.) See \cite[Chapter 1]{MR4398644} for further details  (where a weaker projectivity hypothesis is used).

\begin{prop}
\label{prop:identify_right_quadratic_duals}
\ 
\begin{enumerate}
\item 
The $\kk$-linear category  $(\kk \uwb)_\mp$ is a homogeneous quadratic category over $\kk (\tfb)$ and is free as a right $\kk (\tfb)$-module. Its right quadratic dual identifies as  $(\kk \dwb)_{-\mp}$.
\item 
The $\kk$-linear category $(\kk \dub)_\mp$ is a homogeneous quadratic category over $\kk \fb$ and is free as a right $\kk \fb$-module. Its right quadratic dual identifies as $(\kk \ddb)_{-\mp}$. 
\item 
The $\kk$-linear category $(\kk \ub)_{(\pm; \mp)}$ is a homogeneous quadratic category over $\kk \fb$ and is free as a right $\kk \fb$-module. Its right quadratic dual identifies as $(\kk \db)_{(\pm; -\mp)}$.
\end{enumerate}
\end{prop}

\begin{proof}
The first and third statements are established in \cite{P_opd_wall} and \cite{P_cyclic} respectively. The second statement is established similarly.
\end{proof}

\begin{rem}
\ 
\begin{enumerate}
\item 
The freeness property can be seen directly. For example, for $(\kk \uwb)_\mp$, this corresponds to the fact that $(\kk \uwb)_\mp ((U,V), (X,Y))$ is free as a right $\kk (\aut(U) \times \aut (V))$-module, which follows easily from the definitions. 
\item 
The change in the orientation sign associated to the order of pairs arises for the same reason that the quadratic dual of the symmetric algebra $S^* (V)$ on a finite-dimensional vector space $V$ is the exterior algebra $\Lambda^* (V)$: a `commutation relation' becomes an `anticommutation relation'.  
\item 
More is true: these $\kk$-linear categories are  {\em Koszul}. This fact is exploited in  \cite{P_opd_wall} and \cite{P_cyclic} to analyse the (co)homology of certain associated Koszul complexes. This will not be used here, since the results are based on the explicit form of the Koszul complexes.
\end{enumerate}
\end{rem}

Recall that we have explicit generators for the degree one morphisms (considered as right modules) based on the morphisms introduced in Notation \ref{nota:alpha_beta}: 
\begin{enumerate}
\item 
By Lemma \ref{lem:gen_deg_one_kkuwbmp}, for $m, n \in \nat$, 
$
(\kk \uwb)_\mp ((\m, \n), (\mathbf{m+1}, \mathbf{n+1}))
$ is generated freely as a $\kk ((\sym_m \times \sym_n)\op)$-module by the set $ \{ [\beta_{i,j}] \mid i \in \mathbf{m+1}, \ j \in \mathbf{n+1} \}$.
\item 
By Lemma \ref{lem:kkdub_deg1_generator}, 
for $n \in \nat$, 
$(\kk \dub)_\mp (\n, \mathbf{n+2})$ is generated freely as  a $\kk \sym_n\op$-module by the set $\{ [\alpha_{(i j)}] \}$, indexed by ordered pairs of elements $i \neq j \in \n$.
\item 
By the same result, $(\kk \ub)_{(\pm;\mp)}  (\n, \mathbf{n+2}) $ is generated freely  as  a $\kk \sym_n\op$-module  by the set $\{ [\alpha_{(i j)}] \}$, indexed by  elements $i < j \in \n$.
\end{enumerate}
Note that the degree one morphisms are independent of the sign $\mp$ associated to the ordering of pairs.

These bases yield  the `dual bases' for the respective quadratic duals. Under the identifications given in Proposition \ref{prop:identify_right_quadratic_duals}, this simply correspond to the `opposite' morphisms. For example, we have 
$$
[\beta_{i,j}\op] \in (\kk \dwb)_{-\mp} ( (\mathbf{m+1}, \mathbf{n+1}), (\m, \n)).
$$

\begin{rem}
For $\cala$ as above with right quadratic dual $\cala^\perp$, we have the following general form of a Koszul complex. Let $N$ be a right $\cala$-module and $M$ a left $\cala^\perp$-module. Then we can form:
$$
N \otimes _{\cala ^0} M
$$
using the restriction of the module structures to $\cala^0$. 

This is equipped with a `Koszul differential' that is induced by {\em inner multiplication} by an element $e \in \cala^1 \otimes_{\cala^ 0 } (\cala^\perp)^1$ that is given by the coevaluation corresponding to the duality between $\cala^1$ and $(\cala ^\perp)^1$.  This will be made  explicit below in the cases of interest.
\end{rem}

%%%%%%%%%%%%%%%%%%%%%%%%%%%%%%%%%%%%%%%%%%%%%%%%%%%%
\subsection{Koszul complexes in the walled case}
\label{subsect:recollect_Koszul_dwb}

For $G$  right $(\kk \uwb)_\mp$-module and $F$ a $(\kk \dwb)_{-\mp}$-module, we have the Koszul complex with underlying object
\begin{eqnarray}
\label{eqn:cx_G_F}
G \otimes_{\kk (\tfb)} F.
\end{eqnarray}
The differentials are induced by the inner multiplication by $\sum_{\substack{i \in \mathbf{m+1} \\ j \in \mathbf{n+1}}} [\beta_{i,j}] \otimes [\beta_{i,j}\op]$, for $m,n \in \nat$.
More precisely, this yields the components 
$$
G (\mathbf{m+1}, \mathbf{n+1}) \otimes _{\kk (\sym_{m+1} \times \sym_{n+1})} F (\mathbf{m+1}, \mathbf{n+1}) 
\rightarrow 
G (\mathbf{m}, \mathbf{n}) \otimes _{\kk (\sym_{m} \times \sym_{n})} F (\mathbf{m}, \mathbf{n}) 
$$ 
given by 
$$
g \otimes f \mapsto \sum_{i,j} g [\beta_{i,j}] \otimes [\beta_{i,j}\op] f.
$$

\begin{rem}
The context usually determines the appropriate grading to place upon such Koszul  complexes.
\end{rem}

\begin{exam}
\label{exam:walled_koszul_complexes}
We can consider $(\kk \uwb)_{\mp}$ as a bimodule, as usual. Using the right $(\kk \uwb)_{\mp}$-module structure, we may therefore form the Koszul complex 
\begin{eqnarray}
\label{eqn:uwb_F_kz1}
&&(\kk \uwb)_{\mp} \otimes_{\kk (\tfb)} F.
\end{eqnarray}
This takes values in the category of  $(\kk \uwb)_{\mp}$-modules. We refer to this as the {\em first Koszul complex} associated to $F$.

The complex (\ref{eqn:cx_G_F}) can be recovered from this as 
$$
G \otimes _{(\kk \uwb)_\mp} \Big( (\kk \uwb)_{\mp} \otimes_{\kk (\tfb)} F \Big).
$$ 

In particular, we can consider the case where $G = (\kk \uwb)_{\mp}^\sharp$, the $(\kk \uwb)_{\mp}$-bimodule obtained from $(\kk \uwb)_{\mp}$ by applying vector space duality. This yields the {\em second Koszul complex} associated to $F$:
\begin{eqnarray}
\label{eqn:uwb_F_kz2}
&&(\kk \uwb)_{\mp} ^\sharp\otimes_{\kk (\tfb)} F.
\end{eqnarray}
This takes values in the category of  $(\kk \uwb)_{\mp}$-modules.

The complexes (\ref{eqn:uwb_F_kz1}) and (\ref{eqn:uwb_F_kz2}) are described in detail in \cite{P_opd_wall}. In particular, the (co)homology of these complexes is identified in terms of appropriate $\ext$ and $\tor$ groups.
\end{exam}

%%%%%%%%%%%%%%%%%%%%%%%%%%%%%%%%%%%%%%%%%%%%%%%%%%%%%%%%%%%%%%%
\subsection{Koszul complexes in the directed case}
\label{subsect:Koszul_directed}

For $B$ a right $(\kk \dub)_\mp$-module and $A$ a left $(\kk \ddb)_{-\mp}$-module, we have the Koszul complex with underlying object
$$
B \otimes_{\kk \fb} A.
$$ 
The differentials are induced by inner multiplication with the elements 
$$
\sum_{i \neq j}  [\alpha_{(ij})] \otimes [\alpha_{(ij)}\op]
\in 
 (\kk \dub)_{\mp}(\mathbf{n}, \mathbf{n+2}) \otimes _{\kk\sym_{n}} (\kk \ddb)_{-\mp} (\mathbf{n+2}, \mathbf{n}),
$$
where the sum is over ordered pairs  $i \neq j \in \mathbf{n+2}$.
 Namely, this element induces the component of the differential
$$
B (\mathbf{n+2}) \otimes_{\kk \sym_{n+2}} A (\mathbf{n+2}) 
\rightarrow 
B (\mathbf{n}) \otimes_{\kk \sym_{n}} A (\mathbf{n}). 
$$

\begin{exam}
\label{exam:directed_koszul_complexes}
Similarly to the walled case (see Example \ref{exam:walled_koszul_complexes}), for $A$ a left $(\kk \ddb)_{-\mp}$-module,  we have the two Koszul complexes of $(\kk \dub)_\mp$-modules:
\begin{eqnarray}
\label{eqn:A_koz1}
&&
(\kk \dub)_\mp \otimes_{\kk \fb} A
\\
\label{eqn:A_koz2}
&&
(\kk \dub)_\mp^\sharp \otimes_{\kk \fb} A. 
\end{eqnarray}
\end{exam}

%%%%%%%%%%%%%%%%%%%%%%%%%%%%%%%%%%%%%%%%%%%%%%%%%%%%%%%%%%%%%%%%%%%%%%%%%%%%%
\subsection{Koszul complexes for the undirected (signed) case}
\label{subsect:recollect_Koszul_db}

For $N$ a right $(\kk \ub)_{(\pm; \mp)}$-module and $M$ a left $(\kk \db)_{(\pm; -\mp)}$-module, we have the Koszul complex with underlying object 
$$
N \otimes _{\kk \fb} M .
$$
In this case, the differentials are induced by inner multiplication with the elements 
$$
\sum_{i< j}
 [\alpha_{(ij})] \otimes [\alpha_{(ij)}\op]
\in 
 (\kk \ub)_{(\pm; \mp)}(\mathbf{n}, \mathbf{n+2}) \otimes _{\kk\sym_{n}} (\kk \db)_{(\pm;-\mp)} (\mathbf{n+2}, \mathbf{n})
$$
where the sum is over  $i < j \in \mathbf{n+2}$.  The action of this element is analogous to that in the directed case.

\begin{exam}
\label{exam:koszul_complexes_undirected}
We have the following analogue of Examples \ref{exam:walled_koszul_complexes} and \ref{exam:directed_koszul_complexes}.
For $M$ a left $(\kk \db)_{(\pm; -\mp)}$-module, we have the two Koszul complexes of $(\kk \ub)_{(\pm; \mp)}$-modules:
\begin{eqnarray}
\label{eqn:db_kz1}
&&
(\kk \ub)_{(\pm;\mp)} \otimes _{\kk \fb} M;
\\
\label{eqn:db_kz2}
&&
(\kk \ub)_{(\pm;\mp)}^\sharp \otimes _{\kk \fb} M.
\end{eqnarray}
These are described in detail in \cite{P_cyclic}, where an interpretation of their (co)homology is given in terms of $\ext$ and $\tor$ respectively.
\end{exam}

%%%%%%%%%%%%%%%%%%%%%%%%%%%%%%%%%%%%%%%%%%%%%%%%%%%%%%%%%%%%%%%%%%%%%%%%%
\subsection{Symmetry}
\label{subsect:symmetry}

There is an important underlying symmetry, which allows us to reduce the work required. This is best illustrated by the following example:

\begin{exam}
Consider the walled case, as in Section \ref{subsect:recollect_Koszul_dwb}. We could equivalently encode the modules $G$ and $F$ in the form $G \otimes F$, considered as a module over the  tensor product of $\kk$-linear categories:
$$
(\kk \dwb)_\mp\cten(\kk \dwb)_{-\mp}.
$$
This makes it clear that we can swap the roles of the $\kk$-linear categories. Namely, there is an isomorphic Koszul complex with underlying object 
$$
F \otimes _{\kk (\tfb)} G
$$
where now we are treating $F$ as a right $(\kk \uwb)_{-\mp}$-module and $G$ as a left $(\kk \dwb)_\mp$-module. The Koszul differential being defined by essentially the same element.
\end{exam}

\begin{rem}
Once we have chosen to focus upon the first and second Koszul complexes introduced above:
\begin{eqnarray*}
&&(\kk \uwb)_{\mp} \otimes_{\kk (\tfb)} F
\\
&&(\kk \uwb)_{\mp}^\sharp \otimes_{\kk (\tfb)} F,
\end{eqnarray*}
we are usually thinking of these as functors of $F$. We have thus `broken' the symmetry, since the second tensor factor does not play the same role.
\end{rem}

\section{Comparing Koszul complexes: walled, directed, and undirected}
\label{sect:new_compare}

The purpose of this section is to present general results that relate the various flavour of Koszul complexes that interest us, working with modules over twisted versions of walled Brauer categories, directed Brauer categories, and undirected Brauer categories.

%%%%%%%%%%%%%%%%%%%%%%%%%%%%%%%%%%%%%%%%%%%%%%%%%%%%%%%%%%%%%%%%%
\subsection{Passing between walled and directed}

Recall from Section \ref{subsect:Xi} that we have functors 
\begin{eqnarray*}
\X^*_\mp &:& (\kk \ddb)_{\mp} \dash \modules \rightarrow (\kk \dwb)_\mp \dash\modules 
\\
(\X_\mp)_* &:& (\kk \dwb)_\mp \dash\modules   \rightarrow (\kk \ddb)_{\mp} \dash \modules
\end{eqnarray*}
related by the adjunction $\X^*_\mp \dashv (\X_\mp)_*$. The restriction functor $\X^*_\mp$ is described directly in terms of the functor $\X_\mp$.  Its left adjoint $(\X_\mp)_*$ is described in Proposition \ref{prop:Xi_*}.

At the level of the underlying objects (corresponding to the forgetful functors $  (\kk \ddb)_{\mp} \dash \modules \rightarrow \kk \fb \dash \modules$ and $(\kk \dwb)_\mp \dash\modules   \rightarrow  \kk(\tfb)\dash\modules$ respectively), the adjunction $\X^*_\mp \dashv (\X_\mp)_*$ identifies with  
$$
\amalg^* \ : \  \kk \fb\dash\modules \rightleftarrows \kk (\tfb)\dash\modules \ : \  \amalg_*.
$$

Consider $G$ a right $(\kk \uwb)_\mp$-module and $M$ a $(\kk \ddb)_{-\mp}$-module, so that we have $(\X_\mp)_*G$, a  right $(\kk \dub)_\mp$-module and $\X^*_{-\mp} M$, a $(\kk \dwb)_{-\mp}$-module. We thus have the two Koszul complexes 
\begin{eqnarray}
\label{eqn:G_M_walled}
G \otimes_{\kk (\tfb)} \X^*_{-\mp} M
\\
\label{eqn:G_M_directed}
(\X_\mp)_*G \otimes_{\kk \fb} M,
\end{eqnarray}
with differentials induced by inner multiplication by the elements  $\sum  [\beta_{i,j}] \otimes [\beta_{i,j}\op]$ and $\sum [\alpha_{(ij})] \otimes [\alpha_{(ij)}\op]$ respectively.

\begin{exam}
\label{exam:Phi_action_beta}
Consider $(\m, \n)$ as an object of $(\kk \dwb)_{-\mp}$. Then $\X^*_{-\mp} M (\m, \n)$ identifies as $M (\m \amalg \n)$, by definition of $\X^*_{-\mp}$. Moreover, for $i \in \m$ and $j \in \n$, the action of $[\beta_{i,j}\op]$ on $\X^*_{-\mp} M (\m, \n)$ is given by the action of $[\alpha_{ij'}\op]$ on $M (\mathbf{m+n})$, by Example \ref{exam:uwbord_2_fiordev_alpha_beta}, where we identify $\m \amalg \n$ with $\mathbf{m+n}$ as usual and $j'=m+j$. 
\end{exam}

The underlying objects of these complexes only depend on the underlying $\kk (\tfb)$-modules and $\kk \fb$-modules respectively, so that these  can be written  as
$
G \otimes_{\kk (\tfb)} \amalg^* $ and 
 $
\amalg_*G \otimes_{\kk \fb} M$ respectively.
 These are naturally isomorphic, by Frobenius reciprocity.

The differentials can be described  more explicitly, as follows. For $m, n \in \nat$, the component of the differential of (\ref{eqn:G_M_walled}) 
$$
G (\m , \n) \otimes_{\kk (\sym_m \times \sym_n )}  M (\m \amalg \n)
\rightarrow 
G (\mathbf{m-1} , \mathbf{n-1}) \otimes_{\kk (\sym_{m-1} \times \sym_{n-1} )}  M ((\mathbf{m-1}) \amalg (\mathbf{n-1}))
$$
is determined on an element of the form $x \otimes y$ (for $x \in G (\m , \n)$ and $y \in  M (\m \amalg \n)$) by 
\begin{eqnarray}
\label{eqn:diff_beta}
x \otimes y \mapsto \sum x [\beta_{i,l}] \otimes [\beta_{i,l}\op] y.
\end{eqnarray}

Similarly, for $t \in \nat$, the component of the differential of (\ref{eqn:G_M_directed})
\begin{eqnarray}
\label{eqn:X*G_M_diff}
((\X_\mp)_*G) (\mathbf{t}) \otimes_{\kk \sym_t} M (\mathbf{t}) 
\rightarrow 
((\X_\mp)_*G) (\mathbf{t-2}) \otimes_{\kk \sym_{t-2}} M (\mathbf{t-2}) 
\end{eqnarray}
is determined on $u \otimes v$ (for $u \in ((\X_\mp)_*G) (\mathbf{t})$ and $v \in M (\mathbf{t}) $) by 
\begin{eqnarray}
\label{eqn:diff_alpha}
u \otimes v 
\mapsto 
\sum u [\alpha_{(ij)}] \otimes [\alpha_{(ij)}\op] v.
\end{eqnarray}

\begin{thm}
\label{thm:Xi_Koszul_comparison}
For $G$ a right $(\kk \uwb)_\mp$-module  and $M$ a $(\kk \ddb)_{-\mp}$-module, there is a natural isomorphism of Koszul complexes
$$
G \otimes_{\kk (\tfb)} \X^*_{-\mp} M
\cong 
(\X_\mp)_*G \otimes_{\kk \fb} M.
$$ 
\end{thm}

\begin{proof}
As explained above, by Frobenius reciprocity, the underlying objects of these complexes are isomorphic. 
Now, by Proposition \ref{prop:Xi_*}, for given $t \in \nat$, we have
$$
(
(\X_\mp)_*G 
) (\mathbf{t}) 
= 
\bigoplus_{m + n = t} 
G(\m, \n) \uparrow_{\sym_m \times \sym_n}^{\sym_t}.
$$
and the differential (\ref{eqn:X*G_M_diff}) splits as a direct sum on the components indexed by $m,n$:
$$
G(\m, \n) \uparrow_{\sym_m \times \sym_n}^{\sym_t}
 \otimes_{\kk \sym_t} M (\mathbf{t}) 
\rightarrow 
G(\mathbf{m-1}, \mathbf{n-1}) \uparrow_{\sym_{m-1} \times \sym_{n-1}}^{\sym_{t-2}}
\otimes_{\kk \sym_{t-2}} M (\mathbf{t-2}) 
$$
and thus splits in the same way  as the differential of (\ref{eqn:G_M_walled}). 

It suffices to show that these components (for the respective complexes (\ref{eqn:G_M_walled}) and (\ref{eqn:G_M_directed}))  coincide.
To show this, consider the behaviour of the differential (\ref{eqn:X*G_M_diff}) on an element represented by $ x\otimes y$ in $G(\m, \n) 
 \otimes M (\mathbf{t})$.  By Proposition \ref{prop:Xi_*}, $x[\alpha_{(ij')}]$ is zero unless $i \in \mathbf{m} \subset \mathbf{t}$ and $j' \in \mathbf{n}\subset \mathbf{t}$ (the notation $j'$ is used to harmonize with Example \ref{exam:Phi_action_beta}). After reindexing, replacing $j'$ by $j:= j'-m$, in the latter case we have $x [\alpha_{(ij')}] = x [\beta_{i ,j}]$ (Cf. Example \ref{exam:uwbord_2_fiordev_alpha_beta}).
 
Combining this with the behaviour of $[\beta_{i,l}\op]$ exhibited in Example, \ref{exam:Phi_action_beta}, we obtain the equality:
$$
\sum x [\alpha_{(ij')}] \otimes [\alpha_{(ij')}\op] y 
= 
\sum x [\beta_{(i,l)}] \otimes [\beta_{(i,l)}\op] y,
$$
 over the appropriate indexing sets. This concludes the proof.
\end{proof}

\begin{rem}
\label{rem:thm-Xi_Koszul_comparison_switch}
By the `symmetry'  observed in Section \ref{subsect:symmetry}, there is a `symmetric' version of this result, in which $ \X^*_{\mp}$ appears on the left of the tensor product and $(\X_{-\mp})_*$ on the right. Namely, for $G$ a $(\kk \dwb)_\mp$-module  and $M$ a $(\kk \ddb)_{-\mp}$-module, we have the isomorphic  Koszul complexes
\begin{eqnarray*}
 \X^*_{-\mp} M \otimes_{\kk (\tfb)}  G
&\cong &
M \otimes_{\kk \fb} (\X_\mp)_*G.
\end{eqnarray*}
\end{rem}

%%%%%%%%%%%%%%%%%%%%%%%%%%%%%%%%%%%%%%%%%%%%%%%%%%%%%%%%%%%%%%%%%
\subsection{Passing between directed and undirected}

For given $(\pm; \mp)$, we have the functor 
$
\dmp_{(\pm;\mp)} : \kk \ddb _\mp \rightarrow \kk \db_{(\pm;\mp)}
$
 (see Notation \ref{nota:psi_phi}). This induces the restriction functor
$$
\dmp_{(\pm;\mp)}^* : \kk \db_{(\pm;\mp)}\dash\modules \rightarrow \kk \ddb _\mp\dash \modules.
$$
On the underlying objects (in $\kk \fb$), this is the identity functor.

Recall from Section \ref{subsect:sgntr^*_downward} that we also have the functor 
$$
\sgntr_{(\pm; \mp)} ^*
: 
(\kk \ddb)_\mp\dash\modules 
\rightarrow 
(\kk \db)_{(\pm; \mp)} \dash \modules.
$$
Again, on the underlying objects, this is the identity.

Consider $M$ a right $(\kk \dub)_{-\mp}$-module and $B$ a $(\kk \db)_{(\pm; \mp)}$-module. We can thus form $\sgntr_{(\pm; -\mp)} ^*M$, a right $(\kk \ub)_{(\pm; -\mp)}$-module, and $\dmp^* _{(\pm; \mp)} B$, a $(\kk \ddb)_\mp$-module. We can then consider the Koszul complexes:
\begin{eqnarray}
\label{eqn:cx_M_phi^*G}
M \otimes_{\kk \fb} \dmp^* _{(\pm; \mp)} B
\\
\label{eqn:cx_sgntr_^*M_G}
\sgntr_{(\pm; -\mp)} ^*M \otimes_{\kk \fb} B,
\end{eqnarray}
where the first is formed in the directed setting and the second in the undirected setting.

\begin{thm}
\label{thm:phi^*_sngtr}
For $M$ a $(\kk \ddb_{-\mp})$-module and $B$ a $(\kk \db)_{(\pm; \mp)}$-module, there is a natural isomorphism of Koszul complexes
$$
M \otimes_{\kk \fb} \dmp^* _{(\pm; \mp)} B
\cong 
\sgntr_{(\pm; -\mp)} ^*M \otimes_{\kk \fb} B.
$$
\end{thm}

\begin{proof}
The strategy of the  proof is similar to that of Theorem \ref{thm:Xi_Koszul_comparison}.

It is clear that the underlying objects are naturally isomorphic, since the functors $\sgntr_{(\pm; \mp)} ^*$ and $\dmp_{(\pm;\mp)}^* $ are the identity at the level of  the underlying $\kk \fb$-modules. Thus it remains to check that the differentials correspond.  

In the complex (\ref{eqn:cx_M_phi^*G}) the differential is given by the inner multiplication with the element 
$$
\sum_{i \neq j} [\overrightarrow{\alpha}_{(ij)}] \otimes [\overrightarrow{\alpha}_{(ij)}\op],
$$ 
understood  in the directed context (as stressed by the decoration $\overrightarrow{}$). 

In the complex  (\ref{eqn:cx_sgntr_^*M_G}), the differential is given by inner multiplication by the element 
$$
\sum_{i < j} [\alpha_{(ij)}] \otimes [\alpha_{(ij)}\op],
$$ 
understood in the undirected context.

Consider sections  $u \in M$ and $v \in B$ and the element in the underlying object $M \otimes _{\kk \fb} B$ represented by  $u \otimes v$. The corresponding sections in $\sgntr_{(\pm; -\mp)} ^*M$ and in $\dmp^* _{(\pm; \mp)} B$ will also be denoted by $u$ and $v$, since the underlying objects are the same. We proceed to compare the respective differentials. 

By the definition of  $\dmp^* _{(\pm; \mp)}$, we have the following equality (for $i \neq j$)
$$
u [\overrightarrow{\alpha}_{(ij)}] \otimes [\overrightarrow{\alpha}_{(ij)}\op] v
= 
\left\{ 
\begin{array}{ll}
u [\overrightarrow{\alpha}_{(ij)}] \otimes [\alpha_{(ij)}\op] v & \mbox{if $i<j$}
\\
\pm u [\overrightarrow{\alpha}_{(ij)}] \otimes [\alpha_{(ji)}\op] v & \mbox{if $j<i$}.
\end{array}
\right.
$$
Similarly, by the definition of $\sgntr_{(\pm; -\mp)} ^*$, we have (for $i<j$)
$$
u[\alpha_{(ij)}] \otimes [\alpha_{(ij)}\op] v
= 
u \big( [\overrightarrow{\alpha}_{(ij)}]  + \pm [\overrightarrow{\alpha}_{(ji)}] \big) 
\otimes 
[\alpha_{(ij)}\op] v.
$$

Summing over the respective indexing sets, it is clear that the differentials identify, as required.
\end{proof}

\begin{rem}
As in Remark \ref{rem:thm-Xi_Koszul_comparison_switch}, there is a `symmetric' version of Theorem \ref{thm:phi^*_sngtr}, in which $\dmp^* _{(\pm; -\mp)} $ appears on the left of the tensor product and $\sgntr_{(\pm; -\mp)} ^*$ on the right.
\end{rem}

%%%%%%%%%%%%%%%%%%%%%%%%%%%%%%%%%%%%%%%%%%%%%%%%%%%%%%%%%%%%%%%%%%%%%%%ù
\subsection{Consequences}

Recall from Definition \ref{defn:I} that  $\unwall_{(\pm;\mp)}$ is the composite functor 
$$
\xymatrix{
(\kk \dwb)_\mp\dash \modules 
\ar[r]^{(\X_\mp )_*}
&
(\kk \ddb)_\mp\dash \modules 
\ar[r]^(.45){\sgntr_{(\pm; \mp)} ^*}
&
(\kk \db)_{(\pm; \mp)}\dash \modules.
}
$$

Similarly, we have the restriction functor  $\dupsilon _{(\pm; \mp)}^*$(where  $\dupsilon _{(\pm; \mp)}$ is as in Notation \ref{nota:upsilon}) 
and this identifies as the composite
$$
(\kk \db) _{(\pm; \mp)}\dash\modules 
\stackrel{\dmp_{(\pm; \mp)}^*} {\rightarrow} 
(\kk \ddb)_\mp 
\stackrel{\X_\mp^*}{\rightarrow}
(\kk \dwb)_{\mp}\dash\modules .
$$ 

Putting Theorems \ref{thm:Xi_Koszul_comparison} and \ref{thm:phi^*_sngtr} together, we have the following:

\begin{cor}
\label{cor:unwall_vs_dupsilon}
For $G$ a $(\kk \dwb)_{\mp}$-module and  $B$ a $(\kk \db)_{(\pm; -\mp)}$-module, there is a natural isomorphism of  Koszul complexes
\begin{eqnarray*}
\unwall_{(\pm;\mp)} G  \otimes_{\kk \fb} B
&\cong &   G \otimes_{\kk (\tfb)} \dupsilon _{(\pm; -\mp)}^* B.
\end{eqnarray*}
\end{cor}

\begin{exam}
\label{exam:first_second_kz_cx_unwall_dupsilon}
Corollary \ref{cor:unwall_vs_dupsilon} can, in particular, be applied to the first and second Koszul complexes (as presented in Section \ref{sect:kz_cx}). 
\begin{enumerate}
\item 
For $G$ a $(\kk \dwb)_{\mp}$-module, using the symmetric form of Corollary \ref{cor:unwall_vs_dupsilon}, there are natural isomorphisms of complexes of $(\kk \ub)_{(\pm; -\mp)}$-modules 
\begin{eqnarray*}
(\kk \ub)_{(\pm; -\mp)} \otimes_{\kk \fb} \unwall_{(\pm;\mp)} G 
&
\cong 
&
\dupsilon _{(\pm; -\mp)}^*  (\kk \ub)_{(\pm; -\mp)} \otimes_{\kk (\tfb)} G 
\\
(\kk \ub)_{(\pm; -\mp)}^\sharp \otimes_{\kk \fb} \unwall_{(\pm;\mp)} G 
&
\cong 
&
\dupsilon _{(\pm; -\mp)}^*  \big( (\kk \ub)_{(\pm; -\mp)}^\sharp\big) \otimes_{\kk (\tfb)} G. 
\end{eqnarray*}
Thus, to understand the first and second Koszul complexes associated to the $(\kk \db)_{(\pm; \mp)}$-module $\unwall_{(\pm;\mp)} G$ in terms of Koszul complexes for $G$, we need to understand 
$\dupsilon _{(\pm; -\mp)}^*  (\kk \ub)_{(\pm; -\mp)}$ and $\dupsilon _{(\pm; -\mp)}^*  \big( (\kk \ub)_{(\pm; -\mp)}^\sharp\big)$.
\item 
For $B$ a $(\kk \db)_{(\pm; -\mp)}$-module, there are natural isomorphisms of complexes of $(\kk \uwb)_\mp$-modules:
\begin{eqnarray*}
(\kk \uwb)_\mp \otimes_{\kk (\tfb)} \dupsilon _{(\pm; -\mp)}^* B 
&\cong & 
\unwall_{(\pm;\mp)} (\kk \uwb)_\mp \otimes_{\kk \fb} B
\\
(\kk \uwb)_\mp^\sharp \otimes_{\kk (\tfb)} \dupsilon _{(\pm; -\mp)}^* B 
&\cong & 
\unwall_{(\pm;\mp)} \big( (\kk \uwb)_\mp^\sharp\big) \otimes_{\kk \fb} B.
\end{eqnarray*}
In this case, to understand the first and second Koszul complexes associated to the $(\kk \dwb)_{-\mp}$-module $\dupsilon _{(\pm; -\mp)}^* B $, we need to understand $\unwall_{(\pm;\mp)} (\kk \uwb)_\mp$ and $\unwall_{(\pm;\mp)} \big( (\kk \uwb)_\mp^\sharp\big)$.
\end{enumerate}
\end{exam}

\section{Comparing Koszul complexes using $\unwall_{(\pm;\mp)}$}
\label{sect:compare_koszul}

For $F$ a $(\kk \dwb)_\mp$-module, on applying the functor $\unwall_{(\pm;\mp)}$, we obtain the $(\kk \ub)_{(\pm;\mp)}$-module $\unwall_{(\pm;\mp)} F$. We thus have the first and second Koszul complexes in both the walled and the unwalled contexts. The purpose of this section is to apply the isomorphisms provided by Corollary \ref{cor:unwall_vs_dupsilon}, together with the calculations of 
Part \ref{part:brauer}, to relate these.

%%%%%%%%%%%%%%%%%%%%%%%%%%%%%%%%%%%%%%%%%%%%%%%%%%%%%%
\subsection{Comparing the first Koszul complexes}
\label{subsect:first_Koszul_unwall}

Recall from Proposition \ref{prop:unwall_functor} the functor 
$$
\unwall_{(\pm;\mp)} : (\kk \dwb)_\mp \dash \modules 
\rightarrow
(\kk \db)_{(\pm; \mp)}\dash \modules.
$$
Thus, for $F$ a $(\kk \dwb)_\mp$-module, we can form the two Koszul complexes
\begin{eqnarray*}
 &&(\kk \uwb)_{-\mp} \otimes_{\kk (\tfb)} F
\\
&& (\kk \ub)_{(\pm;-\mp)} \otimes_{\kk \fb} \unwall_{(\pm;\mp)} F;
\end{eqnarray*}
the first is a complex of $(\kk \uwb)_{-\mp}$-modules and the second of $(\kk \ub)_{(\pm;-\mp)}$-modules.

The following result shows that the second complex can be expressed in terms of the first, by using the functor 
$
\lad _{( \pm; -\mp)} : (\kk \uwb)_{-\mp} \rightarrow (\kk \ub)_{(\pm;-\mp)}
$
introduced in  Section \ref{subsect:lad}.

\begin{thm}
\label{thm:unwall_first_complex}
For $F$  a $(\kk \dwb)_\mp$-module, there is a natural isomorphism of complexes of $\kk \ub_{(\pm;-\mp)}$-modules:
\[
(\kk \ub)_{(\pm;-\mp)} \otimes_{\kk \fb} \unwall_{(\pm;\mp)} F
\cong
\lad_{(\pm;-\mp)} \big( (\kk \uwb)_{-\mp} \otimes_{\kk(\tfb)} F \big) .
\]
\end{thm}

\begin{proof}
By (the symmetric form of) Corollary \ref{cor:unwall_vs_dupsilon}, we have the natural isomorphism 
$$
(\kk \ub)_{(\pm;-\mp)} \otimes_{\kk \fb} \unwall_{(\pm;\mp)} F
\cong 
\dupsilon_{(\pm; -\mp)}^* (\kk \ub)_{(\pm; -\mp)} \otimes_{\kk (\tfb)} F
$$
of $(\kk \ub)_{(\pm;-\mp)} $-modules, since the functor $\dupsilon_{(\pm; -\mp)}^*$ is applied with respect to the contravariant variable of $(\kk \ub)_{(\pm;-\mp)}$ (considered as a bimodule) and thus is compatible with the $(\kk \ub)_{(\pm;-\mp)} $-module structure.

Proposition \ref{prop:duspilon_lad} provides the natural  isomorphism 
$$
\dupsilon_{(\pm; -\mp)}^* 
(\kk \ub)_{(\pm; -\mp)}
\cong 
\lad_{(\pm; -\mp)} (\kk \uwb)_{-\mp} 
$$
in 
$(\kk \dwb)_{-\mp} \cten  (\kk \ub)_{(\pm; -\mp)} $-modules. Here the functor $\lad_{(\pm; -\mp)} $ is applied with respect to the covariant variable. 

Putting these isomorphisms together, one obtains the isomorphism of the statement, noting that $\lad_{(\pm; -\mp)}$ can be considered as being applied to the complex $(\kk \uwb)_{-\mp} \otimes_{\kk(\tfb)} F$.
\end{proof}

\begin{rem}
This result states that  the first Koszul complex of $\unwall_{(\pm;\mp)} F$ is isomorphic to the complex obtained by applying the functor $\lad_{(\pm; -\mp)}$  to the first Koszul complex of $F$.
\end{rem}

Theorem \ref{thm:unwall_first_complex} yields the following result relating the homologies of the respective first Koszul complexes:

\begin{cor}
\label{cor:compare_homology_lad}
For $F$  a $(\kk \dwb)_\mp$-module, there is a natural morphism of graded  $\kk \ub_{(\pm;-\mp)}$-modules:
$$
\lad_{(\pm;-\mp)} H_* ( (\kk \uwb)_{-\mp} \otimes_{\kk(\tfb)} F )
\rightarrow  
H_* (
(\kk \ub)_{(\pm;-\mp)} \otimes_{\kk \fb} \unwall_{(\pm;\mp)} F).
$$
\end{cor}

\begin{proof}
The functor $\lad _{(\pm;-\mp)}$ is right exact. Hence it induces a natural morphism 
$$
\lad_{(\pm;-\mp)} H_* ( (\kk \uwb)_{-\mp} \otimes_{\kk(\tfb)} F )
\rightarrow 
H_* \big( \lad_{(\pm;-\mp)} ( (\kk \uwb)_{-\mp} \otimes_{\kk(\tfb)} F )\big).
$$ 
The result then follows by composing with the isomorphism in homology given by Theorem \ref{thm:unwall_first_complex}.
\end{proof}

\begin{exam}
Consider a  $(\kk \db)_{(-;-)}$-module $G$ and its first Koszul complex $(\kk \ub)_{(-;+)} \otimes_{\kk \fb} G$. Then, for $V$ a symplectic $\kk$-vector space, we have the complex
$$
V^{\otimes \bullet}  \otimes_{\kk \fb} G
\cong
V^{\otimes \bullet} \otimes_{(\kk \ub)_{(-;+)} } (\kk \ub)_{(-;+)} \otimes_{\kk \fb} G
,
$$
equipped with the Koszul differential. This construction is natural with respect to $V$, in the appropriate sense (see \cite{P_cyclic}). The left hand side above identifies as the Schur functor $G(V)$ evaluated on $V$ (forgetting the symplectic structure). The `polynomial degree' provides a graduation and the differential of the Koszul complex uses the symplectic form. 

The above applies when $G$ is $\unwall_{(-;-)} F$ for a $(\kk \dwb)_-$-module $F$, giving the associated complex 
$
V^{\otimes \bullet}  \otimes_{\kk \fb} \unwall_{(-;-)} F.
$
Corollary \ref{cor:unwall_vs_dupsilon} yields the isomorphism of Koszul complexes
$$
V^{\otimes \bullet}  \otimes_{\kk \fb} \unwall_{(-;-)} F
\cong 
\dupsilon_{(-;+)}^* (V^{\otimes \bullet}) \otimes _{\kk (\tfb)} F.
$$

We can also  form the Koszul complex 
$$
T^{\bullet, \bullet} (W) \otimes _{\kk (\tfb)} F
$$
using the mixed tensors $T^{\bullet, \bullet} (W)$ (which sends $(\m , \n)$ to $(W^\sharp)^{\otimes m} \otimes W^{\otimes n}$), which have a natural $\kk \dwb$-module structure.

We can compare this with the complex $\dupsilon_{(-;+)}^* (V^{\otimes \bullet}) \otimes _{\kk (\tfb)} F$, taking $W$ to be the underlying $\kk$-vector space of $V$. At the level of the underlying objects the comparison uses Frobenius reciprocity. A key point here is that, by using a symplectic vector space $V$, we do not require to used {\em mixed} tensors to form our Koszul complex.
\end{exam}

%%%%%%%%%%%%%%%%%%%%%%%%%%%%%%%%%%%%%%%%%%%%%%%%%%%%%%
\subsection{Comparing the second Koszul complexes}
\label{subsect:second_Koszul_unwall}

Similarly to the case of the first Koszul complexes, we have for $F$ a $(\kk \dwb)_\mp$-module, the two Koszul complexes
\begin{eqnarray*}
 (\kk\uwb)_{- \mp}^\sharp  \otimes_{\tfb} F
\\
(\kk \ub)_{(\pm;-\mp)}^\sharp \otimes_\fb \unwall_{(\pm;\mp)} F;
\end{eqnarray*}
the first is a complex of $(\kk \uwb)_{-\mp}$-modules and the second of $(\kk \ub)_{(\pm;-\mp)}$-modules.

In the following statement, the $\kk \fb\op \cten  \kk (\tfb)$-module $\big( \amalg_* \op (\kk \ub)_{(\pm; \mp)}^{\cten  2}\big)^\sharp$ is as in Corollary \ref{cor:dupsilon_kkub^sharp}.

\begin{thm}
\label{thm:unwall_second_complex}
For $F$  a $\kk \dwb_\mp$-module, there is a natural isomorphism of complexes of $\kk \fb$-modules:
\[
(\kk \ub)_{(\pm;-\mp)}^\sharp \otimes_{\kk \fb} \unwall_{(\pm;\mp)} F
\cong 
\big( \amalg_* \op (\kk \ub)_{(\pm; \mp)}^{\cten  2}\big)^\sharp
\otimes_{\kk (\tfb)}
\big( (\kk\uwb)_{-\mp}^\sharp  \otimes_{\kk(\tfb)} F \big) 
\]
\end{thm}

\begin{proof}
The proof is similar to that of Theorem \ref{thm:unwall_first_complex}, by combining Corollary \ref{cor:unwall_vs_dupsilon} with 
 Corollary \ref{cor:dupsilon_kkub^sharp}.
\end{proof}

\begin{rem}
Theorem \ref{thm:unwall_second_complex} relates the second Koszul complex of $\unwall_{(\pm;\mp)} F$ to the second Koszul complex of $F$. 
Note that the statement of the theorem is weaker than that of Theorem \ref{thm:unwall_first_complex} in that it only concerns the underlying complex of $\kk \fb$-modules, rather than the full $(\kk \ub)_{(\pm;-\mp)}$-module structure. 

This is not completely anodyne, since the full module structure is important when comparing the behaviour of the first and second Koszul complexes, as in \cite{P_cyclic} and \cite{P_opd_wall}.
\end{rem}

Since the functor $\big( \amalg_* \op (\kk \ub)_{(\pm; \mp)}^{\cten  2}\big)^\sharp
\otimes_{\kk (\tfb)}-$  is exact, one has the immediate consequence:

\begin{cor}
\label{cor:hom_Kz_cx_unwall}
For $F$  a $(\kk \dwb)_\mp$-module, there is a natural isomorphism of graded $\kk \fb$-modules:
$$
H_* ((\kk \ub)_{(\pm;-\mp)}^\sharp \otimes_{\kk \fb} \unwall_{(\pm;\mp)} F)
\cong 
\big( \amalg_* \op (\kk \ub)_{(\pm; \mp)}^{\cten  2}\big)^\sharp
\otimes_{\kk (\tfb)}
 H_* ((\kk\uwb)_{-\mp}^\sharp \otimes_{\kk(\tfb)} F).
$$
\end{cor}

As an example application, we have:

\begin{cor}
\label{cor:direct_comparison_kz_1}
Let  $F$ be  a $(\kk \dwb)_\mp$-module. Suppose that the graded $\kk (\tfb)$-module $ H_* ((\kk\uwb)_{-\mp}^\sharp \otimes_{\kk(\tfb)} F)$ is supported on $(\zero, \zero)$ (i.e., vanishes on $(\m, \n)$ if $m+n >0$). Then there is an isomorphism of graded $\kk$-vector spaces:
$$
H_* ((\kk \ub)_{(\pm;-\mp)}^\sharp \otimes_{\kk \fb} \unwall_{(\pm;\mp)} F) (\zero)
\cong 
H_* ((\kk\uwb)_{-\mp}^\sharp \otimes_{\kk(\tfb)} F)(\zero, \zero)
$$
and $H_* ((\kk \ub)_{(\pm;-\mp)}^\sharp \otimes_{\kk \fb} \unwall_{(\pm;\mp)} F) (\n)=0$ for $n>0$
\end{cor}

\begin{proof}
This follows from Corollary \ref{cor:hom_Kz_cx_unwall} by using the fact that $\big( \amalg_* \op (\kk \ub)_{(\pm; \mp)}^{\cten  2}\big)^\sharp ((\zero, \zero), \n)$ is zero unless $n=0$ when it takes the value $\kk$.
\end{proof}

\begin{rem}
\ 
\begin{enumerate}
\item 
The hypothesis of Corollary \ref{cor:direct_comparison_kz_1} is clearly very strong. It is however satisfied by examples of interest (see Corollary \ref{cor:comparison_opd_unit_case}).
\item 
In general, all the terms $H_* ((\kk\uwb)_{-\mp}^\sharp \otimes_{\kk(\tfb)} F)(\n,\n)$, for $n \in \nat$, can contribute to $
H_* ((\kk \ub)_{(\pm;-\mp)}^\sharp \otimes_{\kk \fb} \unwall_{(\pm;\mp)} F) (\zero)$, via the expression of Corollary \ref{cor:hom_Kz_cx_unwall}.
\end{enumerate}
\end{rem}

There is a variant of Theorem \ref{thm:unwall_second_complex} that is weaker in that it does not provide an isomorphism of complexes,  but is stronger in the sense that it respects the $(\kk \ub)_{(\pm; -\mp)}$-module structure. 

\begin{prop}
\label{prop:inclusion_second_kz_cx}
For $F$  a $\kk \dwb_\mp$-module, there is a natural inclusion of complexes of $(\kk \ub)_{(\pm; -\mp)}$-modules:
$$
\junwall_{(\pm; -\mp)}\big( (\kk \uwb)_{-\mp}^\sharp \otimes_{\kk (\tfb)} F  \big)
\hookrightarrow 
(\kk \ub)_{(\pm;-\mp)}^\sharp \otimes_{\kk \fb} \unwall_{(\pm;\mp)} F.
$$
\end{prop}

\begin{proof}
Corollary \ref{cor:inject_junwall_dupsilon} provides the injection 
 of $ (\kk \dwb)_{-\mp} \cten   (\kk \ub)_{(\pm; -\mp)}$-modules:
 $$
\junwall_{(\pm; -\mp)}\big( (\kk \uwb)_{-\mp}^\sharp\big) 
\hookrightarrow 
 \dupsilon_{(\pm; -\mp)}^*\big( (\kk \ub)_{(\pm; -\mp)} ^\sharp \big). 
$$
This induces an injection of Koszul complexes 
$$
\junwall_{(\pm; -\mp)}\big( (\kk \uwb)_{-\mp}^\sharp \otimes_{\kk (\tfb)} F  \big)
\hookrightarrow 
\dupsilon_{(\pm; -\mp)}^*\big( (\kk \ub)_{(\pm; -\mp)} ^\sharp \big) \otimes_{\kk (\tfb)} F.
$$
By Corollary \ref{cor:unwall_vs_dupsilon}, the codomain is isomorphic to $(\kk \ub)_{(\pm;-\mp)}^\sharp \otimes_{\kk \fb} \unwall_{(\pm;\mp)} F$. The result follows.
\end{proof}

\section{Comparing Koszul complexes using restriction}
\label{sect:compare_kos_restrict}

In Section \ref{sect:compare_koszul}, we used the functor 
$
\unwall_{(\pm;\mp)} : (\kk \dwb)_\mp\dash \modules 
\rightarrow
(\kk \db)_{(\pm; \mp)}\dash \modules
$
to compare Koszul complexes associated to a $(\kk \dwb)_\mp$-module.
 In this section, we go in the other direction, replacing $\unwall_{(\pm;\mp)}$  by the restriction functor along 
the $\kk$-linear functor
 $
\dupsilon _{(\pm;\mp)} : 
(\kk \dwb) _{\mp} 
\rightarrow 
(\kk \db)_{(\pm;\mp)}
$ 
of Notation \ref{nota:upsilon}.

Our main aim is to compare the following respective first (resp. second) Koszul complexes for $G$ a $(\kk \db)_{(\pm;\mp)}$-module:
\begin{enumerate}
\item 
the complexes of $(\kk \uwb)_{-\mp}$-modules:
\begin{eqnarray*}
\label{eqn:upsilon_G_koz1}
&&
(\kk \uwb)_{-\mp} \otimes_{\kk (\tfb)} \dupsilon _{(\pm; \mp)}^* G 
\\
\label{eqn:upsilon_G_koz2}
&&
(\kk\uwb)_{-\mp}^\sharp  \otimes_{\kk (\tfb)} \dupsilon _{(\pm; \mp)}^* G;
\end{eqnarray*}
\item 
the complexes of $(\kk \ub)_{(\pm;-\mp)}$-modules:
\begin{eqnarray*}
\label{eqn:G_koz1}
&&
(\kk \ub)_{(\pm;-\mp)} \otimes_{\kk \fb} G 
\\
&&
\label{eqn:G_koz2}
(\kk \ub)_{(\pm;-\mp)}^\sharp \otimes_{\kk \fb} G.
\end{eqnarray*}
\end{enumerate}
For the comparison, we  exploit Corollary \ref{cor:unwall_vs_dupsilon}.

We also consider the intermediate steps stemming from the fact  the functor $\dupsilon _{(\pm; \mp)}^*$ factorizes across the two restriction functors:
$$
(\kk \db)_{(\pm;\mp)}\dash \modules 
\stackrel{\dmp_{(\pm;\mp)}^*}\longrightarrow 
(\kk \ddb)_{\mp}\dash \modules
\stackrel{\X_\mp^*}{\longrightarrow} 
(\kk \dwb) _{\mp} \dash \modules,
$$
where $\dmp_{(\pm;\mp)} : \kk \ddb _\mp \rightarrow \kk \db_{(\pm;\mp)}$ is as in Notation \ref{nota:psi_phi} and the functor $\X_\mp$ as in Section \ref{subsect:Xi}

%%%%%%%%%%%%%%%%%%%%%%%%%%%%%%%%%%%%%%%%%%%%%%%%%%%%%%%%%%%%%%%%%%%%%%%%%%%%%%%%%%%ù
\subsection{Comparison using $\X^*_\mp$}
\label{subsect:compare_X}

Consider a $(\kk \ddb)_\mp$-module $M$. By Theorem \ref{thm:Xi_Koszul_comparison}, we have the isomorphisms of Koszul complexes of $(\kk \uwb)_{-\mp}$-modules:
\begin{eqnarray*}
(\kk \uwb)_{-\mp} \otimes_{\kk (\tfb)} \X_\mp ^* M 
&
\cong &
(\X_{-\mp} ) _*(\kk \uwb)_{-\mp} \otimes_{\kk \fb} M
 \\
(\kk\uwb)_{-\mp}^\sharp  \otimes_{\kk (\tfb)} \X_\mp^* M
& \cong & 
(\X_{-\mp})_*\big( (\kk\uwb)_{-\mp}^\sharp \big) \otimes_{\kk \fb} M.
\end{eqnarray*}

To proceed, we analyse $(\X_{-\mp})  _*$ applied to $(\kk \uwb)_{-\mp}$ and to $ (\kk\uwb)_{-\mp}^\sharp $. These are both $(\kk \ddb)_{(-\mp)} \cten  (\kk \uwb)_{-\mp}$-modules.

For notational clarity, in the following general result, we replace $- \mp$ by $\mp$. 

\begin{prop}
\label{prop:apply_X_*}
\ 
\begin{enumerate}
\item 
There is a surjection of $(\kk \ddb)_{(\mp)} \cten  (\kk \uwb)_{\mp}$-modules
$$
\Upsilon_\mp^* (\kk \dub)_{\mp}
\twoheadrightarrow 
(\X_{\mp})  _*(\kk \uwb)_{\mp}. 
$$
After restriction to $\kk \fb\op \cten  (\kk \uwb)_{-\mp}$-modules, this admits a section. 
\item 
There is an isomorphism of $(\kk \ddb)_{(\mp)} \cten  (\kk \uwb)_{\mp}$-modules
$$
(\X_{\mp})_*\big( (\kk\uwb)_{\mp}^\sharp \big)
\cong 
\Upsilon_\mp^*( (\kk \dub)_{\mp}^\sharp).
$$
\end{enumerate}
\end{prop}

\begin{proof}
The first point is the counterpart of Corollary \ref{cor:surject_uspilon^*kkub_unwall_kkuwb}.  This requires paying attention to the fact that we are working  here with directed pairs. 
Consider an object $U$ of $(\kk \ddb)_{(\mp)}$  (i.e., a finite set) and an object $(X_1, X_2)$ of $(\kk \uwb)_{\mp}$ (i.e., a pair of finite sets). Then the underlying surjection is of the form
$$
(\kk \dub)_{\mp} (U, X_1 \amalg X_2) 
\twoheadrightarrow 
\bigoplus_{U = U_1 \amalg U_2}
(\kk \uwb)_{\mp} ((U_1, U_2), (X_1, X_2)). 
$$
Consider a generator $[f]$ of the domain represented by an element  $ f\in \dubord (U, X_1 \amalg X_2)$, which identifies as $\fb (U \amalg (\two \times \n), X_1 \amalg X_2)$ for the appropriate $n \in \nat$. This is sent to zero if 
$f$ contains a directed pair of one of the following forms: 
\begin{enumerate}
\item 
 a directed pair contained within $X_1$;
\item
 a directed pair contained within $X_2$; 
 \item 
 a directed pair of the form $(x_2, x_1)$ with $x_i \in X_i$.
 \end{enumerate}
  Otherwise, $f$ determines a decomposition $(U_1, U_2)$ of $U$ and a morphism  $f^{\mathrm{wall}} \in \uwbord ((U_1, U_2), (X_1, X_2))$ with the property that $f$ is recovered by applying the disjoint union (compare the results of Section \ref{subsect:unwalling_results}). In this case, the surjection is given by $[f] \mapsto [f^\mathrm{wall}]$ together with the inclusion of the summand indexed by $(U_1, U_2)$. 

From the above, it is clear that the surjection admits a linear section induced by the disjoint union of finite sets. Moreover, the surjection and its section are easily seen to be morphisms of 
$\kk \fb\op \cten  (\kk \uwb)_{-\mp}$-modules. To conclude, it remains to check that the surjection is also compatible with the full $(\kk \dub)_{\mp}$-module structure. This is checked by using the explicit description of $(\X_\mp)_*$ given in Proposition \ref{prop:Xi_*}. 

The second point is equivalent, via vector space duality, to the assertion that there is an isomorphism
\begin{eqnarray}
\label{eqn:Upsilon_X_iso}
(\Upsilon_\mp)_\sharp (\kk \uwb)_\mp 
\cong 
\X^*_\mp (\kk \dub)_\mp
\end{eqnarray}
of $(\kk \dwb)_\mp\cten  (\kk \dub)_\mp$-modules. (Note that the right adjoint $(\X_\mp)_*$ becomes the left adjoint $(\Upsilon_\mp)_\sharp$ upon passing across the duality.)  This is the analogue of Proposition \ref{prop:duspilon_lad} and is proved by the same argument. 
\end{proof}

\begin{rem}
The isomorphism (\ref{eqn:Upsilon_X_iso}) in the proof of Proposition \ref{prop:apply_X_*} has terms of the form
$$
\bigoplus_{U= U_1 \amalg U_2} 
(\kk \uwb)_\mp  ((X_1, X_2), (U_1, U_2) ) 
\stackrel{\cong}{\rightarrow} 
(\kk \dub)_\mp (X_1 \amalg X_2, U).
$$
\end{rem}

Putting Theorem \ref{thm:Xi_Koszul_comparison} together with  Proposition \ref{prop:apply_X_*} yields the following comparison results for the first and second Koszul complexes associated to $\X_\mp ^* M $ and $M$ respectively:

\begin{thm}
\label{thm:comparison_kz_cx_X}
For  a $(\kk \ddb)_\mp$-module $M$,
\begin{enumerate}
\item 
 there is a natural surjection of complexes of $(\kk \uwb)_{-\mp}$-modules:
 $$
 \Upsilon_{-\mp}^* \big ( (\kk \dub)_{-\mp} \otimes_{\kk \fb} M \big)  
\twoheadrightarrow 
(\kk \uwb)_{-\mp} \otimes_{\kk (\tfb)} \X_\mp ^* M 
 $$
 and this admits a section when restricted to complexes of $\kk (\tfb)$-modules;  
\item 
there is a natural isomorphism of complexes of $(\kk \uwb)_{-\mp}$-modules:
$$
\Upsilon_{-\mp}^*\big ( (\kk \dub)_{-\mp}^\sharp \otimes_{\kk _\fb} M \big) 
\cong 
(\kk \uwb)_{-\mp}^\sharp \otimes_{\kk (\tfb)} \X_\mp ^* M.
$$
\end{enumerate}
\end{thm}

On passage to homology, we have:

\begin{cor}
\label{cor:comparison_kz_cx_X_homology}
For  a $(\kk \ddb)_\mp$-module $M$,
\begin{enumerate}
\item 
 there is a natural surjection of graded $(\kk \uwb)_{-\mp}$-modules:
 $$
 \Upsilon_{-\mp}^*  H_* ( (\kk \dub)_{-\mp} \otimes_{\kk \fb} M)  
\twoheadrightarrow 
H_* ((\kk \uwb)_{-\mp} \otimes_{\kk (\tfb)} \X_\mp ^* M); 
 $$
\item 
there is a natural isomorphism of graded $(\kk \uwb)_{-\mp}$-modules:
$$
\Upsilon_{-\mp}^* H_* ( (\kk \dub)_{-\mp}^\sharp \otimes_{\kk _\fb} M ) 
\cong 
H_* ((\kk \uwb)_{-\mp}^\sharp \otimes_{\kk (\tfb)} \X_\mp ^* M).
$$
\end{enumerate}
\end{cor}

\begin{proof}
The functor $\Upsilon_{-\mp}^*$ is exact, hence commutes with the passage to homology. The second point follows immediately. 
 For the first point, one  uses the existence of the section to conclude that the induced morphism in homology is surjective.
\end{proof}

%%%%%%%%%%%%%%%%%%%%%%%%%%%%%%%%%%%%%%%%%%%%%%%%%%%%%%%%%%%%%%%%%%%%%%%%%
\subsection{Comparison using $\dmp^*_{(\pm; \mp)}$}
\label{subsect:compare_phi}
For $G$ a $(\kk \db)_{(\pm; \mp)}$-module, by Theorem \ref{thm:phi^*_sngtr}, we have the isomorphisms of Koszul complexes in $(\kk \dub)_{-\mp}$-modules:
\begin{eqnarray*}
 (\kk \dub)_{-\mp} \otimes_{\kk \fb} \dmp^* _{(\pm; \mp)} G
&\cong &
\sgntr_{(\pm; -\mp)} ^*   (\kk \dub)_{-\mp} \otimes_{\kk \fb} G
\\
 (\kk \dub)_{-\mp}^\sharp \otimes_{\kk \fb} \dmp^* _{(\pm; \mp)} G
&\cong &
\sgntr_{(\pm; -\mp)} ^*  \big(  (\kk \dub)_{-\mp}^\sharp\big)  \otimes_{\kk \fb} G
\end{eqnarray*}

To proceed, we relate $\sgntr_{(\pm; -\mp)} ^*   (\kk \dub)_{-\mp}$ to $(\kk \ub)_{(\pm; - \mp)}$ and $\sgntr_{(\pm; -\mp)} ^*  \big(  (\kk \dub)_{-\mp}^\sharp\big) $ to $(\kk \ub)_{(\pm; - \mp)}^\sharp$. The key point here is provided by Corollary \ref{cor:sct_versus_sgntr}, which shows that $\sgntr_{(\pm; -\mp)} ^* $ is naturally isomorphic to restriction along the section $\sct_{(\pm; -\mp)}$. 

The comparison is then based on the following elementary Lemma:

\begin{lem}
\label{lem:rho_sigma_elementary}
Suppose that the morphism of associative $\kk$-algebras $\rho : A \rightarrow B$ admits a section $\sigma : B \rightarrow A$. Then $\rho$ induces a morphism of left $B$-modules:
$$
\sigma^* A \twoheadrightarrow B,
$$
where $\sigma^*$  indicates restriction of the canonical left module structure along $\sigma$.   

Dualizing, $\rho$ induces an inclusion of right $B$-modules
$$
B^\sharp \hookrightarrow \sigma^* A^\sharp. 
$$
\end{lem}

\begin{proof}
Clearly  $\rho$ induces a surjective morphism of left $A$-modules $A \rightarrow \rho^* B$, writing $\rho^*$ for restriction along $\rho$. On applying the functor $\sigma^*$, since $\rho \circ \sigma = \id$ by hypothesis, one obtains the surjection of $B$-modules. The dual case follows similarly. 
\end{proof}

\begin{rem}
There is an obvious counterpart of the Lemma, switching the roles of the left and right module structures. 
\end{rem}

Lemma \ref{lem:rho_sigma_elementary} will be applied to the canonical surjection $\ump_{(\pm; \mp)} : (\kk \dub)_\mp \twoheadrightarrow (\kk \ub) _{(\pm; \mp)}$ (again, for notational clarity, considering the case $\mp$ here rather than $-\mp$ that occurs in the application). Here we use the natural {\em bimodule} structures, applying the Lemma with respect to the {\em right} structures; hence the statement of the result uses $\ump_{(\pm; \mp)}^*$ to indicate the restriction along the {\em left} module structure. 

\begin{prop}
\label{prop:compare_sngtr_psi}
\ 
\begin{enumerate}
\item 
The functor  $\ump_{(\pm; \mp)}$ induces  a surjection of $  (\kk \db)_{(\pm; \mp)} \cten  (\kk \dub)_\mp$-modules
$$
\sgntr_{(\pm; \mp)}^* (\kk \dub)_\mp
\twoheadrightarrow 
\ump_{(\pm; \mp)}^* (\kk \ub)_{(\pm; \mp)}.
$$
\item 
The  functor  $\ump_{(\pm; \mp)}$ induces  an inclusion  of $ (\kk \db)_{(\pm; \mp)}  \cten  (\kk \dub)_\mp$-modules
\begin{eqnarray}
\label{eqn:include_psi_sgntr}
\ump_{(\pm; \mp)}^* \big( (\kk \ub)_{(\pm; \mp)}^\sharp \big)
\hookrightarrow
\sgntr_{(\pm; \mp)}^* \big( (\kk \dub)_\mp^\sharp \big) 
.
\end{eqnarray}
\end{enumerate}
\end{prop}

\begin{proof}
We use the natural isomorphism $\sgntr_{(\pm; \mp)} ^* \cong \sct_{(\pm; \mp)}^* $  provided by Corollary \ref{cor:sct_versus_sgntr}. Then the result follows from Lemma \ref{lem:rho_sigma_elementary} (applying $\sgntr_{(\pm; \mp)} ^*$ with respect to the contravariant variable).
\end{proof}

Using these results, we get the following comparison morphisms between the first (respectively second) Koszul complexes for $\dmp^* _{(\pm; \mp)} G$ and for $G$ respectively.

\begin{cor}
\label{cor:comparison_kz_cx_phi}
For $G$ a $(\kk \db)_{(\pm; \mp)}$-module, 
\begin{enumerate}
\item 
there is a natural surjection of complexes in $(\kk \dub)_{-\mp}$-modules:
$$
 (\kk \dub)_{-\mp} \otimes_{\kk \fb} \dmp^* _{(\pm; \mp)} G
 \twoheadrightarrow 
 \ump_{(\pm;- \mp)}^* \big( (\kk \ub)_{(\pm; -\mp)} \otimes_{\kk \fb} G \big) ; 
$$
\item 
there is a natural inclusion of complexes in $(\kk \dub)_{-\mp}$-modules:
$$
\ump_{(\pm;- \mp)}^* \big( (\kk \ub)_{(\pm; -\mp)}^\sharp \otimes_{\kk \fb} G \big)
\hookrightarrow 
(\kk \dub)_{-\mp}^\sharp \otimes_{\kk \fb} \dmp^* _{(\pm; \mp)} G.
$$
\end{enumerate}
\end{cor}

%%%%%%%%%%%%%%%%%%%%%%%%%%%%%%%%%%%%%%%%%%%%%%%%%%%%%%%%%%%%%%%%%%%%%%%%%%
\subsection{Comparing the first Koszul complexes using $\dupsilon_{(\pm;\mp)}^*$}
 Recall that, for $G$ a $(\kk \db)_{(\pm;\mp)}$-module, we have the two first Koszul complexes 
\begin{eqnarray*}
(\kk \uwb)_{-\mp} \otimes_{\kk (\tfb)} \dupsilon _{(\pm; \mp)}^* G 
\\
(\kk \ub)_{(\pm;-\mp)} \otimes_{\kk \fb} G 
\end{eqnarray*}
of $(\kk \uwb)_{-\mp}$-modules and $(\kk \ub)_{(\pm;-\mp)}$-modules respectively. 

To relate these, rather than combining the results of Sections \ref{subsect:compare_X} and \ref{subsect:compare_phi}, we proceed directly by using Corollary \ref{cor:surject_uspilon^*kkub_unwall_kkuwb} in conjunction with Corollary \ref{cor:unwall_vs_dupsilon}, which gives the isomorphism of complexes of $(\kk \uwb)_{-\mp}$-modules:
$$
(\kk \uwb)_{-\mp} \otimes_{\kk (\tfb)} \dupsilon _{(\pm; \mp)}^* G 
\cong 
\unwall_{(\pm;-\mp)} (\kk \uwb)_{-\mp} 
\otimes_{\kk (\fb)} G.
$$
(Note that $G$ here corresponds to $B$ in the statement of Corollary \ref{cor:unwall_vs_dupsilon}.)

The following result is the counterpart of  the first Koszul complex case of Theorem \ref{thm:comparison_kz_cx_X}:

\begin{thm}
\label{thm:first_kz_cx_dupsilon}
 For $G$ a $(\kk \db)_{(\pm;\mp)}$-module, there is a natural surjection of  complexes of $(\kk \uwb)_{-\mp}$-modules: 
$$ 
 \upsilon_{(\pm;-\mp)} ^* \big( (\kk \ub)_{(\pm;-\mp)} \otimes_{\kk \fb} G  \big)
 \twoheadrightarrow 
 (\kk \uwb)_{-\mp} \otimes_{\kk (\tfb)} \dupsilon _{(\pm; \mp)}^* G. 
 $$ 
 \end{thm}

\begin{proof}
 Corollary \ref{cor:surject_uspilon^*kkub_unwall_kkuwb} gives the surjection of $(\kk \db)_{(\pm;-\mp)} \cten  (\kk \uwb)_{-\mp}$-modules:
 $$
  \upsilon_{(\pm;-\mp)} ^*  (\kk \ub)_{(\pm;-\mp)}
  \twoheadrightarrow 
  \unwall_{(\pm;-\mp)} (\kk \uwb)_{-\mp} ,
 $$
(noting that $\upsilon_{(\pm;-\mp)}^*$ is applied with respect to the covariant variable of $ (\kk \ub)_{(\pm;-\mp)}$). The result follows by forming the associated Koszul complexes and applying Corollary \ref{cor:unwall_vs_dupsilon}.
\end{proof}

Since the functor $ \upsilon_{(\pm;-\mp)} ^* $ is exact, Theorem \ref{thm:first_kz_cx_dupsilon} has the immediate consequence:

\begin{cor}
\label{cor:first_kz_cx_dupsilon_homology} 
 For $G$ a $(\kk \db)_{(\pm;\mp)}$-module, there is a natural morphism of graded $(\kk \uwb)_{-\mp}$-modules 
$$ 
 \upsilon_{(\pm;-\mp)} ^* H_* \big( (\kk \ub)_{(\pm;-\mp)} \otimes_{\kk \fb} G  \big)
 \rightarrow 
 H_* \big( (\kk \uwb)_{-\mp} \otimes_{\kk (\tfb)} \dupsilon _{(\pm; \mp)}^* G\big). 
 $$ 
\end{cor}

%%%%%%%%%%%%%%%%%%%%%%%%%%%%%%%%%%%%%%%%%%%%%%%%%%%%%%%%%%%%%%%%%%%%%%%%%%
\subsection{Comparing the second Koszul complexes using $\dupsilon_{(\pm;\mp)}^*$}
\label{subsect:2nd_Kz_dupsilon^*}
 For $G$ a $(\kk \db)_{(\pm;\mp)}$-module,  we have the second Koszul complexes 
\begin{eqnarray*}
(\kk \uwb)_{-\mp}^\sharp \otimes_{\kk (\tfb)} \dupsilon _{(\pm; \mp)}^* G 
\\
(\kk \ub)_{(\pm;-\mp)}^\sharp \otimes_{\kk \fb} G 
\end{eqnarray*}
of $(\kk \uwb)_{-\mp}$-modules  and $(\kk \ub)_{(\pm;-\mp)}$-modules respectively. 

Corollary \ref{cor:unwall_vs_dupsilon} gives the isomorphism of complexes of $(\kk \uwb)_{-\mp}$-modules
$$
(\kk \uwb)_{-\mp}^\sharp \otimes_{\kk (\tfb)} \dupsilon _{(\pm; \mp)}^* G 
\ \cong \ 
\unwall_{(\pm;-\mp)} \big( (\kk \uwb)_{-\mp}^\sharp \big)  
\otimes_{\kk (\fb)} G.
$$

To apply this, we use the following result that is derived from those of Sections \ref{subsect:compare_X} and \ref{subsect:compare_phi} (stated for $(\pm; \mp)$ rather than $(\pm; -\mp)$ for notational clarity):

\begin{prop}
\label{prop:inclusion_upsilon*_unwall_sharp}
There is an inclusion of complexes of $(\kk \db)_{(\pm; \mp)} \cten  (\kk \uwb)_\mp$-modules:
$$
\upsilon_{(\pm;\mp)}^* \big( (\kk \ub)_{(\pm;\mp)}^\sharp\big) 
\hookrightarrow 
\unwall_{(\pm; \mp)} (\kk \uwb)_{\mp}^\sharp.
$$
\end{prop}

\begin{proof}
By definition, $\unwall_{(\pm; \mp)}$ is the composite $\sgntr_{(\pm; \mp)}^* \circ (\X_{\mp})_*$. Proposition \ref{prop:apply_X_*}
gives the isomorphism of $(\kk \ddb)_{(\mp)} \cten  (\kk \uwb)_{\mp}$-modules
$$
(\X_{\mp})_*\big( (\kk\uwb)_{\mp}^\sharp \big)
\cong 
\Upsilon_\mp^*\big( (\kk \dub)_{\mp}^\sharp\big).
$$
This gives 
$$
\unwall_{(\pm; \mp)} \big( (\kk\uwb)_{\mp}^\sharp \big)
\cong 
\sgntr_{(\pm; \mp)}^* \Upsilon_\mp^*\big( (\kk \dub)_{\mp}^\sharp\big)
\cong 
 \Upsilon_\mp^* \sgntr_{(\pm; \mp)}^*\big( (\kk \dub)_{\mp}^\sharp\big),
$$
using the fact that the functors $ \Upsilon_\mp^*$ and $ \sgntr_{(\pm; \mp)}^*$ are applied with respect to different variables for the second isomorphism.

The functor $ \Upsilon_\mp^*$ is exact, hence applying it to the inclusion (\ref{eqn:include_psi_sgntr}) of Proposition \ref{prop:compare_sngtr_psi} gives 
\begin{eqnarray*}
 \Upsilon_\mp^* \ump_{(\pm; \mp)}^* \big( (\kk \ub)_{(\pm; \mp)}^\sharp \big)
\hookrightarrow
 \Upsilon_\mp^* \sgntr_{(\pm; \mp)}^* \big( (\kk \dub)_\mp^\sharp \big) 
.
\end{eqnarray*}
The result follows, since $\upsilon_{(\pm;\mp)}^*$ is naturally isomorphic to $\Upsilon_\mp^* \sgntr_{(\pm; \mp)}^* $.
\end{proof}

\begin{rem}
The morphisms of $\kk$-vector spaces underlying Proposition \ref{prop:inclusion_upsilon*_unwall_sharp} are easily described. After dualizing, these correspond to the surjections 
$$
\bigoplus_{X = X_1 \amalg X_2} 
(\kk \uwb)_\mp ((U_1, U_2), (X_1, X_2) ) 
\twoheadrightarrow 
(\kk \ub)_{(\pm; \mp)} (U_1 \amalg U_2, X)
$$
induced by $\amalg$ and by passage from $(\kk \dub)_\mp$ to $(\kk \ub)_{(\pm;\mp)}$. The significance of the Proposition is that it keeps track of the naturality.
\end{rem}

This yields the second Koszul complex counterpart of Theorem \ref{thm:first_kz_cx_dupsilon}:

\begin{thm}
\label{thm:second_kz_cx_dupsilon}
For $G$ a $(\kk \db)_{(\pm; \mp)}$-module, there is a natural inclusion of complexes of $(\kk \uwb)_{-\mp}$-modules: 
$$
\upsilon_{(\pm;-\mp)}^*
\big( (\kk \ub)_{(\pm;-\mp)}^\sharp \otimes_{\kk \fb} G  \big) 
\hookrightarrow 
(\kk \uwb)_{-\mp}^\sharp \otimes_{\kk (\tfb)} \dupsilon _{(\pm; \mp)}^* G .
$$
\end{thm}

\begin{proof}
The proof is similar to that of Theorem \ref{thm:first_kz_cx_dupsilon}, based on Proposition \ref{prop:inclusion_upsilon*_unwall_sharp} and the isomorphism provided by Corollary \ref{cor:unwall_vs_dupsilon}.
\end{proof}

On passing to homology, we have:

\begin{cor}
\label{cor:second_kz_cx_dupsilon_homology}
For $G$ a $(\kk \db)_{(\pm; \mp)}$-module, there is a natural morphism of graded  $(\kk \uwb)_{_\mp}$-modules: 
$$
\upsilon_{(\pm;-\mp)}^*
H_* ( (\kk \ub)_{(\pm;-\mp)}^\sharp \otimes_{\kk \fb} G  ) 
\rightarrow 
H_* (
(\kk \uwb)_{-\mp}^\sharp \otimes_{\kk (\tfb)} \dupsilon _{(\pm; \mp)}^* G ).
$$
\end{cor}

\part{Application to graph complexes}
\label{part:graph}

\section{Graph complexes associated to  cyclic operads and dioperads}
\label{sect:koszul_graph} 
 
In this short section we restate the forms of the first and second Koszul complexes associated to cyclic operads and dioperads and recall (from \cite{P_cyclic} and \cite{P_opd_wall}) that the second Koszul complexes can be interpreted as hairy graph complexes.

%%%%%%%%%%%%%%%%%%%%%%%%%%%%%%%%%%%%%%%%%%%%%%%%%%%%%%%%%%%%%%%%%%%%%%%%%%%%%%%%%%%%
\subsection{The cyclic case}

For cyclic operads, we have the following  consequence of Theorem \ref{thm:cpd_db_kdbm} (as considered in \cite{P_cyclic}):

\begin{cor}
\label{cor:cpd_Koszul_complexes}
For $\cpd$ a cyclic operad in $\nuco$, there are the following natural Koszul complexes
\begin{enumerate}
\item 
of $(\kk\ub)_{(+;-)}$-modules:
\begin{eqnarray*}
&&
(\kk \ub)_{(+;-)} \otimes_{\kk \fb} \sfb^* \cpd 
\\
 &&
(\kk \ub)_{(+;-)}^\sharp \otimes_{\kk \fb} \sfb^* \cpd ;
\end{eqnarray*}
\item 
of $(\kk \ub)_{(-; +)}$-modules:
\begin{eqnarray*}
&&
(\kk \ub)_{(-;+)} \otimes_{\kk \fb} \lfb^* \cpd 
\\
 &&
(\kk \ub)_{(-;+)}^\sharp \otimes_{\kk \fb} \lfb^* \cpd .
\end{eqnarray*}
\end{enumerate}
\end{cor}

\begin{rem}
We consider the `second Koszul complexes' $(\kk \ub)_{(+;-)}^\sharp \otimes_{\kk \fb} \sfb^* \cpd$ and $(\kk \ub)_{(-;+)}^\sharp \otimes_{\kk \fb} \lfb^* \cpd$ as the odd and even hairy graph complexes associated to the cyclic operad $\cpd$, generalizing the construction of Conant and Vogtmann \cite{MR2026331} that, in turn, generalized the construction of Kontsevich \cite{MR1341841,MR1247289}. 
 Here, we do not require that the graphs be connected. (See \cite{P_cyclic} for more details on these hairy graph complexes.) 
\end{rem}

\begin{exam}
\label{exam:KCV}
The `first Koszul complexes'  are less familiar, although they encode important structures. For example, in the even case, the complex is a precursor of the Chevalley-Eilenberg complex for the  Kontsevich-Conant-Vogtmann Lie algebra \cite{MR2026331}. We outline how this works, referring to \cite{P_cyclic} for more details.

Let $V$ be a symplectic $\kk$-vector space. Write $V^{\otimes \bullet}$ for the associated $(\kk \db)_{(-;+)}$-module, as in Example \ref{exam:restrict_sympl_vs_module}. Then one can form the complex
$$
V^{\otimes \bullet } \otimes_{(\kk \ub)_{(-;+)} }(\kk \ub)_{(-;+)} \otimes_{\kk \fb} \lfb^* \cpd 
\cong 
V^{\otimes \bullet} \otimes_{\kk \fb} \lfb^* \cpd ,
$$
equipped with the Koszul differential. 

Using the general observation of Section \ref{subsect:such_S_Lambda}, the underlying object of $V^{\otimes \bullet} \otimes_{\kk \fb} \lfb^* \cpd$ is isomorphic to the exterior algebra $\Lambda^* \cpd (V)$ on the Schur functor of $\cpd$ evaluated on $V$ (forgetting the symplectic structure). The complex is isomorphic to the Chevalley-Eilenberg complex of the Lie algebra $\cpd (V)$, using the Kontsevich-Conant-Vogtmann Lie algebra structure.

The second Koszul complex (aka. hairy graph complex) is, by construction, obtained from the first by applying the functor $(\kk \ub)_{(-;+)}^\sharp \otimes_{(\kk \ub)_{(-;+)}} -$. This process can be compared with the stabilization with respect to $V$ of the above Chevalley-Eilenberg complex, as considered in  \cite{MR2026331}, generalizing the ideas of Kontsevich. (This is explained in \cite{P_cyclic}.)
\end{exam}

\begin{rem}
From the point of view of cyclic operad theory (or, more generally, modular operads), hairy graph complexes are subsumed in the construction of the Feynman transform. Here, by sacrificing the information coming from connectivity of graphs, we have lost some of the structure.
\end{rem}

%%%%%%%%%%%%%%%%%%%%%%%%%%%%%%%%%%%%%%%%%%%%%%%%%%%%%%%%%%%%%%%%%%%%%%%%%%%
\subsection{The dioperad case}

Using the module structures given by Theorem \ref{thm:diopd_dwb_kkdwb-_modules}, we can form the associated Koszul complexes:

\begin{cor}
\label{cor:diopd_Koszul_complexes}
For $\diopd$ a dioperad in $\nudo$, there are the following natural Koszul complexes
\begin{enumerate}
\item 
of $(\kk\uwb)_{-}$-modules:
\begin{eqnarray*}
&&
(\kk \uwb)_{-} \otimes_{\kk (\tfb)} \stfb^* \diopd 
\\
 &&
(\kk \uwb)_{-}^\sharp \otimes_{\kk (\tfb)} \stfb^* \diopd ;
\end{eqnarray*}
\item 
of $\kk \uwb$-modules:
\begin{eqnarray*}
&&
\kk \uwb \otimes_{\kk (\tfb)} \ltfb^* \diopd 
\\
 &&
(\kk \uwb)^\sharp \otimes_{\kk (\tfb)} \ltfb^* \diopd .
\end{eqnarray*}
\end{enumerate}
\end{cor}

\begin{rem}
Generalizing the case of operads treated in \cite{P_opd_wall}, we consider the second Koszul complexes $(\kk \uwb)_{-}^\sharp \otimes_{\kk (\tfb)} \stfb^* \diopd $ and $(\kk \uwb)^\sharp \otimes_{\kk (\tfb)} \ltfb^* \diopd $ as (a walled version of) hairy graph complexes associated to the dioperad $\diopd$. Again, the graphs are not required to be connected.
\end{rem}

For the `first Koszul complexes', we have the following  counterpart of Example \ref{exam:KCV}: 

\begin{exam}
\label{exam:generalize_der}
Let $W$ be a finite-dimensional $\kk$-vector space and write $T^{\bullet, \bullet} (W)$ for the natural $\kk \dwb$-module that sends $(\m , \n)$ to $(W^\sharp)^{\otimes m} \otimes W^{\otimes n}$. 
 We may form the complex 
$$
T^{\bullet, \bullet} (W)\otimes_{\kk \uwb} \kk \uwb \otimes_{\kk (\tfb)) } \ltfb^* \diopd
\cong 
T^{\bullet, \bullet} (W) \otimes_{\kk (\tfb) } \ltfb^* \diopd.
$$
As in Section \ref{subsect:such_S_Lambda}, the right hand side can be rewritten as 
$$
\Lambda^* (T^{\bullet, \bullet} (W) \otimes_{\kk (\tfb)} \diopd ) ,
$$
the exterior algebra on the `bivariant Schur functor' associated to the underlying $\kk (\tfb)$-module of $\diopd$.  

Inspecting the form of the differential shows that the dioperad structure of $\diopd$ makes $T^{\bullet, \bullet} (W) \otimes_{\kk (\tfb)} \diopd$ naturally into a Lie algebra with associated Chevalley-Eilenberg complex isomorphic to $T^{\bullet, \bullet} (W) \otimes_{\kk (\tfb) } \ltfb^* \diopd$. 

In the case where $\diopd$ is an operad (i.e., identifies with $\opd_\diopd$), this Lie algebra is simply $\der ( \opd_\diopd (W))$, the Lie algebra of derivations of the free $\opd_\diopd$-algebra on $W$, and the complex is the Chevalley-Eilenberg complex of $\der ( \opd_\diopd (W))$. This setting (and a wheeled generalization) was studied by Dotsenko in \cite{MR4945404}, who then considered the stabilization with respect to $W$. His work exploited the {\em wheeled bar construction} as a model for (connected) hairy graph complexes.

Similarly to the cyclic case, as explained in \cite{P_opd_wall}, this procedure can be compared with the passage from the first Koszul complex to the hairy graph complex by applying the functor $(\kk \uwb)^\sharp \otimes_{\kk \uwb} -$.
\end{exam}

\section{Comparison results for graph complexes}
\label{sect:graph_cx_compare}

In this section we apply the general results on comparison of Koszul complexes to the cases of modules arising from cyclic operads and from dioperads. For the results, we focus upon the second Koszul complexes since these are more familiar, corresponding to (hairy) graph complexes.

%%%%%%%%%%%%%%%%%%%%%%%%%%%%%%%%%%%%%%%%%%%%%%%%%%%%%%%%%%%%%%%%%%%%%%%%%%%%%%%%%%%%%%
\subsection{Initiating the comparisons}

We seek  to understand the relationship between the first and second Koszul complexes in the respective cyclic operad and dioperad cases (as recalled in Section \ref{sect:koszul_graph}), via the functors of Corollary \ref{cor:opds_to_copds} and of Corollary \ref{cor:copds_to_opds}:
\begin{eqnarray*}
\nuo \hookrightarrow \nudo \stackrel{\amalg_*}{ \rightarrow} \nuco
\
\nuco \stackrel{\amalg^*} {\rightarrow} \nudo \stackrel{\opd_\bullet}{\rightarrow} \nuo.
\end{eqnarray*}

To do this, we use the isomorphisms given by Theorem \ref{thm:diopd_2_cyclic_modules}:
\begin{eqnarray*}
\sfb^* (\amalg_* \diopd) & \cong & \unwall_{(+;+)} \stfb^* \diopd 
\\
\lfb^* (\amalg_* \diopd) & \cong & \unwall_{(-;-)} \ltfb^* \diopd 
\end{eqnarray*}
and those given by Theorem \ref{thm:cyclic_2_diopd_modules}:
\begin{eqnarray*}
\stfb^* (\amalg^* \cpd) &\cong & \dupsilon_{(+;+)}^* \sfb^* (\cpd) 
\\
\ltfb^* (\amalg^* \cpd) & \cong & \dupsilon_{(-;-)}^* \lfb^* (\cpd).
\end{eqnarray*}

\begin{rem}
\label{rem:opd_amalg^*cpd}
If we are interested in the underlying operad associated to the cyclic operad $\cpd$, then this corresponds to $\opd_{\amalg^* \cpd} \hookrightarrow \amalg^*\cpd$.  In this case, Theorem \ref{thm:cyclic_2_diopd_modules} yields the inclusions
\begin{eqnarray*}
\stfb^* (\opd_{\amalg^* \cpd}) &\hookrightarrow & \dupsilon_{(+;+)}^* \sfb^* (\cpd) 
\\
\ltfb^* (\opd_{\amalg^* \cpd}) & \hookrightarrow  & \dupsilon_{(-;-)}^* \lfb^* (\cpd)
\end{eqnarray*}
of the respective modules.
\end{rem}

This leads us to consider the isomorphisms given in Example \ref{exam:first_second_kz_cx_unwall_dupsilon}. 

\begin{exam}
\label{exam:compare_kz_cx_diopd}
For a dioperad $\diopd$ in $\nudo$,  for the first Koszul complexes this gives:
\begin{eqnarray*}
(\kk \ub)_{(+;-)} \otimes_{\kk \fb} \sfb^* (\amalg_* \diopd) & \cong & (\kk \ub)_{(+;-)}  \otimes_{\kk \fb}\unwall_{(+;+)} \stfb^* \diopd 
\\
&\cong &
\dupsilon_{(+;-)}^* (\kk \ub)_{(+;-)} \otimes _{\kk(\tfb)} \stfb^* \diopd
\\
(\kk \ub)_{(-;+)} \otimes_{\kk \fb} \lfb^* (\amalg_* \diopd) & \cong & (\kk \ub)_{(-;+)}  \otimes_{\kk \fb}\unwall_{(-;-)} \ltfb^* \diopd 
\\
&\cong &\dupsilon_{(-;+)}^* (\kk \ub)_{(-;+)} \otimes _{\kk(\tfb)} \ltfb^* \diopd.
\end{eqnarray*}
For the second Koszul complexes (aka. hairy graph complexes), this gives: 
\begin{eqnarray*}
(\kk \ub)^\sharp _{(+;-)} \otimes_{\kk \fb} \sfb^* (\amalg_* \diopd) & \cong & (\kk \ub)_{(+;-)} ^\sharp  \otimes_{\kk \fb}\unwall_{(+;+)} \stfb^* \diopd
\\
& \cong & \dupsilon_{(+;-)}^* \big((\kk \ub)_{(+;-)}^\sharp\big)  \otimes _{\kk(\tfb)} \stfb^* \diopd
\\
(\kk \ub)_{(-;+)}^\sharp \otimes_{\kk \fb} \lfb^* (\amalg_* \diopd) & \cong & (\kk \ub)_{(-;+)}^\sharp  \otimes_{\kk \fb}\unwall_{(-;-)} \ltfb^* \diopd
\\
& \cong &\dupsilon_{(-;+)}^* \big((\kk \ub)_{(-;+)} ^\sharp\big) \otimes _{\kk(\tfb)} \ltfb^* \diopd.
\end{eqnarray*}
\end{exam}

\begin{exam}
\label{exam:cyclic_kz_cx_compare}
For a cyclic operad $\cpd$ in $\nuco$, for the first Koszul complexes, this gives:
\begin{eqnarray*}
(\kk \uwb)_- \otimes_{\kk (\tfb)} \stfb^* (\amalg^* \cpd)
 &\cong &
(\kk \uwb)_- \otimes_{\kk (\tfb)} \dupsilon_{(+;+)}^* \sfb^* (\cpd) 
\\
&\cong &
\unwall_{(+;-)} (\kk \uwb)_-  \otimes_{\kk \fb} \sfb^* (\cpd) 
\\
 \kk \uwb \otimes_{\kk (\tfb)} \ltfb^* (\amalg^* \cpd)
 &\cong &
\kk \uwb \otimes_{\kk (\tfb)} \dupsilon_{(-;-)}^* \lfb^* (\cpd) 
\\
&\cong &
\unwall_{(-;+)}( \kk \uwb )  \otimes_{\kk \fb} \lfb^* (\cpd) .
\end{eqnarray*}
For the second Koszul complexes (aka. hairy graph complexes):
\begin{eqnarray*}
(\kk \uwb)_-^\sharp  \otimes_{\kk (\tfb)} \stfb^* (\amalg^* \cpd)
 &\cong &
(\kk \uwb)_-^\sharp  \otimes_{\kk (\tfb)} \dupsilon_{(+;+)}^* \sfb^* (\cpd) 
\\
&\cong &
\unwall_{(+;-)} \big( (\kk \uwb)_-^\sharp)  \otimes_{\kk \fb} \sfb^* (\cpd) 
\\
 (\kk \uwb) ^\sharp \otimes_{\kk (\tfb)} \ltfb^* (\amalg^* \cpd)
 &\cong &
(\kk \uwb)^\sharp  \otimes_{\kk (\tfb)} \dupsilon_{(-;-)}^* \lfb^* (\cpd) 
\\
&\cong & 
\unwall_{(-;+)} \big ( (\kk \uwb)^\sharp \big)    \otimes_{\kk \fb} \lfb^* (\cpd) .
\end{eqnarray*}

\end{exam}

%%%%%%%%%%%%%%%%%%%%%%%%%%%%%%%%%%%%%%%%%%%%%%%%%%%%%%
\subsection{The cyclic case}
Let $\cpd$ be a cyclic operad in $\nuco$ and $\amalg^* \cpd$ be the associated dioperad in $\nudo$. As in Example \ref{exam:cyclic_kz_cx_compare}, we can relate the associated second Koszul complexes (aka. hairy graph complexes) and then apply Theorem \ref{thm:second_kz_cx_dupsilon} and its homological consequence, Corollary \ref{cor:second_kz_cx_dupsilon_homology}. This yields:

\begin{thm}
\label{thm:compare_cyclic_case_grph_cx}
For $\cpd$ in $\nuco$, 
\begin{enumerate}
\item
there is an inclusion of complexes of $(\kk \uwb)_-$-modules:
$$
\upsilon^*_{(+;-)} \big( (\kk \ub)^\sharp_{(+;-)} \otimes_{\kk \fb} \sfb^* \cpd) \big)
\hookrightarrow 
(\kk \uwb)_-^\sharp \otimes_{\kk (\tfb)} \stfb^* (\amalg^* \cpd);
$$
\item 
there is an inclusion of complexes of $\kk \uwb$-modules:
$$
\upsilon^*_{(-;+)} \big( (\kk \ub)^\sharp_{(-;+)} \otimes_{\kk \fb} \lfb^* \cpd) \big)
\hookrightarrow 
(\kk \uwb)^\sharp \otimes_{\kk (\tfb)} \ltfb^* (\amalg^* \cpd).
$$
\end{enumerate}

Hence, passing to homology:
\begin{enumerate}
\item
there is a morphism of graded $(\kk \uwb)_-$-modules:
$$
\upsilon^*_{(+;-)} H_* ( (\kk \ub)^\sharp_{(+;-)} \otimes_{\kk \fb} \sfb^* \cpd) )
\rightarrow 
H_* ((\kk \uwb)_-^\sharp \otimes_{\kk (\tfb)} \stfb^* (\amalg^* \cpd));
$$
\item 
there is a morphism of graded $\kk \uwb$-modules:
$$
\upsilon^*_{(-;+)} H_*( (\kk \ub)^\sharp_{(-;+)} \otimes_{\kk \fb} \lfb^* \cpd) )
\rightarrow 
H_* ((\kk \uwb)^\sharp \otimes_{\kk (\tfb)} \ltfb^* (\amalg^* \cpd)).
$$
\end{enumerate}
\end{thm}

\begin{rem}
\ 
\begin{enumerate}
\item 
Theorem \ref{thm:compare_cyclic_case_grph_cx} gives a way of comparing the (odd or even) hairy  graph complex associated to the cyclic operad $\cpd$ with the (odd or even) hairy  graph complex associated to the dioperad $\amalg^* \cpd$.  
\item 
Unfortunately, this does not give a direct comparison with the (odd or even) hairy graph complex associated to the operad $\opd_{\amalg^* \cpd}$, although we do have inclusions of complexes 
\begin{eqnarray*}
(\kk \uwb)_-^\sharp \otimes_{\kk (\tfb)} \stfb^* (\opd_{\amalg^* \cpd})
&\hookrightarrow & 
(\kk \uwb)_-^\sharp \otimes_{\kk (\tfb)} \stfb^* (\amalg^* \cpd);
\\
(\kk \uwb)^\sharp \otimes_{\kk (\tfb)} \ltfb^* (\opd_{\amalg^* \cpd})
&\hookrightarrow & 
(\kk \uwb)^\sharp \otimes_{\kk (\tfb)} \ltfb^* (\amalg^* \cpd).
\end{eqnarray*}
(Cf. Remark \ref{rem:opd_amalg^*cpd}.)
\end{enumerate}
\end{rem}

%%%%%%%%%%%%%%%%%%%%%%%%%%%%%%%%%%%%%%%%%%%%%%%%%%%%%%%%%%%%%%%%%%%%%%%%%%%%%%%%%
\subsection{The dioperad case}

Let $\diopd$ be a dioperad in $\nudo$ and $\amalg_* \diopd$ be its associated cyclic operad. As in Example \ref{exam:compare_kz_cx_diopd} we may compare Koszul complexes, applying Theorem \ref{thm:unwall_second_complex} and its homological consequence Corollary \ref{cor:hom_Kz_cx_unwall}. This yields:

\begin{thm}
\label{thm:compare_diopd_case_grph_cx}
For $\diopd$ in $\nudo$, there are isomorphisms of complexes of $\kk \fb$-modules:
\begin{eqnarray*}
(\kk \ub)_{(+;-)}^\sharp \otimes_{\kk \fb} \sfb^* (\amalg_* \diopd) 
&\cong &
\big( \amalg_* \op (\kk \ub)_{(+; -)}^{\cten  2}\big)^\sharp
\otimes_{\kk (\tfb)}
\big( (\kk\uwb)_{-}^\sharp  \otimes_{\kk(\tfb)} \stfb^* \diopd  \big) 
\\
(\kk \ub)_{(-;+)}^\sharp \otimes_{\kk \fb} \lfb^* (\amalg_* \diopd) 
&\cong &
\big( \amalg_* \op (\kk \ub)_{(-; +)}^{\cten  2}\big)^\sharp
\otimes_{\kk (\tfb)}
\big( (\kk\uwb)^\sharp  \otimes_{\kk(\tfb)} \ltfb^* \diopd \big). 
\end{eqnarray*}
Hence, on passage to homology, there are isomorphisms of graded $\kk \fb$-modules:
 \begin{eqnarray*}
H_* ((\kk \ub)_{(+;-)}^\sharp \otimes_{\kk \fb} \sfb^* (\amalg_* \diopd) )
&\cong &
\big( \amalg_* \op (\kk \ub)_{(+; -)}^{\cten  2}\big)^\sharp
\otimes_{\kk (\tfb)}
H_* \big( (\kk\uwb)_{-}^\sharp  \otimes_{\kk(\tfb)} \stfb^* \diopd  \big) 
\\
H_* ((\kk \ub)_{(-;+)}^\sharp \otimes_{\kk \fb} \lfb^* (\amalg_* \diopd) )
&\cong &
\big( \amalg_* \op (\kk \ub)_{(-; +)}^{\cten  2}\big)^\sharp
\otimes_{\kk (\tfb)}
H_*  \big( (\kk\uwb)^\sharp  \otimes_{\kk(\tfb)} \ltfb^* \diopd \big). 
\end{eqnarray*}
\end{thm} 

\begin{rem}
Theorem \ref{thm:compare_diopd_case_grph_cx} tells us that we can calculate the (odd or even) hairy graph homology of the cyclic operad $\amalg_*\diopd$ in terms of the (odd or even)  hairy graph homology of the dioperad $\diopd$. Clearly, this restricts directly to the case where $\diopd$ is an operad. 
\end{rem}

Corollary \ref{cor:direct_comparison_kz_1} has the following striking consequence:

\begin{cor}
\label{cor:comparison_opd_unit_case}
Suppose that $\opd$ is an operad with unit (i.e.,  in $\opds$). Then 
there are isomorphisms of $\kk$-vector spaces
 \begin{eqnarray*}
H_* ((\kk \ub)_{(+;-)}^\sharp \otimes_{\kk \fb} \sfb^* (\amalg_* \opd) )(\zero) 
&\cong &
H_* \big( (\kk\uwb)_{-}^\sharp  \otimes_{\kk(\tfb)} \stfb^* \opd  \big) (\zero, \zero) 
\\
H_* ((\kk \ub)_{(-;+)}^\sharp \otimes_{\kk \fb} \lfb^* (\amalg_* \opd) )(\zero) 
&\cong &
H_*  \big( (\kk\uwb)^\sharp  \otimes_{\kk(\tfb)} \ltfb^* \opd \big)(\zero, \zero) . 
\end{eqnarray*}
All other homology groups vanish.
\end{cor}

\begin{proof}
This follows from 
Corollary \ref{cor:direct_comparison_kz_1},
 using  that $H_* \big( (\kk\uwb)_{-}^\sharp  \otimes_{\kk(\tfb)} \stfb^* \opd  \big)$ and $H_*  \big( (\kk\uwb)^\sharp  \otimes_{\kk(\tfb)} \ltfb^* \opd \big)$ are both supported on $(\zero, \zero)$, since $\opd$ has a unit.

This  property is the counterpart of \cite[Theorem 10.1]{P_cyclic} which is for cyclic operads with a unit. In the operadic case, this can be viewed as a mild generalization of the acyclicity of the bar construction of a unital associative $\kk$-algebra. The result can be proved by using the same arguments as in \cite{P_cyclic}, transposed to the setting of \cite{P_opd_wall}.
\end{proof}

\part*{}
\appendix

\section{Day convolution}
\label{sect:day_convolution}

This section reviews the convolution products on $\kk \fb$-modules and on $\kk (\tfb)$-modules and some of their properties.

%%%%%%%%%%%%%%%%%%%%%%%%%%%%%%%%%%%%%%%%%%%%%%%%%%%%%%%%%%%%%%%%%%%%%%%%%%%%%%
\subsection{The disjoint union of finite sets and associated functors}
\label{subsect:disjoint_union}

The disjoint union of finite sets yields the functor $\amalg  : \fb \times \fb \rightarrow \fb$ and this defines a symmetric monoidal structure on $\fb$, with unit $\emptyset$. 
This extends to a symmetric monoidal structure (denoted here by $\tmalg$) on $\tfb$:
\begin{eqnarray*}
 \tmalg &:& (\tfb) \times (\tfb) \rightarrow \tfb
 \\
&&\big( (X,Y), (U,V)\big)  \mapsto  (X \amalg U, Y \amalg V).
\end{eqnarray*}

These symmetric monoidal functors are related via $\amalg$:

\begin{prop}
\label{prop:amalg_circ_tmalg} 
The functor $\amalg : \tfb \rightarrow \fb$ is symmetric monoidal. In particular, there is a natural isomorphism 
$
\amalg \circ \tmalg \cong \amalg \circ (\amalg \times \amalg)
$ 
of functors from $\fb^{\times 4}$ to $\fb$.
\end{prop}

\begin{proof}
This follows for example from the corresponding statement with $\fb$ replaced by the category of finite sets, so that the symmetric monoidal structures are the cocartesian ones.
\end{proof}

The above  yield  the restriction functors $\amalg^* : \f (\fb ) \rightarrow \f (\tfb)$ and $\tmalg ^* : \f (\tfb) \rightarrow  \f ((\tfb)^{\times 2}) $.
For a $\kk \fb$-module $F$ and a $\kk (\tfb)$-module $G$, these identify respectively as
\begin{eqnarray*} 
\amalg^* F (U, V)&= &F (U \amalg V) 
\\
\tmalg^* G ((X,Y) , (U,V))&= &G (X \amalg U, Y \amalg V). 
\end{eqnarray*}

Using Kan extension, these restriction functors admit both left and right adjoints. We spell this out explicitly.

\begin{defn}
\label{defn:amalg_tmalg_left}
\ 
\begin{enumerate}
\item 
Let $\amalg_* : \f (\tfb) \rightarrow \f (\fb)$ be the functor given on objects, for $G$ a $\kk (\tfb)$-module, by 
\[
(\amalg_* G )(X):= \bigoplus_{ U \amalg V=X} G (U, V),
\]
where the sum is indexed by ordered decompositions of $X$ into two subsets.
\item 
Let $\tmalg_* : \f ((\tfb)^{\times 2}) \rightarrow \f (\tfb)$ be the functor given on objects, for $H$ a $\kk ((\tfb)^{\times 2})$-module, by 
\begin{eqnarray*}
%\label{eqn:tmalg_*H}
(\tmalg_* H) (U, V):= \bigoplus_{\substack{U_1 \amalg U_2 =U\\  V_1 \amalg V_2=V}} H ((U_1, V_1), (U_2, V_2)),
\end{eqnarray*}
where the sum is indexed over pairs of ordered decompositions of $U$ and of $V$ into two subsets.
\end{enumerate}
\end{defn}

\begin{prop}
\label{prop:amalg_tmalg_adjoints}
\
\begin{enumerate}
\item 
The functor $\amalg_* : \f (\tfb) \rightarrow \f (\fb)$ is both left and right adjoint to $\amalg^* : \f (\fb ) \rightarrow \f (\tfb)$.
\item 
The functor $\tmalg_* : \f ((\tfb)^{\times 2}) \rightarrow \f (\tfb) $ is both left and right adjoint to $\tmalg^* :\f (\tfb) \rightarrow  \f ((\tfb)^{\times 2})$.
\end{enumerate}
\end{prop}

\begin{proof}
This is standard and can be seen by analysing the Kan extensions. The fact that the left and right adjoints are naturally isomorphic corresponds to the fact that, for finite groups, induction and coinduction are naturally isomorphic.
\end{proof}

\begin{exam}
\label{exam:units_amalg}
Write  $\kk _\mathbf{0}$ for  the $\kk \fb$-module supported on $\mathbf{0}= \emptyset$ with value $\kk$ and $\kk _{(\mathbf{0}, \mathbf{0})}$ for the $\kk (\tfb)$-module supported on $(\mathbf{0}, \mathbf{0})$ with value $\kk$. Then there are natural isomorphisms 
$
\amalg^* \kk_\mathbf{0} \cong   \kk _{(\mathbf{0}, \mathbf{0})}$ and 
$
\amalg_* \kk _{(\mathbf{0}, \mathbf{0})} \cong   \kk_\mathbf{0}$. 
\end{exam}

Proposition \ref{prop:amalg_circ_tmalg} leads to the following technical result:

\begin{lem}
\label{lem:projection_formula}
\
\begin{enumerate}
\item 
The functors $\amalg_* \circ \tmalg _* $ and $\amalg_* \circ (\amalg \times \amalg)_*$  from $\f (\fb^{\times 4})$ to $\f (\fb)$
are naturally isomorphic. 
\item 
The functors $\tmalg_* \circ (\amalg \times \amalg)^*$ and $\amalg^* \circ \amalg_*$ from $\f (\fb \times \fb)$ to $\f (\fb \times \fb)$ are naturally isomorphic.
\end{enumerate}
\end{lem}

\begin{proof}
The first statement follows directly from Proposition \ref{prop:amalg_circ_tmalg}, since the  composites are  naturally isomorphic to the  respective left adjoint of the restriction functors associated to the isomorphic functors $\amalg \circ \tmalg $ and $\amalg \circ (\amalg \times \amalg)$.

For the second statement, we treat $\tmalg_*$ and $\amalg_*$ as left adjoints constructed by left Kan extension. There is thus a natural transformation 
\begin{eqnarray}
\label{eqn:projection_iso}
\tmalg_* \circ (\amalg \times \amalg)^* \rightarrow \amalg^* \circ \amalg_*
\end{eqnarray}
 that corresponds by adjunction to
$
\amalg_*\circ \tmalg_* \circ (\amalg \times \amalg)^*
\rightarrow 
\amalg_*
$
given by the composite of the natural isomorphism $\amalg_*\circ \tmalg_* \circ (\amalg \times \amalg)^* \cong \amalg_* \circ (\amalg \times \amalg)_*\circ (\amalg \times \amalg)^*$ provided by the first part, followed by the morphism induced by the adjunction counit $(\amalg \times \amalg)_*\circ (\amalg \times \amalg)^* \rightarrow \id$.

It remains to show that this is a natural isomorphism. This can be seen by analysing the  natural transformation using the description of the left Kan extensions, as follows. Consider a $\kk (\tfb)$-module $F$ and $(X,Y) \in \ob \tfb$. Then $(\tmalg_* \circ (\amalg \times \amalg)^* F) (X,Y)$ is given by the colimit 
$$
\mathrm{colim} 
_{\substack{\tmalg((U,V),(W,Z)) \rightarrow (X,Y)}}
F ( U\amalg V, W \amalg Z) 
$$
taken over the slice category $(\tfb)^{\times 2} _{/ (\tfb)}$ defined with respect to $\tmalg$. Correspondingly,  $(\amalg^* \circ \amalg_*F) (X,Y)$ is given by the colimit
$$
\mathrm{colim}_{\substack{\amalg (A, B) \rightarrow (X \amalg Y)}} F (A, B),
$$
taken over the slice category $\tfb _{ / \fb}$ defined with respect to $\amalg$. The map between the colimits is induced by the functor between the respective slice categories induced by $(\amalg \times \amalg) : (\tfb)^{\times 2}  \rightarrow \tfb$ and $\amalg : \tfb \rightarrow \fb$, using the compatibility from Proposition \ref{prop:amalg_circ_tmalg}.

Now, a morphism $f$ in $\fb$ of the form $\amalg (A, B) = A \amalg B \rightarrow X \amalg Y$ determines a decomposition $A = A_X \amalg A_Y$ and $B = B_X \amalg B_Y$ as subsets, taking $A_X$ (respectively $B_X$) to be the preimage of $X$ under $f|_A$ (resp. $f|_B$). The morphism $f$ thus induces the morphism in $\tfb$:
$$
\tmalg ((A_X, A_Y), (B_X, B_Y)) \rightarrow (X,Y).
$$
Using this, one deduces that the functor between slice categories is cofinal, and hence that the morphism between the colimits is an isomorphism, as required.
\end{proof}

%%%%%%%%%%%%%%%%%%%%%%%%%%%%%%%%%%%%%%%%%%%%%%%%%%%%%%%%%%%%%%%%%%%%%%%%%%
\subsection{The Day convolution products}
\label{subsect:introduce_convolution}

Day convolution provides the respective symmetric monoidal structures 
$(\f (\fb), \odot, \kk_\zero) $ and 
$(\f (\tfb), \circledcirc , \kk_{(\zero,\zero)})$.
Here, the convolution products are given by the composites of the exterior tensor products followed by the left Kan extension respectively:
\begin{eqnarray*}
\odot 
&:&\f (\fb) \times \f (\fb) \stackrel{\boxtimes}{\rightarrow} 
\f (\tfb) 
\stackrel{\amalg_*}{\rightarrow }
\f (\fb) 
\\
\circledcirc 
&:&\f (\tfb) \times \f (\tfb) \stackrel{\boxtimes}{\rightarrow} 
\f ((\tfb)^{\times 2})
\stackrel{\tmalg_*}{\rightarrow}
\f (\tfb). 
\end{eqnarray*}

Explicitly, for $\kk \fb$-modules $F_1$ and $F_2$, and $X$ a finite set:
\begin{eqnarray*}
(F_1 \odot F_2) (X) := \bigoplus_{S_1 \amalg S_2 = X} F_1(S_1) \otimes F_2 (S_2),
\end{eqnarray*}
where the sum is over ordered decompositions of $X$ into two  subsets. The symmetry $\tau : F_1 \odot F_2 \stackrel{\cong}{\rightarrow} F_2 \odot F_1$ is induced by the isomorphism of vector spaces
$
F_1(S_1) \otimes F_2(S_2) \stackrel{\cong}{\rightarrow} F_2(S_2) \otimes F_1(S_1)
$ 
given by the symmetry in $\kk$-vector spaces, for each ordered decomposition $S_1\amalg  S_2=X$.

Likewise, for $\kk (\tfb)$-modules $G_1$ and $G_2$:
\begin{eqnarray*}
\label{eqn:convolution_bimodules}
(G_1 \circledcirc G_2) (X,Y) := \bigoplus_{\substack{S_1 \amalg S_2 = X\\ T_1\amalg T_2 =Y}} G_1(S_1,T_1) \otimes G_2 (S_2,T_2).
\end{eqnarray*}
The symmetry $\tau : G_1 \circledcirc G_2 \rightarrow G_2 \circledcirc G_1$ is defined similarly to the case $\odot$.

%%%%%%%%%%%%%%%%%%%%%%%%%%%%%%%%%%%%%%%%%%%%%%%%%%%%%%%%%%%%%%%%%%%%%%%%%%%%%%%%%%%%%%
\subsection{Compatibility of the Day convolution products}

The Day convolution products for $\f (\fb)$ and $\f(\tfb)$ are compatible via the following:

\begin{thm}
\label{thm:compatibility_day}
The functors 
\begin{eqnarray*}
\amalg_* &:& \f (\tfb) \rightarrow \f (\fb)\\
\amalg^* &:& \f (\fb) \rightarrow \f (\tfb) 
\end{eqnarray*}
are symmetric monoidal with respect to the respective Day convolution products.
\end{thm}

\begin{proof}
The fact that $\amalg_*$ is symmetric monoidal is a formal consequence of Proposition \ref{prop:amalg_circ_tmalg} using the universal property of the Day convolution product 
\cite{MR862873}. (This can also be checked directly.)

This implies (using that $\amalg^*$ is right adjoint to $\amalg_*$) that $\amalg^*$ is lax symmetric monoidal by \cite{MR360749}, for example. In particular, for $\kk \fb$-modules $F_1$ and $F_2$, by adjunction one obtains the natural transformation 
\begin{eqnarray}
\label{eqn:lax_monoidal}
(\amalg^* F_1) \circledcirc (\amalg^* F_2) \rightarrow \amalg^* (F_1 \odot F_2)  
\end{eqnarray}
which yields the lax symmetric monoidality. To conclude, it suffices to show that this is an isomorphism. 

This is based on the analysis of the following diagram of functors:
\[
\xymatrix{
\f (\fb ) \times \f(\fb) 
\ar[r]^\boxtimes 
\ar[d]_{\amalg^* \times \amalg^*}
&
\f (\tfb) 
\ar[r]^{\amalg_*}
\ar[d]_{(\amalg \times \amalg)^*}
&
\f (\fb) 
\ar[d]_{\amalg^*}
\\
\f (\tfb ) \times \f(\tfb) 
\ar[r]_(.6)\boxtimes 
&
\f ((\tfb)^{\times 2}) 
\ar[r]_{\tmalg_*}
&
\f (\tfb). 
}
\]
Here the left hand square commutes up to natural isomorphism by compatibility of the external tensor product with restriction.  The right hand square commutes up to natural isomorphism by Lemma \ref{lem:projection_formula}; more precisely, the natural transformation  (\ref{eqn:projection_iso}) is a natural isomorphism. 

Since the composites of the rows are respectively $\odot$ and $\circledcirc$, this provides a natural isomorphism $\amalg^* \circ \odot \cong \circledcirc \circ (\amalg\* \times \amalg^*)$ of functors $\f (\fb) \times \f(\fb) \rightarrow \f (\tfb)$.  To conclude, it remains to check that this argument establishes that the natural morphism (\ref{eqn:lax_monoidal}) is an isomorphism. 
 This boils down to checking that the construction of the natural transformation (\ref{eqn:lax_monoidal}) is compatible with the construction of (\ref{eqn:projection_iso}).
\end{proof}

%%%%%%%%%%%%%%%%%%%%%%%%%%%%%%%%%%%%%%%%%%%%%%%%%%%%%%%%%%%%%%%%%%%%%%%%%%%%%%%%%%%%%%%
\subsection{Analysing Day convolution using the Schur correspondence}

In this section we take $\kk$ to be a field of characteristic zero and write $\fdvs$ for the category of finite-dimensional $\kk$-vector spaces.  Basic references for the background material of this section are 
\cite[Appendix I.A]{MR3443860} for polynomial functors
and 
\cite{MR927763} for analytic functors in characteristic zero. We give a conceptual viewpoint on Theorem \ref{thm:compatibility_day} based upon the Schur correspondence.

The Schur functor construction defines a functor 
$
\f (\fb) \rightarrow \f (\fdvs).
$
 This associates to a $\kk \fb$-module $M$ the analytic functor
$$
V \mapsto M(V):= \bigoplus_{n \in \nat} V^{\otimes n} \otimes_{\kk \sym_n} M(\n).
$$
This construction is symmetric monoidal with respect to the Day convolution product on $\f (\fb)$ and the pointwise tensor product on $\f (\fdvs)$. Explicitly, for $\kk \fb$-modules $M$ and $N$, and $V$ a finite dimensional $\kk$-vector space, there is a natural isomorphism
$
(M \odot N) (V) \cong M(V) \otimes N(V)
$
satisfying the appropriate coherence and symmetry axioms.

The Schur functor construction extends to $\kk (\tfb)$-modules using bi-functors: there is a symmetric monoidal functor $
\f (\tfb) \rightarrow \f ( \fdvs \times \fdvs)
$
 given for a $\kk(\tfb)$-module $B$ by 
$$
(V, W) \mapsto B(V,W) := \bigoplus_{m,n\in \nat} (V^{\otimes m} \otimes W^{\otimes n}) \otimes_{\kk (\sym_m \times \sym_n)} B (\m, \n).
$$
The symmetric monoidality corresponds to the natural isomorphism for bimodules $B_1$ and $B_2$:
\begin{eqnarray}
\label{eqn:sym_mon_circledcirc}
(B_1 \circledcirc B_2) (V,W) \cong B_1 (V, W) \otimes B_2 (V,W).
\end{eqnarray}

\begin{rem}
\label{rem:schur_equivalence}
The  Schur correspondence gives an equivalence between the category of $\kk \fb$-modules and the full subcategory of {\em analytic functors} in $\f (\fdvs)$. The analogous statement folds for $\kk (\tfb)$-modules. These are equivalences of symmetric monoidal categories.
\end{rem}

We now consider the counterparts of the pair of functors $\amalg^* : \f (\fb) \rightarrow \f (\tfb)$ and $
\amalg_* : \f (\tfb) \rightarrow \f (\fb)$. We use the following well-known result:

\begin{lem}
\label{lem:oplus_diag}
The functors $\oplus : \fdvs \times \fdvs \rightarrow \fdvs$ and $\diag : \fdvs \rightarrow \fdvs \times \fdvs$ define a pair of adjoint functors ($\oplus$ is both left and right adjoint to $\diag$).
\end{lem}

This implies:

\begin{prop}
\label{prop:oplus*_diag*_sym_mon}
The restriction functors 
$$
\oplus^* : \f( \fdvs) \rightleftarrows \f(\fdvs \times \fdvs) : \diag^*
$$
form a pair of adjoint functors. Moreover, both $\bigoplus^*$ and $\diag^*$ are symmetric monoidal with respect to the respective pointwise tensor products. 
\end{prop}

\begin{proof}
The first statement follows immediately from Lemma \ref{lem:oplus_diag}, since precomposition with an adjunction yields an adjunction. The second statement is simply the fact that precomposition with any functor is symmetric monoidal. 
\end{proof}

\begin{prop}
\label{prop:compare_adjoints}
The following diagrams commute up to natural isomorphism
$$
\xymatrix{
\f (\fb) 
\ar[r]
\ar[d]_{\amalg^*} 
&
\f (\fdvs)
 \ar[d]^{\oplus^*}
&&
 \f (\tfb) 
\ar[r]
\ar[d]_{\amalg_*} 
&
\f (\fdvs \times \fdvs)
 \ar[d]^{\diag^*}
 \\
 \f (\tfb) 
 \ar[r]
 &
 \f (\fdvs \times \fdvs)
&& 
 \f (\fb) 
 \ar[r]
 &
 \f (\fdvs),
}
$$
where the horizontal functors are the respective Schur constructions. 
\end{prop}

\begin{proof}
The proof of this result is a variant of the proof that the Schur constructions are symmetric monoidal. For completeness, we sketch the proof, giving the argument for objects.

Consider the first diagram.  One can reduce to  the case $M$  a $\kk \fb$-module supported on $\n$. Passing around the top of the diagram yields the bifunctor 
$$
(V, W) \mapsto (V \oplus W)^{\otimes n} \otimes _{\kk \sym_n} M (\n).
$$
Now, $(V \oplus W)^{\otimes n}$ is naturally $\sym_n$-equivariantly isomorphic to 
$
\bigoplus_{k+l =n} (V^{\otimes k} \otimes W^{\otimes l})\uparrow _{\sym_k \times \sym_l}^{\sym_n}. 
$
Then, using Frobenius reciprocity, the bifunctor can be rewritten as 
\begin{eqnarray}
\label{eqn:schur_bifunctor}
(V, W) \mapsto  \bigoplus_{k+l =n} (V^{\otimes k} \otimes W^{\otimes l}) \otimes _{\kk (\sym_k \times \sym_l)} M (\n)\downarrow _{\sym_k \times \sym_l}^{\sym_n}.
\end{eqnarray}

Now consider passing around the bottom of the diagram. The $\kk (\tfb)$-module $\amalg^* M$ is given by 
$$
(\mathbf{s}, \mathbf{t}) \mapsto 
\left\{
\begin{array}{ll}
M (\n)\downarrow _{\sym_s \times \sym_t}^{\sym_n}  & s+t =n
\\
0 & \mbox{otherwise.}
\end{array}
\right.
$$
Clearly, the associated Schur bifunctor is naturally isomorphic to (\ref{eqn:schur_bifunctor}), as required.

The proof for the second diagram is similar. Here, one can reduce to the case $B$ a $\kk (\tfb)$-module supported on $(\mathbf{s}, \mathbf{t})$, so that $\amalg_* B$ is supported on $\n$, where $n=s+t$, with value $B(\mathbf{s}, \mathbf{t})\uparrow_{\sym_s\times \sym_t}^{\sym_n}$. Now, the Schur bifunctor associated to $B$ is $(V, W) \mapsto  (V^{\otimes s} \otimes W^{\otimes t}) \otimes_{\kk (\sym_s \times \sym_t)} B (\mathbf{s}, \mathbf{t})$. Applying $\diag^*$ thus yields the functor 
$V \mapsto V^{\otimes n}\downarrow^{\sym_n}_{\sym_s \times \sym_t} \otimes_{\kk (\sym_s \times \sym_t)} B (\mathbf{s}, \mathbf{t})$. By Frobenius reciprocity, this is naturally isomorphic to the Schur functor associated to $B(\mathbf{s}, \mathbf{t})\uparrow_{\sym_s\times \sym_t}^{\sym_n}$, as required.
\end{proof}

\begin{rem}
Using the Schur correspondence recalled in Remark \ref{rem:schur_equivalence} together with Proposition \ref{prop:compare_adjoints}, Proposition \ref{prop:oplus*_diag*_sym_mon} implies  Theorem \ref{thm:compatibility_day} when working over a field of characteristic zero.
\end{rem}

%\input{symm_exterior}

%%%%%%%%%%%%%%%%%%%%%%%%%%%%%%%%%%%%%%%%%%%%%%%%%%%%%%%%%%%%%%%%%%%%%%%%%%%%%%%%
\section{Symmetric and exterior powers}
\label{sect:symm_exterior}

This short section recalls the definition of the symmetric and exterior powers with respect to the respective Day convolution products on $\kk\fb$-modules and $\kk(\tfb)$-modules.

%%%%%%%%%%%%%%%%%%%%%%%%%%%%%%%%%%%%%%%%%%%%%%%%%%%%%%%
\subsection{Introducing the powers}

For $d \in \nat$ and $F$ a $\kk \fb$-modules (respectively $G$ a $\kk (\tfb)$-module), one has the respective $d$-iterated convolution products 
$F^{\odot d}$
 and 
$G^{\circledcirc d}$; the symmetric group $\sym_d$ acts on these by place permutations. One can thus introduce:

\begin{defn}
\label{defn:stfb_ltfb}
For $d \in \nat$,
\begin{enumerate}
\item 
for $F$ a $\kk \fb$-module, define: 
\begin{eqnarray*}
\sfb^d F &:= &F^{\odot d}/ \sym_d; 
\\
\lfb^d F &:=& (F^{\odot d}\otimes \sgn_d)/ \sym_d,
\end{eqnarray*}
where $\sym_d$ acts diagonally on $F^{\odot d}\otimes \sgn_d$;
\item 
for $G$ a $\kk (\tfb)$-module, define: 
\begin{eqnarray*}
\stfb^d G &:= &G^{\circledcirc d}/ \sym_d; 
\\
\ltfb^d G &:=& (G^{\circledcirc d}\otimes \sgn_d)/ \sym_d,
\end{eqnarray*}
where $\sym_d$ acts diagonally on $G^{\circledcirc d}\otimes \sgn_d$.
\end{enumerate}
\end{defn}

\begin{exam}
\label{exam:d=2_fb}
To illustrate the case of $\kk \fb$-modules, take $d=2$ and consider a $\kk \fb$-module $F$. Evaluating on a set $X$, by definition 
$$
F^{\odot 2} (X) = \bigoplus_{X=S_1 \amalg S_2} F(S_1) \otimes F(S_2),
$$
where the sum is over {\em ordered} decompositions of $X$ into two subsets. The action of $\sym_2$ relates the terms indexed by the ordered decompositions $S_1\amalg S_2$ and $S_2 \amalg S_1$. If $X \neq \emptyset$, then these ordered decompositions are different. 

The difference between $\sfb^2$ and $\lfb^2$ is already apparent evaluating on $\emptyset$, since 
\begin{eqnarray*}
\sfb^2 F (\emptyset) &=& S^ 2 (F(\emptyset)) \\
\lfb^2 F (\emptyset ) &=& \Lambda^ 2 (F(\emptyset)),
\end{eqnarray*}
using the usual symmetric power and exterior power functors on vector spaces.

Now, if $X \neq \emptyset$, the underlying vector spaces of $\sfb^2 F (X)$ and $\lfb^2 (X)$ are both (non-canonically) isomorphic to 
$$
\bigoplus_{(S_1, S_2)} F(S_1 ) \otimes F (S_2), 
$$
where the sum is now over {\em unordered} decompositions of $X$ into two subsets. 

However, they are not necessarily isomorphic as  $\aut(X)$-modules. For instance, take $F= \kk_{\one}$, the $\kk \fb$-module supported on $\one$ with value $\kk$, and consider $X= \mathbf{2}$. Then, there are isomorphisms of $\sym_2$-modules:
\begin{eqnarray*}
\sfb^2 \kk_\one  (\mathbf{2}) &\cong & \triv_2 
\\
\lfb^2 \kk_\one  (\mathbf{2}) &\cong & \sgn_2 ,
\end{eqnarray*}
respectively the trivial and the sign representations.
\end{exam}

Generalizing the identification in the above example, we have:

\begin{prop}
\label{prop:sfb_lfb_stfb_ltfb_underlying}
Let  $d$ be a positive integer. 
\begin{enumerate}
\item 
For $F$ a $\kk \fb$-module such that $F (\emptyset) =0$ and $X$ a finite set, the underlying vector spaces of $\sfb^d F (X)$ and $\lfb^d F (X)$ are both non-canonically isomorphic to 
$$
\bigoplus_{(S_1, \ldots ,S_d) } 
\bigotimes _{i=1}^d F (S_i)
$$
where the sum is over a set of representatives of unordered decompositions of $X$ into $d$ non-empty subsets.  
\item 
For $G$ a $\kk(\tfb)$-module such that $G (\emptyset, \emptyset) =0$ and $(X,Y)$ a pair of finite sets, the underlying vector spaces of $\stfb^d G (X,Y)$ and $\lfb^d G (X,Y)$ are both non-canonically isomorphic to 
$$
\bigoplus_{\substack{(S_1, \ldots ,S_d) \\ (T_1, \ldots , T_d) } }
\bigotimes _{i=1}^d G (S_i,T_i)
$$
where the sum is over a set of representatives of the $\sym_d$-orbits of $\{ (S_i, T_i) \mid \amalg_i S_i = X , \amalg_i T_i = Y, i \in \mathbf{d} \}$, for the obvious action of $\sym_d$. 
\end{enumerate}
\end{prop}

\begin{rem}
It is clear from Example \ref{exam:d=2_fb} that, without the hypothesis $F (\emptyset) =0$, the underlying vector spaces of $\sfb^d F (X) $ and $\lfb^d F (X)$ are not in general isomorphic. However,  one can  refine the statement of Proposition \ref{prop:sfb_lfb_stfb_ltfb_underlying} to treat the general case. Likewise for $\kk (\tfb)$-modules and $\stfb^*$ and $\ltfb^*$.
\end{rem}

\begin{exam}
\label{exam:stfb_ltfb}
Take $G$ to be $\kk_{(\mathbf{1}, \mathbf{1})}$, the $\kk (\tfb)$-module supported on $(\mathbf{1}, \mathbf{1})$ with value $\kk$. Then $\stfb^d \kk_{(\mathbf{1}, \mathbf{1})}$ and $\ltfb^d \kk_{(\mathbf{1}, \mathbf{1})}$ are both supported on $(\mathbf{d}, \mathbf{d})$. There are isomorphisms of $\kk \sym_d$-bimodules
\begin{eqnarray*}
\stfb^d \kk_{(\mathbf{1}, \mathbf{1})} (\mathbf{d}, \mathbf{d})
& \cong & \kk \sym_d 
\\
\stfb^d \kk_{(\mathbf{1}, \mathbf{1})} (\mathbf{d}, \mathbf{d})
& \cong & \sgn_d \otimes \kk \sym_d ,
\end{eqnarray*} 
using the regular bimodule structure on $\kk \sym_d$ and, in the second case, twisting the left action by $\sgn_d$. For $d>1$ these are not isomorphic.
\end{exam}

%%%%%%%%%%%%%%%%%%%%%%%%%%%%%%%%%%%%%%%%%%%%%%%%%%%%%%%%%%%%%%%%%
\subsection{Analysis using the Schur constructions}
\label{subsect:such_S_Lambda}

Fix a natural number $d$. Then we have the usual symmetric and exterior power functors $S^d$ and $\Lambda^d$ on $\kk$-vector spaces. These are respectively the Schur functors associated to $\triv_d$ and to $\sgn_d$ (considering these as $\kk \fb$-modules supported on $\mathbf{d}$).

The fact that the Schur constructions for $\kk \fb$-modules and for $\kk (\tfb)$-modules are symmetric monoidal implies the following:

\begin{prop}
Fix $d \in \nat$.
\begin{enumerate}
\item 
For $F$ a $\kk \fb$-module, there are natural isomorphisms
\begin{eqnarray*}
(\sfb^d F) (V) 
&\cong & S^d (F(V)) \\
(\lfb^d F) (V) 
&\cong & \Lambda^d (F(V)).
\end{eqnarray*}
\item 
For $G$ a $\kk (\tfb)$-module, there are natural isomorphisms
\begin{eqnarray*}
(\stfb^d G) (V,W) 
&\cong & S^d (F(V)) \\
(\ltfb^d G) (V,W) 
&\cong & \Lambda^d (G(V,W)).
\end{eqnarray*}
\end{enumerate}
\end{prop}

%%%%%%%%%%%%%%%%%%%%%%%%%%%%%%%%%%%%%%%%%%%%%%%%%%%%%%%%%%%
%\bibliographystyle{alphaurl}
%\bibliography{uwb.bib}

\end{document}